\documentclass[letterpaper,11pt,reqno]{amsart}
\usepackage[margin=1in]{geometry}
\usepackage{setspace}
\usepackage{eucal}
\usepackage{tikz}
\usepackage{stackrel}
\usepackage{tikz-cd}
\usepackage{xypic}
\usepackage{amsfonts,amssymb}
\usepackage{epsfig}
\usepackage{enumerate}
\usepackage{enumitem}
\usepackage{mathrsfs}
\usepackage{amsthm}
\usepackage{amssymb}
\usepackage[utf8]{inputenc}
\usepackage{mathtools}
\usepackage{abraces}
\usepackage{ulem}
\usepackage{kotex}
\usepackage{amsmath}
\usepackage[colorlinks]{hyperref}

\hypersetup{
  linkcolor  = black,
  citecolor  = black,
  urlcolor   = black,
  colorlinks = true,
}

\usetikzlibrary{decorations.pathreplacing,calligraphy}
\newtheorem{theorem}{Theorem}[section]
\newtheorem{lemma}[theorem]{Lemma}
\newtheorem{question}[theorem]{Question}
\newtheorem*{question*}{Question}
\newtheorem{prop}[theorem]{Proposition}
\newtheorem{corollary}[theorem]{Corollary}

\theoremstyle{definition}
\newtheorem{defn}[theorem]{Definition}
\newtheorem{remark}[theorem]{Remark}
\newtheorem{example}[theorem]{Example}

\newcommand{\Q}{\mathbb Q}
\newcommand{\C}{\mathbb C}
\newcommand{\R}{\mathbb R}
\newcommand{\K}{\mathbb K}
\renewcommand{\P}{\mathbb P}

\newcommand{\T}{\mathbb T}
\newcommand{\D}{\mathbb D}

\newcommand{\Z}{\mathbb Z}
\newcommand{\N}{\mathbb N}

\newcommand{\id}{\textnormal{id}}

\DeclareFontFamily{U}{matha}{\hyphenchar\font45}
\DeclareFontShape{U}{matha}{m}{n}{
      <5> <6> <7> <8> <9> <10> gen * matha
      <10.95> matha10 <12> <14.4> <17.28> <20.74> <24.88> matha12
      }{}

\newcommand{\blor}{\mathbin{
  \raisebox{.1ex}{%
    \rotatebox[origin=c]{90}{\usefont{U}{matha}{m}{n}\symbol{\string"CE}}}}}
    
\DeclareFontFamily{U}{matha}{\hyphenchar\font45}
\DeclareFontShape{U}{matha}{m}{n}{
      <5> <6> <7> <8> <9> <10> gen * matha
      <10.95> matha10 <12> <14.4> <17.28> <20.74> <24.88> matha12
      }{}

\def\id{\textnormal{id}}

\def\id{\textnormal{id}}

\def\lra#1{\overset{#1}{\longrightarrow}}

\newlist{HF}{enumerate}{1}
\setlist[HF]{label=(HF\arabic*)}

\newlist{PB}{enumerate}{1}
\setlist[PB]{label=(PB\arabic*)}

\newlist{CZ}{enumerate}{1}
\setlist[CZ]{label=(CZ\arabic*)}

\newlist{S1HF}{enumerate}{1}
\setlist[S1HF]{label=(S1HF\arabic*)}

\numberwithin{equation}{section}

\title{Symplectic Criteria on Stratified Uniruledness of Affine Varieties and Applications}

\author{Dahye Cho}
\email{dahye@math.stonybrook.edu}

\begin{document}

\begin{abstract}
We develop criteria for affine varieties to admit uniruled subvarieties of certain dimensions. The measurements are from long exact sequences of versions of symplectic cohomology, which is a Hamiltonian Floer theory for some open symplectic manifolds including affine varieties. Symplectic cohomology is hard to compute, in general. However, certain vanishing and invariance properties of symplectic cohomology can be used to prove that our criteria for finding uniruled subvarieties hold in some cases. We provide applications of the criteria in birational geometry of log pairs in the direction of the Minimal Model Program.
\end{abstract}

\maketitle
\bibliographystyle{halpha}

\tableofcontents

\section{Introduction}
A projective variety $X$ over complex numbers $\C$ is uniruled if for a generic point $x \in X$, there is a rational curve $\P^{1} \to X$ passing through $x$, where $\P^{n}$ denote $\C P^{n}$. An affine variety $M$ over $\C$ is called $k$-uniruled if for a generic point $x \in M$, there is a polynomial map $(\P^{1} \setminus\Gamma) \to M$ passing through $x$, where $\Gamma$ is a set of at most $k$ distinct points. (An affine variety $M$ is also called $\C$-uniruled if it is $1$-uniruled, i.e. for a generic point $x \in M$ there is a polynomial map $\C \to M$ passing through $x$.) A Liouville domain $\underline{M}$ is $J$-holomorphic $(k,\mathcal{E})$-uniruled if, for every convex almost complex structure $J$ on $\underline{M}$ and for any point $x$ in the interior $\text{int}\underline{M}$ of $\underline{M}$ where $J$ is integrable near $x$, there is a $J$-holomorphic map $\Sigma \to \text{int}\underline{M}$ passing through $x$, where $\Sigma$ is a Riemann sphere with the number boundary is at most $k-1$, and the energy is at most $\mathcal{E}$. A surprising fact is that both uniruledness of projective varieties and $k$-uniruledness of affine varieties are symplectic invariant by Koll\'ar, Ruan and McLean using Gromov-Witten theory \cite{kollar96}, \cite{ruan99}, \cite{mclean14}. Moreover, uniruledness is a symplectic birational invariant by Hu-Li-Ruan \cite{huliruan08} (i.e. Uniruledness is invariant under symplectic blow ups and blow downs). Biolley used Floer cohomology to detect hyperbolicities \cite{biolley04}. In \cite{mcduffuniruled}, McDuff proved that every closed symplectic manifold with Hamiltonian $S^{1}$-action is uniruled by detecting nonzero Gromov-Witten invariant using Seidel representation \cite{seidelthesis} of the fundamental group of Hamiltonian diffeomorphism into the small quantum cohomology. McLean proved that $J$-holomophic $(1,\mathcal{E})$-uniruledness implies $\C$-uniruledness on affine varieties \cite{mclean14}. Based on \cite{biolley04}, Albers-Frauenfelder-Oancea showed that vanishing symplectic cohomology implies $J$-holomophic disk-uniruledness by SFT-neck stretching \cite{albers.frauenfelder.oancea17}. Zhou proved that, for an exact domain, the existence of a symplectic dilation implies $(1,\mathcal{E})$-uniruled \cite{zhou19} based on \cite{mclean14}, \cite{diogilisi18s}. Li also considered it using the cyclic dilation in \cite{yinli19}.

In this paper, we develop a criterion for affine varieties or log pairs to be stratified uniruled, which means a criterion to find uniruled subvarieties of certain dimensions, inspired by \cite{biolley04}, \cite{mclean14}. To define a measurement, we use symplectic cohomology, that is a Hamiltonian Floer theory for some open symplectic manifolds including affine varieties. The key observation is that the connecting maps of symplectic cohomology are Floer cylinders which we deform to non-constant pseudo-holomorphic curves. Moreover, we can deform the pseudo-holomorphic curves to rational curves passing through cycles representing cohomology classes. The deformation/degeneration techniques are heavily based on \cite{hermann}, \cite{biolley04}, \cite{bourgeois.oancea09}, \cite{mclean14}. We show how to apply the criteria effectively by applying certain vanishing and invariance properties of symplectic cohomology or by computing spectral sequences that converges to symplectic cohomology, \cite{mclean18}, \cite{mcleancomputation}, \cite{ganatrapomerleano}, \cite{mclean-ritter18}. An application of the criterion comes when two affine varieties have isomorphic symplectic cohomologies but different cohomologies. For example, the Kaliman modification of affine varieties of complex dimension higher than $2$, does not change symplectic cohomology. Therefore, we can detect uniruled subvarieties of affine varieties after Kaliman modification. Moreover, the same idea works on the fact that for a Stein manifold, symplectic cohomology is invariant under attaching flexible Weinstein $k$-handles. We explore some applications as corollaries. If the homological mirror symmetry would convert symplectic cohomology more computable, we expect to use this work to understand log birational geometry and suggest some questions considering the minimal models.\\

The following is the main theorem of this paper. Let $M$ be a smooth affine variety of complex dimension $n$ together with a trivialization of some power of its canonical bundle. Let $X$ be a smooth projective variety compactifying $M$ with an ample divisor $D$. There is a long exact sequence:  
$$\begin{tikzcd} [cramped, sep=small] \label{exact triangle}
\cdots \to SH^{*-1}(M)\arrow[r] & SH^{*-1}_{+}(M)\arrow[r,"\delta"] & H^{*}(M)\arrow[r,"\iota"] & SH^{*}(M)\to \cdots.
\end{tikzcd}
$$
where, $\delta: SH^{*}_{+}(M)\to H^{*}(M)$ is the connecting map.

\begin{theorem}[\ref{main theorem00}] \label{main theorem0}
If a cohomology class $[\mho]\in H^{m}(M)$ be in the image of the map $\delta$ for $m=2k$ or $2k+1$ for some $k\in \N$. Then there exists a subvariety $\Xi_{\mho}\subset M$ of dimension at least $n-k$ satisfying the following properties:
\begin{enumerate}
    \item $\Xi_{\mho}$ is $\C$-uniruled. In other words, for each $p\in \Xi_{\mho}$, there exists a non-constant algebraic map $\overline{v}_{p}:\C\to \Xi_{\mho}$ whose image contains $p$.
    \item For any exhausting finite type Morse function $f$ on $X\setminus D$, $\Xi_{\mho}$ set-theoretically intersect with a unstable submanifold $\mho_{f}$ of $f$ representing $[\mho]\in H^{m}(M)\cong HM^{m}(M,f)$.
\end{enumerate}
\end{theorem}

As an application, we can define a measurement of the co-dimension of maximal uniruled subvariety and have the main criterion as follow. 

\begin{defn}
    $\ell(M):=$ min$\{$deg$([\alpha]):[\alpha] \in H^*(M)$ with $[\alpha]=\delta(\gamma)$ for some $0\neq \gamma\in SH^{*-1}_{+}(M)\}$.
\end{defn}

\begin{remark}
The same definition holds for all Liouville domains.
\end{remark}

\begin{corollary} [\ref{4.1}] \label{main theorem1} 
If $\ell(M)=2k$ or $2k+1$ $(0\leq k <n)$, then $M$ admits a $(n-k)$-dimensional family of affine lines. Here, an affine line means the image of $\C$ in $M$ under a nonconstant polynomial map. Moreover, $M$ admits a uniruled subvariety of dimension $n-k$.
\end{corollary}

\begin{corollary} [\ref{cor1}]
If $\ell(M)=0$ or $1$, then $M$ is $\C$-uniruled.
\end{corollary}

\begin{corollary}[\ref{SH=0}]
If $SH^*(M)=0$, then $M$ is $\C$-uniruled. (See Theorem 5.4 in \cite{zhou19})
\end{corollary}

\begin{theorem} [\ref{1h-attach}]
Let $\underline{M}$ be the associated Liouville domain obtained by intersection of $M$ and a large $2n$-ball. Assume that $\underline{M}$ is of complex dimension equal to or bigger than $3$. Suppose that $\underline{M}$ is a connected Liouville domain with a Weinstein $1$-handle attached, then $M$ is $\C$-uniruled.
\end{theorem}

More generally, in the category of Weinstein manifolds, we consider attaching flexible Weinstein handles in \cite{weinsteinhandle}, \cite{murphy12}, \cite{cieliebakeliashberg12}, \cite{cieliebakeliashberg14}, which does not change symplectic cohomology [Theorem 5.6, \cite{BEE12}]. An example of a flexible Weinstein handle is a subcritical handle, \cite{cieliebakhandle}.

\begin{theorem} [\ref{flexiblehandle}]
Let $W$ be a Weinstein manifold of dim$_{\R}W=2n$ with $l(W)=\infty$. Suppose that we have a Weinstein manifold $W_{\Supset k}$, obtained by attaching flexible $k$-handles to $W$ so that $\text{rank }H_{k}(W_{\Supset k})>\text{rank }H_{k}(W)$. Then $l(W_{\Supset k})=2n-k$. Hence, if $W_{\Supset k}$ is symplectomorphic to an affine variety $M$, then $M$ admits a $\C$-uniruled subvariety of complex dimension $n-k$.
\end{theorem}

\begin{remark}
To apply the theorem above in practice, we need to understand the condition when a Stein manifold $W_{\Supset k}$ is symplectomorphic to an affine variety.
\end{remark}

A Liouville domain corresponding to an affine variety has a structure of symplectic convex Lefschetz fibration [Lemma 8.6, \cite{mclean09}]. The Kaliman modification $M^{\vee}$ of an affine variety $M$ become an affine variety having $M$ as a sub-Lefschetz fibration so that all the singular points are contained in $M$. Therefore, using the fact that symplectic cohomology does not change under Kaliman modification, we can construct affine varieties with non-trivial $\ell(M)$.

\begin{theorem} [\ref{construction1}]
For each $n\geq 3$, there is an affine variety $M^{2n}_{2}$ of dim$_{\C}M=n$ with $\ell(M^{2n}_{2})=2$.
\end{theorem}

An affine variety $M$ is called cylindrical if it contains  a dense principal Zariski open subset $U=M\setminus (f=0)$, for some $f\in \mathcal{O}(X)$, isomorphic to $\C\times M'$ for an affine variety $M'$. We call such $U$ a cylinder [Definition 3.1.4 \cite{kishimoto.prokhorov.zaidenberg11}]. Applying the idea of \cite{mclean18}, \cite{varolgunes19} to the case of compact subsets of the completion of a Liouville domain,

\begin{theorem} [\ref{cylindrical}]
Let $M$ be an cylindrical affine variety having a dense affine open subset $U=M\setminus (f=0)$ for some $f\in \mathcal{O}(X)$, $SH^{*}(\underline{M}\subset M)=0$. Therefore $\widehat{\ell}(M)=0$ and $M$ is $\C$-uniruled. (For the definitons of $SH^{*}(\underline{M}\subset M)$ and $\widehat{\ell}(M)$, we refer the subsection containing lemma \ref{relativelemma1}.)
\end{theorem}

Applying the similar idea of proof of theorem \ref{cylindrical} to special case of half-point attachment of an affine threefold, we have the following theorem.

\begin{theorem} [\ref{affine3fold1}]
Consider special case of Kishimoto's $\#$- minimal model program on affine threefolds that restrict half-point attachment at a point on a hypersurface. Then
\begin{enumerate}
    \item $\widehat{\ell}(M^{i})$'s are not decreasing. 
    \item Let $M$ be an affine threefold and $M^{\#}$ be a terminal object of $M$ in Kishimoto's $\#$-minimal model program. Suppose that $\pi_{1}(M^{\#})=0$, $M\ncong M^{\#}$ and at least one of steps of half-point attachment is done on a hypersurface divisor. Then $\widehat{\ell}(M)\geq 2$.
\end{enumerate}
\end{theorem}

Given two Liouville domains $\underline{M}$ and $\underline{N}$, define the end-connected sum $\underline{M}\#_{e}\underline{N}$ by attaching a $1$-handle to points $p\in \partial \underline{M}$, $q\in \partial \underline{N}$ in the boundary, respectivley. Then $\underline{M}\#_{e}\underline{N}$ become a Liouville domain [Theorem 2.10, \cite{mclean08}]. Moreover, $SH^{*}(M\#_{e}N)\cong SH^{*}(M)\times SH^{*}(N)$ as rings for complex dimension greater than 1 [Theorem, 2.17, \cite{mclean08}]. From the following diagram,
\begin{center}
\begin{tikzcd} [cramped, sep=small]
0 \arrow[r] &SH^{*}_{+}(M) \arrow[r]\arrow[d] & SH^{*}_{+}(M\#_{e} N) \arrow[r] \arrow[d]     & SH^{*}_{+}(N) \arrow[r]\arrow[d] &0\\
0 \arrow[r] &H^{*+1}(M,\partial M) \arrow[r]  & H^{*+1}(M\#_{e} N,\partial (M\#_{e} N)) \arrow[r] & H^{*+1}(N,\partial N) \arrow[r]  &0,
\end{tikzcd}
\end{center}

\begin{prop}
For the end-connected sum $\underline{M}\#_{e}\underline{N}$ of two Liouville domains $\underline{M}$ and $\underline{N}$, $\ell(M\#_{e}N)=\text{min}\{\ell(M),\ell(N)\}$.
\end{prop}

By the K$\ddot{\text{u}}$nneth formula,

\begin{prop}
$\ell(M\times N)=\text{min}\{\ell(M),\ell(N)\}$
\end{prop}

\begin{example}[\ref{mappingcylindereg}]
Let $\phi:M\to M$ be a compact-supported symplectomorphism on a Liouville domain $(M,\omega)$ whose support is contained in the interior of $M$. Assume that $\ell(M)=\infty$. Let $M_{\Supset k}$ be a $k$-handle attaching on $M$. $\Phi$ be an extension of $\phi$ to $M_{\Supset k}$ by identity on attached locus. Denote $\C^{*}\rtimes_{\phi} M:=(\R \times \R \times M)/\Z\cdot (s,\theta, x)\sim (s+1,\theta, \phi(x))$ the symplectic mapping cylinder of $\phi:M\to M$. Then for any $k<dim_{\C}M$, $SH^{*}(\C^{*}\rtimes_{\phi} M)\cong SH^{*}(\C^{*}\rtimes_{\Phi} (M_{\Supset k}))\cong SH^{*}((\C^{*}\rtimes_{\phi} M)_{\Supset k+1})$ but $\ell(\C^{*}\rtimes_{\phi} (M_{\Supset k}))=\text{min}\{k+2, \ell(\C^{*}\rtimes_{\phi} M)\}$, $\ell((\C^{*}\rtimes_{\phi} M)_{\Supset k+1})=\text{min}\{k+1, \ell(\C^{*}\rtimes_{\phi} M)\}$.
\end{example}

We can extend the usual symplectic cohomology $SH^*(M)$ to the twisted symplectic cohomology $SH^*(M;\underline{\Lambda}_{\beta})$: symplectic complexes have the same generators with Novikov field $\Lambda$ as a coefficient ring. But, the differential map is twisted by a closed 2-form $\xi$ on $M$, where $\beta$ is a singular cocycle representing a class in $H^{1}(\mathcal{L}M;\R)$ for $\xi\in H^{2}(M;\R)$. It means that the differential map remember the "area" of Floer cylinder with respect to $\xi$ in \cite{ritter09}. We also consider symplectic cohomology bulk-deformed by homology classes in \cite{usher}, \cite{siegel}. There exists a long exact sequence [Lemma 8.1, \cite{ritter13}],
$\begin{tikzcd} 
\cdots\to SH^{*-1}(M; \underline{\Lambda}_{\beta})\to SH^{*-1}_{+}(M; \underline{\Lambda}_{\beta})\lra{\delta} H^{*}(M; \underline{\Lambda}_{\beta}) \lra{\iota} SH^{*}(M;\underline{\Lambda}_{\beta})\to\cdots
\end{tikzcd}$.

The cotangent bundle $(T^{*}\P^{1},\omega_{std})$ over $\P^{1}$ is $\C$-uniruled but $\ell(T^{*}\P^{1})=2$. On the other hand, twisted symplectic cohomology $SH^{*}(T^{*}\P^{1},\Lambda_{\omega_{I}})$ vanishes, \cite{ritter10}. So, by considering the minimum of $\ell(M,\omega)$ over all the twisted cohomolgy, we get stronger measures as follows.

\begin{defn}
\begin{enumerate}
    \item $\ell (M;\underline{\Lambda}_{\beta}):=\text{min}\{\text{deg}([\alpha]):\text{ for }[\alpha] \in H^{*}(M;\underline{\Lambda}_{\beta})$ with $[\alpha]=\delta(\gamma)$ for some $0\neq\gamma\in SH^{*-1}_{+}(M;\underline{\Lambda}_{\beta})\}$.
    \item $\ell (M;\Lambda_{2}):=$min$\{\ell (M;\underline{\Lambda}_{\beta})$ : for all $\beta\in$ image$(H^{1}(\mathcal{L}M)\leftarrow H^{2}(\mathcal{L}M\times S^{1})\leftarrow H^{2}(M))\}$
    \item $\ell (M;\Lambda_{\text{all}}):=$min$\{\ell(M;\underline{\Lambda}_{\beta})$ : for all  bulk deformation by $\beta\in H^{*}(M)\}$.
\end{enumerate}
\end{defn}

\begin{corollary}
If $\ell(M;\underline{\Lambda}_{\beta})$or $\ell(M;\Lambda_{2})=2k$ or $2k+1$ $(0\leq k <n)$, then $M$ admits a uniruled subvariety of dimension $n-k$.
\end{corollary}

\begin{corollary}
If $\ell(M,\Lambda_{2})=0$, then $M$ is $\C$-uniruled.
\end{corollary}

\begin{example} \cite{ritter10}
$SH^{*}(T^{*}\P^{1},\Lambda_{\omega_{I}})=0$ implies $\ell(T^{*}\P^{1},\Lambda_{2})=0$. Therefore, $\ell(M,\Lambda_{2})$ is a stronger quantitative measure than $\ell(M)$ since $\ell(T^{*}\P^{1},\omega_{std})=2$. 
\end{example}

When we extend the coefficient ring to $\Lambda_{Mori}$ using the cone of effective curves, the Mori cone, of $H_{2}(X)$, we could get more precise results in log-birational geometry, that is a generalization of Pascaleff's work on log Calabi-Yau surfaces in \cite{pascaleff}. We will explain how to compute $\ell(M,\Lambda_{Mori})$ using homological mirror symmetry and cohomology from tropicalization in a subsequent paper.

Let $\phi:M\to M$ be a Liouville diffeomorphism on a Liouvlle domain $(M,\omega)$. Denote $\C^{*}\rtimes_{\phi} M:=(\R \times \R \times M)/\Z\cdot (s,\theta, x)\sim (s+1,\theta, \phi(x))$, the symplectic mapping cylinder of $\phi$, \cite{mclean19}. Denote $\Sigma\rtimes_{\phi} M:=\{(\R\times S^{1}\setminus \Z\times\{1\})\times M\}/(s,\theta,x)\sim (s+1,\theta,\phi(x))$ the symplectic mapping torus of $\phi$, [(1-5), \cite{kartal21}]. We can define similar measurments $\ell$ for symplectic mapping cylinder/tori.

\begin{example}
For the identity map $\phi=id_{M}$ on $M$, 

\begin{enumerate}
    \item $\C^{*}\rtimes_{id} M=\C^{*}\times M$, where $\C^{*}:=\C\setminus \{0\}$. By the K$\ddot{\text{u}}$nneth formula, $SH^{*}(\C^{*}\rtimes_{id} M)\cong SH^{*}(\C^{*})\otimes SH^{*}(M)$. Therefore,  $\ell(\C^{*}\rtimes_{\id} M)=\text{min}\{\ell(\C^{*})=\infty,\ell(M)\}=\ell(M)$.
    \item $\Sigma\rtimes_{id} M=\Sigma_{1,1}\times M$, where $\Sigma_{1,1}:= \T^{2} \setminus\{\text{a point}\}$, a once-punctured torus. Similarly, $\ell(\Sigma\rtimes_{id} M)=\ell(M)$.
\end{enumerate} 
\end{example}

\begin{example} \label{mappingcylindereg}
Let $\phi:M\to M$ be a Liouville diffeomorphism on a Liouville domain $(M,\omega)$. Assume that $\ell(M)=\infty$. Let $M_{\Supset k}$ be a $k$-handle attaching on $\partial M$. $\Phi$ be an extension of $\phi$ to $M_{\Supset k}$ by identity on attached locus. Then for any $k<dim_{\C}M$, $SH^{*}(\C^{*}\rtimes_{\phi} M)\cong SH^{*}(\C^{*}\rtimes_{\Phi} (M_{\Supset k}))\cong SH^{*}((\C^{*}\rtimes_{\phi} M)_{\Supset k+1})$ but $\ell(\C^{*}\rtimes_{\phi} (M_{\Supset k}))=\text{min}\{k+2, \ell(\C^{*}\rtimes_{\phi} M)\}$, $\ell((\C^{*}\rtimes_{\phi} M)_{\Supset k+1})=\text{min}\{k+1, \ell(\C^{*}\rtimes_{\phi} M)\}$.
\end{example}

Let $M$ be a Liouville domain. Recall that $SH^{*}(M)=\oplus_{g\in H_{1}(M)}SH^{*}_{g}(M)$ is a $H_{1}(M)$-graded algebra with the pair-of-pants product $\cdot_{p.p}$ : for $\gamma_{1}\in SH^{*}_{g_{1}}(M), \gamma_{2}\in SH^{*}_{g_{2}}(M)$, $\gamma_{1}\cdot_{p.p}\gamma_{2}\in SH^{*}_{g_{1}+g_{2}}(M)$. The pair-of-pants product counts three-punctured spheres with two positive periodic orbits and one negative periodic orbit as asymptotic orbits. So if the pair-of-pants product of two periodic orbits is the unit $e$ of the symplectic cohomology, which is the Poincar$\acute{\text{e}}$ dual of the fundamental class $[M]$, then the same idea gives us the following.

\begin{defn}\label{ell(p.p)}
$\ell_{p.p}(M):=\text{min}\{\text{deg }\alpha: \exists \alpha\neq 0 \in H^{*}(M) \text{ and } \exists g\in H_{1}(M) \text{ such that } \alpha =\gamma_{1}\cdot_{p.p} \gamma_{2} \text{ for some } \gamma_{1}\in SH^{*}_{g}(M), \gamma_{2}\in SH^{*}_{-g}(M)\}$.
\end{defn}

\begin{example}
If $SH^{*}(M)$ has an element $\gamma\neq 0$ satisfying $\gamma^{k}=[M]$, then $\ell_{p.p}(M)=0$.
\end{example}

\begin{example}
For the mapping cylinder of the identity map, $\ell_{p.p}(\C^{*}\rtimes_{id} M)=0$. Moreover, for a Liouville diffeomorphism $\phi:M\to M$, $\ell_{p.p}(\C^{*}\rtimes_{\phi} M)\geq \ell_{p.p}(\C^{*}\rtimes_{id} M)$.
\end{example}

\begin{defn}\label{ell(phi)}
$\ell_{p.p}(\phi):=\text{min}\{k:\ell_{p.p}(\C^{*}\rtimes_{\phi^{k}} M)=0\}$, where $\phi^{k}:=\phi\circ \cdots \circ \phi$, $k$th iteration of $\phi$.
\end{defn}

\begin{example}
If $\phi^{k}=id$ for some $k$, then $l_{p.p}(\phi)\leq k$.
\end{example}

\begin{example}
Define $f:\C^{n+1}\to \C$ to be $f(z_{1},z_{2},\cdots,z_{n+1)}:=z_{1}^{k_{1}}+\cdots+z_{n+1}^{k_{n+1}}$ with an isolated singular point at $0$. Let $S_{\epsilon}$ be the sphere of radius $\epsilon$ centered at $0$ and $L_{f}:=S_{\epsilon}\cap f^{-1}(0)$. Let $\phi:f^{-1}(\delta)\to f^{-1}(\delta)$ be a Milnor monodromy map. Then $\phi^{l.c.m(k_{i})}=id$. Therefore, $l_{p.p}(\phi)\leq l.c.m(k_{i})$.
\end{example}

From the proof of the main theorem, 

\begin{theorem}
If $\ell_{p.p}(M)=0$, then $M$ is $\C^{*}$-uniruled.
\end{theorem}

Let us briefly explain the main idea to prove the theorem \ref{main theorem0}. Assume that for some $[\alpha]\in H^{\ell}(M)$, $\delta(\gamma)=[\alpha]$ for some $0\neq\gamma\in SH^{\ell-1}_{+}(M)$. Then there exists a Floer cylinder $u:\R\times S^{1}\to M$ with negative asymptotic periodic orbit $\gamma$ and with positive asymptotic orbit which lands on homology class represented by a pseudo-cycle of a unstable submanifold of index $2n-\ell$ in Morse homology. Applying Bourgeois-Oancea's correspondece theorem on the Morse-Bott symplectic cohomology, we get a pseudo-holomorphic curve in $M$ passing through a point on a unstable submanifold representing a cycle in $HM^{\ell}(M)$. McLean's degeneration to the normal cone method in \cite{mclean14} let us get a non-constant rational map $f:\C\to M$ passing through a point in a cycle of index $2n-\ell$. Because of the set-theoretic intersection with pseudo-cycle, we get a family of rational curves with evaluation. By Grothendieck's construction of Hilbert schemes (or Quot Schems), we get a uniruled subvariety of dimension at least $\lfloor \ell/2 \rfloor$ that parametrizes the family of rational curves. For readers' convenience, let us provide the structure of this paper. To read the section \ref{section3} on proof of the main theorem, the readers only need to check definitions in subsections \ref{Liouville domain corr to affine} and \ref{def2}. Other symplectic structures and properties of symplectic cohomology, summarized in the section 1 and 2, will be used in the section 4 on applications.\\

{\it Acknowledgments}: This paper is a part of my PhD thesis. I would like to thank my advisor Mark McLean for suggesting this project and for his kind guidance, helpful advice, patience and warm encouragement. I would also like to thank Jason Starr for his inspiring lectures, helpful advice and warm words. I would like to thank my friends and professors in the Math department at Stony Brook for their continued support and for providing a friendly environment. I have also learned a lot by attending, seminars and classes at Stony Brook University(Symplectic/Contact/Gauge theory seminar, Algebraic Geometry seminar, Student AG seminar), British Isles Graduate Workshop in 2018, Kylerec Workshop, Categorical Symplectic Topology Conference, MSRI summer school on Current Trends in Symplectic Topology, Joint Math Meetings 2020 (AMS-AWM Special Session for Women in Symplectic and Contact Geometry), Symplectic Zoominar, and Western Hemisphere Virtual Symplectic Seminar. I appreciate all the organizers and people whom I met there. During this project, I was supported by John W. Milnor Endowed Graduate Fund Travel Award, Stony Brook Math Department Summer Research Funding, and GSEU Professional Development Award at Stony Brook.

\section{Symplectic Cohomology}
We review known definitions and theorems about symplectic cohomology in this section so experts who are familiar with the basics can jump to the section \ref{section3}. To read section \ref{section3}, the subsections 2.1, 2.2, and 2.3 are needed only. Symplectic cohomology is well-defined on the completion of Liouville domains [\cite{viterbo99}, \cite{hermann}, \cite{seidelbiased}, \cite{mclean08}, \cite{bourgeois.oancea09}, \cite{abouzaid15}, \cite{ganatrapomerleano}, and \cite{diogilisi18s}]. Every smooth affine variety has a unique Liouville domain, up to isotopy through Liouville domains, whose completion is symplectomorphic to the affine variety \cite{mclean12}. Since, on Liouville domains, symplectic form is exact, many assumptions, computations related to the definition of symplectic cohomology become simpler. Symplectic cohomology still be well-defined for non-exact convex symplectic manifold with the weak monotonicity assupmtion, for example, the blow-up of $\P^{1}$ at a point \cite{ritter10}, \cite{mclean-ritter18}. To get a $\Z$-grading on cohomology, we assume a Liouville domain with the vanishing first Chern class, $c_{1}=0$. For simplicity, we mainly consider simply-connected varieties, $\pi_{1}=0$, for example, complete intersections of $\text{dim}_{\C}>3$ by the Lefschetz hyperplane theorem, then consider the action by the fundamental group for $\pi_{1}\neq 0$. For symplectic manifolds with $\pi_{1}\neq 0$, we use Novikov ring coefficient.

\subsection{Symplectic Structure on an Affine Variety or on a Log Pair}

A Stein manifold is a properly embedded complex submanifold of $\C^{N}$ for some $N$, equivalently, a complex manifold with an exhausting plurisubharmonic function. Smooth affine varieties are Stein manifolds of finite type, which means a complex manifold $(M, J)$ that can be properly embedded in $\C^{N}$ for some $N$ by admitting a function $f:M\to \R$ satisfying the following conditions,
\begin{itemize}
    \item (Plurisubharmonic) $(-dd^{c}f)(v,Jv)>0$ for all non-zero vectors $v$ where $d^{c}:=d\circ J$,
    \item (Exhausting) $f:M\to \R$ is bounded below and the preimage of every compact set is compact,
    \item (Of finite type) $f$ has only finitely many singularities.
\end{itemize}

Since the space of exhausting plurisubharmonic functions on a given Stein manifold is convex, so contractible and open in $C^{2}$, we can choose an exhausting plurisubharmonic generalized Morse function $f:M\to \R$ on a Stein manifold $(M,J)$. 

A Weinstein structure on a symplectic manifold $(M, \omega)$ is a pair $(f, Z)$, where $f:M\to \R$ is an exhausting generalized Morse function, and $Z$ is a complete Liouville vector field that is gradient-like for $f$. Given a Stein manifold $(M,J,f:M\to \R)$, one can find a symplectic form $\omega_{f}$ so that $(f, Z_{f}:=\nabla_{g_{f}}f)$ become a Weinstein structure on $(M, \omega_{f})$, where. Given a Weinstein structure $(Z,f)$ on $(M, \omega)$, one can construct a Stein structure $(J, f)$ such that the canonical Weinstein structure on $(M, J, f)$ is Weinstein homotopic to the given one. We refer \cite{cieliebakeliashberg12} for the detail. A Weinstein structure on a symplectic manifold carry a Morse-theoretic handle decomposition, called Weinstein handle decomposition: Weinstein cobordism is attaching a standard handle along a neighborhood of an isotropic sphere in the contact level set with a Liouville vector field transverse to contact level sets [Theorem 5.1, \cite{weinsteinhandle}].

\subsubsection{Liouville Domain corresponds to Affine Variety} \label{Liouville domain corr to affine}
A Liouville domain is a compact manifold $\underline{M}$ with boundary $\partial \underline{M}$ and a $1$-form $\theta_{\underline{M}}$ with $d\theta_{\underline{M}}$ as a symplectic form and satisfying that for any vector field $Z$ with $\iota_{Z}d\theta_{\underline{M}}=\theta_{\underline{M}}$ is transverse to $\partial \underline{M}$ and pointing outwards along $\partial \underline{M}$. We call such a vector field $Z$, the Liouville vector field. Using the Liouville flow of $Z$, we glue an half-infinite cylindrical end, $([-\epsilon,\infty)\times \partial \underline{M},d(e^{r}\alpha))$, which is a symplectization of the contact boundary, to $\underline{M}$ along $((-\epsilon,0)\times \partial \underline{M}$ so that $Z$ extends to $\partial _{r}$ on $((0,\epsilon)\times \partial \underline{M}$. Then we get a completion of a Liouville domain, $\widehat{\underline{M}}:=\underline{M}\cup _{\partial \underline{M}}([0,\infty)\times \partial \underline{M})$.

In [(4b), \cite{seidelbiased}], [Lemma 2.1, \cite{mclean12}], for any smooth affine variety $M$, there is an associated Liouville domain $\underline{M}$, whose completion, constructed by attaching a positive half of symplectization of contact boundary $\partial\underline{M}$ of $M$ to $M$, is symplectomorphic to $M$, which is unique up to Liouville isomorphism, i.e. an isotopy through Liouville domains which let us consider the symplectic cohomology of affine varieties as the completion of Liouville domains. The construction is the following: \textit{Let $\iota:M\hookrightarrow \C^{N}$ be a algebraic embedding as a closed subvariety. Let $(r_{i},\theta_{i})$ be polar coordinates for the $i$'th factor in $\C^{N}$. Define $R:=\frac{1}{4}\sum _{i}r^{2}_{i}$ and $\theta:=-d^{c}R=\frac{1}{2}\sum_{i}r_{i}^{2}d\mathcal{\theta}_{i}$, $d\theta=\sum_{i}r_{i}dr_{i}d\mathcal{\theta}_{i}$, the standard symplectic form on $\C^{N}$. Then there exists $C>0$ such that for all $c\geq C$, $((R|_{M})^{-1}(-\infty,c],\theta_{M})$ is a Liouville domain whose completion is symplectomorphic to $(M, d\theta|_{M})$. Such a number $C$ is the minimum of numbers $c$ so that $R|_{M}$ has no singularities for $(R|_{M})^{-1}(c,\infty)$}, [Lemma 2.1, \cite{mclean12}].

By Hironaka's resolution of singularities \cite{hironaka}, we can consider an affine variety  $M$ as the complement of simple normal crossing divisors $D=\{D_{i}\}^{l}_{i=1}$ in a projective variety $X$. We can consider a log pair $(X, D)$ as an affine variety $M=X\setminus D$ with compactification data. Let us recall the symplectic structure of a $X\setminus D$ from [Lemme 5.19, Theorem 5.20, \cite{mclean12}], which has good properties that divisors are intersecting positively and the wrapping numbers of them are negative and small tubular neighborhood of them are fibrations with fibers, product of punctured disks. Recall [Theorem 5.20, \cite{mclean12}]: \textit{Let $M$ be a smooth affine variety. Then $M$ is convex deformation equivalent to a finite type convex symplectic manifold $(W, \theta_{W})$ with the following properties:
\begin{enumerate}
    \item $W$ is symplectomorphic to $X\subset D$, where $D=\{D_{i}\}^{k}_{i=1}$ are co-dimension$_{\R}$ $2$ symplectic submanifolds transversely intersecting.
    \item There are neighborhoods $UD_{i}$ of $D_{i}$ and fibration $\pi_{i}:UD_{i}\twoheadrightarrow D_{i}$ such that 
    \begin{itemize}
        \item For $1\leq i_{1}<i_{2}<\cdots < i_{l}\leq k$, $\pi_{i_{l}}\circ \cdots \circ \pi_{i_{1}}: \cap^{l}_{j=1}UD_{i_{j}}\twoheadrightarrow D_{\{i_{1},\cdots , i_{l}\}}$\\
        has fibers that are symplectomorphic to $\prod _{j=1}^{l} B_{\epsilon}$, where $B_{\epsilon}$ is the disk of radius $\epsilon$.
        \item On the fiber $\prod _{j=1}^{l} B_{\epsilon}$ of $\pi_{i_{l}}\circ \cdots \circ \pi_{i_{1}}$, one of factor map $\pi_{i_{k}}$, $1\geq k \geq l$, maps this fiber to itself.
        \item The symplectic structure on $UD_{i}$ induces a natural connection for $\prod _{j=1}^{l} B_{\epsilon}$ given by the $\omega$ orthogonal vector bundles to the fibers. 
    \end{itemize}
    \item $\theta_{W}$ restricted to the fiber $B_{\epsilon}$ of $\pi_{i}$ is equal to $(r_{i}^{2}+\kappa_{i})d\theta_{i}$ for some $\kappa < 0$.
\end{enumerate}}

For any affine variety $M$, we can construct two Liouville domains $\underline{M}, \overline{M}$ with $\underline{M}\hookrightarrow M\hookrightarrow \overline{M}$, where $\underline{M}, \overline{M}$ are Liouville deformation equivalent. Indeed, two inclusions are exact symplectic embeddings of a codimension $0$ submanifold and a homotopy equivalence. This fact will be used when we make holomorphic curves on $\underline{M}$ into holomorphic curves on $\overline{M}$, then get rational curves on $M$. Let us state the sandwich theorem for affine varieties [Lemma $4.3$, Lemma $4.4$, \cite{mclean14}]\label{sandwich}: \textit{
Let $M$ be a smooth affine variety. Let $L$ be a line bundle associated to a simple normal crossing compactification $(X,D)$ of $M$ and $s$ be a section with $s^{-1}(0)=D$. Define $f:=-log||s||$, $\theta=-d^{c}f$, $\omega=-dd^{c}f$. Then
\begin{itemize}
    \item $f$ is a proper and bounded below function on $M$ with $df(X_{\theta})>0$ on $f^{-1}([c,\infty))$ for some $c\gg1$, and 
    \item there is an exact symplectic embedding $\iota:(M, \theta)\hookrightarrow (\overline{M}:=f^{-1}((-\infty,c]), \nu\cdot\theta)$ for some $\nu$ which is a homotopy equivalence and $(\overline{M}, \nu\cdot\theta))$ is Liouville deformation equivalent to the associated Liouville domain $\underline{M}$ of the affine variety $M$. 
\end{itemize}}

\subsubsection{Symplectic Convex Lefschetz fibration Structure on Affine Varieties}
Symplectic geometric understanding on Lefschetz fibration was initiated by Donaldson. Seidel has foundational works on categorification in symplectic geometry of Picard-Lefschetz theory. 
McLean constructed symplectic convex Lefschetz fibration structure on any (Kaliman modification of) affine varieties [Theorem 2.32, \cite{mclean09}]. A symplectic convex Lefschetz fibration structure with monodromy information is a complexified Morse function and is useful for systematic understanding on symplectic structure, for example, dimension deduction. So let us summarize the known results [Chap.III, \cite{seidel.book}, \cite{mclean08}, Chap.8, \cite{mclean09}, \cite{smith.review}]. Here, systematic understanding means that we chop the structure into smaller dimensional pieces or a complex of simpler objects like Morse-theoretic building blocks and then glue them to get an homological invariants, like spectral sequences, long exact sequences. 

Let $X$ be a projective variety with an ample divisor $D$ and $M:=X\setminus D$ be an affine variety. Let $\mathcal{L}$ be an ample line bundle on $X$ with two holomorphic sections $s, t$ with $s^{-1}(0)=D$. Define a map $p:=t/s :M\to \C$.

\begin{defn} An (open) algebraic Lefschetz fibration structure on a pair $(M,p:M\to \C)$ is following.
\begin{enumerate}
    \item $\overline{t^{-1}(0)}$ is smooth, reduced and intersects each stratum of $D$ transversally.
    \item $p$ has only nondegenerate critical points and there is at most one of these points on each fiber.
\end{enumerate}
\end{defn}

The symplectic structure on $M$, $\omega=dd^{c}(-log||s||)$, is compatible with $p$, in other words, $\omega$ restricted to the fiber is a symplectic form away from singular part. But the horizontal boundary is not trivial. To define symplectic cohomology on Lefschetz fibrations, we need compact convex condition.

\begin{defn} A symplectic Lefschetz fibration is a pair $(M, p:M\to B)$, where $(M,\omega=d\theta)$ is a compact exact symplectic manifold with corners whose boundary is the union of horizontal $\partial_{h}M$ and vertical $\partial_{v}M$ meeting in a codimension $2$ corner and $B$ is a surface with boundary. The smooth map $p:M\to B$ satisfies the following.
\begin{enumerate}
    \item The map $p$ has only finitely many critical points.
    \item $\omega$ restricts to a symplectic form on smooth locus of each fiber $p^{-1}(b)\setminus M^{\text{crit}}$, for every $b\in B$.
    \item There is an integrable complex structure $J_{0}$ (resp. $j_{0}$) defined on some neighborhood of $M^{\text{crit}}$ (resp. $S^{\text{crit}}$) such that $p$ is $(J_{0},j_{0})$-holomorphic map near $M^{\text{crit}}$.
    \item There is at most one critical point in each fiber and at any critical point, the Hessian $D^{2}p$ is nondegenerate as a complex quadratic form. 
    \item $\omega$ is K\"ahler form for $J_{0}$ near $M^{\text{crit}}$.
\end{enumerate}
\end{defn}

\begin{defn} A symplectic Lefschetz fibration $(M, p:M\to B)$ is compact convex if 
\begin{enumerate}
    \item (Exact) $p:M\to B$ is a proper map with $\partial_{v} M=p^{-1}(\partial B)$ and such that $p|_{\partial_{v}M}$ is a smooth fiber bundle. Also, there is a neighborhood $N$ of $\partial_{h}(M)$ such that $p|_{N}:N\to B$ is a product fibration $B\times \text{nbhd}(\partial F)$ where exact symplectic structure is pulled back from $\text{nbhd}(\partial F)$.
    \item (Compact) Fiber $(F, \theta|_{F})$ is a compact convex symplectic manifold. 
\end{enumerate}
\end{defn}

Using well-defined exact parallel transport maps for an algebraic Lefschetz fibration [Section 2, \cite{fukaya.seidel.smith.08}], 

\begin{lemma}[Lemma 8.6, \cite{mclean09}] \label{sympl lefs}
Let $(M, p:M\to \C, \omega:=-dd^{c}(-log||s||))$ be an open algebraic Lefschetz fibration. Then there exists a convex symplectic structure $\theta'$ so that $(M,p,\theta')$ is a symplectic convex Lefschetz fibration. 
\end{lemma}

Kaliman modification of a triple $(M, D, C)$ of an smooth affine variety $M$, an irreducible divisor $D$, a closed subvariety $C$ contained in the smooth part of $D$ is known to be an operation making a new affine variety that is diffeomorphic to the original one but not biholomorphic. McLean constructed symplectic convex Lefschetz fibration structure on Kaliman modification and showed that they are convex deformation equivalent to each other \cite{mclean09}. 

\begin{defn}
The Kaliman modification $M^{\vee}$ of $(M,D,C)$ is defined by $M^{\vee}:=\mathcal{B}l_{C}M\setminus \tilde{D}$, where $\mathcal{B}l_{C}M$ is the blow up of $M$ along $C$ and $\tilde{D}$ is the proper transformation of $D$. 
\end{defn}

\begin{lemma}[Lemma 8.7, \cite{mclean09}]
Let $M$ be an affine variety of $\text{dim}_{\C}M\geq 3$ with compactifying projective variety $X$ by $D$. Choose a irreducible divisor $Z$ of $X$ so that $D\cup Z$ is ample. Choose a point $q\in Z\cap M$ smooth in $Z$. Denote $M^{\vee}:=\text{Kaliman}(M,(Z\cap M), \{q\})$ and $M^{\sqcup}:=M\setminus Z$. Then there exist algebraic Lefschetz fibrations $p^{\vee}:M^{\vee}\to \C$, $p^{\sqcup}:M^{\sqcup}\to \C$ such that $p^{\sqcup}$ is a subfibration of $p^{\vee}$ and the singularities of $p^{\vee}$ are contained in $M^{\sqcup}$.
\end{lemma}

We apply the same idea of proof of Lemma 8.7 in \cite{mclean09} to a half-point attachment from the weighted blow-up of a point with weights $(1,1,b)$. 

\begin{defn} [Definition 1.1, \cite{kishimoto06}] \label{halfpointattach}
Let $M$ be a normal quasi-projective complex threefold and $X$ be a normal projective threefold compactifying $M$ with $D:=X\setminus M$. Choose a point $q\in D$ where $X$ is smooth. Let $f:\tilde{X}\to X$ be the weighted blow-up at $q$ with weights $(1,1,b)$, where $b\in \N$. Let $E\cong \P^{2}_{(1,1,b)}$ be the exceptional divisor of $f$ and $\tilde{D}:=f^{-1}(D)$. $\tilde{M}:=\tilde{X}\setminus \tilde{D}$ is called to be a half-point attachment to $M$ of $(b,k)$-type if $\tilde{D}|_{E}=\sum_{j=1}^{k}m_{j}l_{j}$, where $l_{j}$'s are the mutually distinct rulings on $E$ and $m_{j}\in \N$ with $\sum _{j=1}^{k}m_{j}=b$. $\tilde{M}\setminus M\cong \C^{(k-1)*}\times \C$.
\end{defn}

We use the same idea of Lemma 8.7 in \cite{mclean09} to the following case.
Let $\tilde{M}$ be a half-point attachment to an affine threefold $M$ at a smooth point on a hypersurface of $(b,1)$-type. We want to show that there exist algebraic Lefschetz fibration $\tilde{p}:\tilde{M}\to \C$, $p:M\to \C$ such that $p$ is a subfibration of $\tilde{p}$ and the singularity of $\tilde{p}$ are contained in $M$. Let $X$ be a projective variety with ample divisor $D$ with $X\setminus D=M$. Assume that $D$ is effective and very ample so that we can choose a rational function $f$ on $X$ with $Y:=\overline{f^{-1}(0)}$ satisfying that $D\sim Y+D'$ for an ample divisor $D'$ and $Y\cap (X\setminus D')$ is reduced and irreducible. Define $M':=X\setminus D'$ an affine variety containing $M$. Choose a smooth point $q\in Y\setminus D'$.

Let $\pi:\tilde{X}\to X$ be a weighted blow-up at $q$ with weights $(1,1,b)$. $\tilde{D}$ is a strict transformation of $D$ under $\pi$. Let $s$ be a section of a line bundle $\mathcal{O}_{X}(D)$ with $s^{-1}(0)=Y+D'$. Choose another section $t$ of $\mathcal{O}_{X}(bD)$ satisfying the following,
\begin{enumerate}
    \item $p:=\frac{t}{s^{b}}:M\to \C$ is an algebraic Lefschetz fibration on $M=M'\setminus Y$, and
    \item $t^{-1}(0)$ is a smooth subvariety passing through $q$ and transverse to $Y$ at the point $q$.
\end{enumerate} 
Let $E$ be the exceptional divisor $\pi^{-1}(q)$. Then $(\pi^{*}s)^{-1}(0)=E+D''$ for some divisor $D''$. Choose an effective divisor $D^{\vee}$ on $\tilde{X}$ compactifying $\tilde{M}:=\pi^{-1}(M')\setminus \tilde{Y}$ such that $D^{\vee}-E$ is ample. We can choose a meromorphic section $h$ of $\mathcal{O}_{\tilde{X}}(bD^{\vee}-bE)$ such that $h^{-1}(0)\subset D^{\vee}$, $\text{ord}_{E}(h)=-b$ and such that $h$ is holomorphic away from support of $D^{\vee}$, $E$. Then $\pi^{*}s^{b}\otimes h\in H^{0}(\mathcal{O}(bD^{\vee}-bE+b\pi^{*}D))$ is non-zero away from $D^{\vee}$. Therefore, the map $\tilde{p}:=\frac{(\pi^{*}t)\otimes h}{(\pi^{*}s)\otimes h}$ defines an algebraic Lefschetz fibration on $\tilde{M}$ and the restriction of $\tilde{p}$ on $M$ is $p$.
Choose a holomorphic coordinate $(z_{1},z_{2},z_{3})$ around $q$ and a holomorphic trivialization of $\mathcal{O}_{X}(Z+D')$ so that $s(z_{1},z_{2},z_{3})=z_{1}^{b}$ and $t(z_{2},z_{2},z_{3})=z_{2}^{b}$. Locally on $E$, $\{((z_{1},z_{2},z_{3}),[Z_{1}:Z_{2}:Z_{3}]_{(1,1,b)})\in \C^{3}\times \P^{2}_{(1,1,b)}:Z_{1}\cdot z_{2}=Z_{2}\cdot z_{1}, Z_{2}^{b}\cdot z_{3}=Z_{3}\cdot z_{2}^{b}, Z_{1}^{b}\cdot z_{3}=Z_{3}\cdot z_{1}^{b}\}$. To show that all the critical points of $\tilde{p}$ are in $M$, locally on a chart $(Z_{1}=1)$, we can choose a trivialization of $\mathcal{O}_{\tilde{X}}(\tilde{D}-E)$ so that $\pi^{*}s^{b}=z_{1}^{b}$, $\pi^{*}t=
Z_{2}^{b}\cdot z_{1}^{b}$, and $h=\frac{1}{z_{1}^{b}}$ from $\text{ord}_{E}(h)=-b$. Hence, $\tilde{p}=Z_{2}^{b}$ has no singular points near $E\setminus \tilde{Y}$ where $Z_{2}\neq 0$.

\begin{remark} It seems interesting to understand the Lefschetz fibration structure of a half-point attachment of $(b,k)$-type in general. However, the author does not know about it. 
\end{remark}

\subsubsection{Mapping Cylinder of Milnor Monodromy Map of Isolated Singularities}

Let $f:\C^{n+1}\to \C$ be a polynomial with an isolated singular point at $0$. Let $S^{2n+1}_{\epsilon}\subset \C^{n+1}$ be the $(2n+1)$-sphere of radius $\epsilon$ centered at $0$ with the standard contact structure $\eta:=TS^{2n+1}_{\epsilon}\cap J_{std}TS^{2n+1}_{\epsilon}$. The link of $f$ at $0$ is a contact submanifold $L_{f}:=f^{-1}(0)\cap S^{2n+1}_{\epsilon}\subset (S^{2n+1}_{\epsilon}, \eta)$ where $\epsilon>0$ is sufficiently small. Let $\mathcal{N}_{S^{2n+1}}L_{f}:=(TS^{2n+1}_{\epsilon}|_{L_{f}}/TL_{f})$ be the normal bundle of $L_{f}$ with the pushforwarded symplectic form from $T^{\perp}:=\{v\in \eta|_{x}:x\in L_{f}, d\alpha_{S^{2n+1}_{\epsilon}}(v,w)=0, \forall w\in \eta_{S^{2n+1}_{\epsilon}}|_{x}\cap TL_{f}|_{x}\}$ with a trivialization $\Phi_{f}:=(id_{L_{f}}, \overline{df}):\mathcal{N}_{S^{2n+1}}L_{f}\to L_{f}\times \C$. Given an isolated singularity $(f:\C^{n+1}\to \C,0)$, we can assoicate a contact open book, called the Milnor open book of $f$, and a Liouville domain, called the mapping cylinder of the Milnor monodromy map on the Milnor fibration. 

Let $arg(f):\C^{n+1}\setminus f^{-1}(0)\to \R/2\pi\Z$ be the argument of $f$. The Milnor fibration associated to an isolated singularity $(f,0)$ is a smooth fibration, 
$$arg(f):S^{2n+1}_{\epsilon}\setminus f^{-1}(0)\to \R/2\pi \Z, arg(f)(z):=arg(f(z)),$$
with a fiber $M_{f}:=arg(f)^{-1}(0)$, called Milnor fiber. There is a compactly supported diffeomorphism, called the Milnor monodromy map, $\phi:M_{f}\to M_{f}$.
Then a contact pair with normal bundle data $(L_{f}\subset S^{2n+1}_{\epsilon}, \eta, \Phi_{f})$ is supported an open book $(S^{2n+1}_{\epsilon}, arg(f)|_{S^{2n+1}_{\epsilon}})$.
An abstract contact open book is a triple $(M,\theta_{M},\phi)$ where $(M,\theta_{M})$ is a Liouville domain and $\phi:M\to M$ is an exact symplectomorphism with support in the interior of $M$. Given an abstract contact open book $(M,\theta_{M},\phi)$, we can construct contact structure $\alpha_{T_{\phi}}$ on the mapping torus $\pi_{T_{\phi}}:T_{\phi}:=M\times [0,1]/ \sim$ of $\phi:M\to M$ by
\begin{enumerate}
    \item $(x,1)\sim (\phi(x),0)$,
    \item $\pi_{T_{\phi}}(x,t):=t$ for all $(x,t)\in T_{\phi}$, and
    \item $\alpha_{T_{\phi}}:=\theta_{M}+d(\rho(t)F_{\phi})+Cdt$, where $F_{\phi}:M\to \R$ is a smooth function with support in the interior of $M$ with $\phi^{*}\theta_{M}=\theta_{M}+dF_{\theta}$, $\rho:[0,1]\to [0,1]$ is a smooth function equal to $0$ near $0$ and $1$ near $1$, $C>0$ is large enough so that $\alpha_{T_{\phi}}$ is a contact form. 
\end{enumerate}  
The above construction gives a one-to-one correspondence between graded abstract contact open books up to isotopy and graded open books up to isotopy [\cite{giroux02}, \cite{caubel.nemethi.popescupampu.06}, Theorem 3.15, \cite{mclean19}].

Let $(M,\theta_{M}, \phi)$ be an abstract contact open book. Let $\check{\phi}$ be a small positive slope perturbation of $\phi$. The mapping cylinder of $\phi$ (Definition B.1, \cite{mclean19}) is a triple $(\C^{*}\rtimes_{\phi} M,\pi_{\check{\phi}}, \theta_{\check{\phi}})$ where
\begin{enumerate}
    \item $\C^{*}\rtimes_{\phi} M:=(\R\times \R \times M)/\Z$, where $(s,t,x)\sim (s,t-1,\check{\phi}(x))$,
    \item $\check{\phi}:\C^{*}\rtimes_{\phi} M\to \R\times (\R/\Z)\cong \C^{*}$ defined by $\check{\phi}(s,t,x)=(s,t)$, and
    \item $\theta_{\check{\phi}}:=sdt+\kappa \theta_{M}+\kappa d(\rho(t)F_{\check{\phi}})$ where
    \begin{itemize}
        \item $F_{\check{\phi}}:M\to \R$ is a smooth function with support in the interior of $M$ satisfying $\check{\phi}^{*}\theta_{M}=\theta_{M}+dF_{\check{\phi}}$,
        \item $\rho:[0,1]\to [0,1]$ is a smooth function $\rho(0)=0, \rho(1)=1$,
        \item $\kappa>0$ is a constant small enough so that $d\theta_{\check{\phi}}$ is symplectic.
    \end{itemize}
\end{enumerate}

\subsubsection{Legendrian Front Projection and Flexible Weinstein Cobordism}

In this subsection, we review definitions and theorems of Legendrian front projection and flexible Weinstein cobordisms,  \cite{weinsteinhandle}, \cite{cieliebakeliashberg12}, \cite{murphy12}. During fundamental lectures in 1966-1968 [Section 46, Appendix 16, \cite{arnoldlecture}], Arnol'd emphasized geometric optics as the fundamental notions of Hamiltonian mechanics so that one keep in mind when working on symplectic/contact structure. He provided one of examples of Lagrangian submanifolds is the manifold of all oriented normals to a smooth submanifolds in $\R^{n}$ is a Lagrangian submanifold of the space of lines in $\R^{n}$ passing through the origin, which is the cotangent bundle of the $(n-1)$-sphere. The fibration from the space of lines in $\R^{n}$ to the unit sphere of directions is a Lagrangian fibration. A projection of a Lagrangian submanifold to the base of a Lagrangian fibration is called a "Lagrangian mapping", for example, mapping each point of a transversely oriented hypersurface in $\R^{n}$ to the unit vector at the origin in the direction of the normal. A caustic of a Lagrangian mapping is defined as the set of critical values of the mapping. The caustic of a Lagrangian mapping is the image of the cuspidal edge of the front of a Legendrian mapping under a projection from the space of $1$-jets onto the phase space.

Murphy introduced loose Legendrians, roughly speaking, smooth Legendrians look like unfolding of wrinkled Legendrian(zig-zag, Figure 2 in \cite{murphy12}) along some markings which let us wiggle them to get the h-principle [Figure 3, \cite{murphy12}]. Define a (front-projected) plane curve $\psi:\R\to \R^{2}$ and its rescaling by
\begin{itemize}
    \item $\psi(u)=(\psi^{x}(u),\psi^{z}(u))=(u^{3}-u, \frac{9}{4}u^{5}-\frac{5}{2}u^{3}+\frac{5}{4}u).$
    \item $\psi_{\delta}(u)=(\delta^{\frac{3}{2}}\psi^{x}(u/\sqrt{\delta}),\delta^{\frac{5}{2}}\psi^{z}(u/\sqrt{\delta}))=(u^{3}-\delta u,\frac{9}{4}u^{5}-\frac{5\delta}{2}u^{3}+\frac{5\delta ^{2}}{4}u).$
\end{itemize}

Define $\Psi:\R^{n}\to \R^{n+1}$ by $\Psi(x_{1},\cdots,x_{n-1},u)=(x_{1},\cdots, x_{n-1}, \psi_{x_{n-1}}(u)).$ The singular set is $\{x_{n-1}=3u^{2}\}$. For $x_{n-1}>0$, the singularity is cusp $\times $ $\R^{n}$ and, for $x_{n-1}=0$, singularity is called an unfurled swallowtail.

\begin{defn} [p. 7, \cite{murphy12}]
\begin{enumerate}
    \item A wrinkle is a map $w:\R^{n}\to\R^{n+1}$ defined by $w(\boldsymbol{x},u):=(\boldsymbol{x},\psi_{1-|\boldsymbol{x}|^{2}}(u))$, for $\boldsymbol{x}\in \R^{n-1}$.
    \item A wrinkled embedding of $V^{n}$ to $W^{n+1}$ is a topological embedding $f:V^{n}\to W^{n+1}$ that is smooth away from some finite collection of codimension 1 spheres $S^{n-1}_{j}\subset V^{n}$ near which the map is a wrinkle.
    \item A wrinkled Legendrian embedding is a topological embedding $f:\Lambda \to (N^{2n+1},\xi)$ satisfying the following,
    \begin{enumerate}
        \item $\text{Image} (df)\subset \xi$. 
        \item $df$ is full rank outside a subset of codimension $2$.
        \item The singular set is diffeomorphic to a disjoint union of Legendrian wrinkles $\{S^{n-1}_{j}\}$, near which the front projection of the image of $f$ is a wrinkled embedding.
    \end{enumerate}
\end{enumerate}
\end{defn}

Let $C\subset \R^{3}$ be the cube of side length $1$ and $\Lambda_{0}\subset C$ be a properly embedded Legendrian arc with zig-zag front which is $\{y=z=0\}$ near the boundary. Let $Z_{\rho}:=(-\rho,\rho)\subset \R$ and  $V_{\rho}:=\{(q,p):|q|<\rho, |p|<\rho\} \subset T^{*}\R^{n-1}$. Then $\Lambda_{0}\times Z_{\rho}$ is a Legendrian submanifold of $C\times V_{\rho}$, called the stabilization of the Legendrian $\{y=z=0\}\times Z_{\rho}\subset (C\times V_{\rho})$ along $Z_{\rho}$ [Section 4.3, \cite{ekholm.etnyre.sullivan05}, Section 7.4, \cite{cieliebakeliashberg12}].

\begin{defn} [Section 4.2, \cite{murphy12}]
    A Legendrian submanifold $\Lambda$ of a contact manifold $(N^{2n+1},\eta)$ is called loose if for each connected component of $\Lambda$, there is an open subset $U\subset N$ satisfying $(\Lambda\cap U,U)$ is contactomorphic to $(\Lambda_{0}\times Z_{\rho},C\times V_{\rho})$.
\end{defn}

\begin{defn}[Definition 11.29, \cite{cieliebakeliashberg12}]
    An elementary $2n$-dimensional Weinstein cobordism $(W,\omega, X, \phi)$ is called flexible if the attaching spheres of all index $n$ handles form a loose Legendrian link in $\partial_{-}W$. A Weinstein cobordism or manifold is called flexible if it can be decomposed into elementary Weinstein cobordisms. 
\end{defn}

\begin{example}
\begin{enumerate}
    \item A subcritical handle attaching is a flexible Weinstein cobordism [Remark 11.30, \cite{cieliebakeliashberg12}].
    \item Let $(W, \omega)$ be a convex symplectic manifold of dimension $2n\geq 4$ of Morse type at most $n+1$. Then its 1-stabilization $(W,\omega)\times (\R^{2},\omega_{st})$ is flexible Weinstein. [Theorem 1.1, \cite{eliashberg.ogawa.yoshiyasu21}]
\end{enumerate}
\end{example}

Casals-Murphy developed a recipe to construct a Legendrian front projection of Weinstein Lefschetz (bi)fibration of affine varieties: given Weinstein Lefschetz fibration structure with fiber $(F,\lambda, \phi)$ is a plumbing of vanishing spheres along tree, one apply sequences of Seidel's half-Dehn twist on matching paths, draw a diagram of Legendrian front projection of the Legendrian lifts of the vanishing cycles to the contact manifold $\partial(F\times D^{2}, \lambda+\lambda_{standard})$, and simplify the diagram by Reidemeister moves and Legendrian handleslides [Recipe 3.3, \cite{casalsmurphy19}]. Using the recipe, if the Legendrian front diagram contains a zig-zag, then the Legendrian submanifold is loose and the Weinstein manifold is flexible. For example, they explicitly apply the recipe to the Koras-Russell cubic $\mathfrak{C}:=\{(x,y,z,w)\in \C^{4}:x+x^{2}y+w^{3}+z^{2}=0\}$, which is known to be diffeomorphic to $\C^{3}$ but not algebraically isomorphic to it. They proved that indeed $\mathfrak{C}$ is Weinstein-equivalent to $\C^{3}$ by the h-principle of the loose Legendrian. 

Acu-(Capovilla-Searle)-Gadbled-Marinkovi\'c-Murphy-Starkston-Wu launched an algorithm for constructing Weinstein handlebodies for complements of some toric divisors in toric surfaces from their moment polytope data. We refer \cite{7people.20} and the references therein. 

\subsubsection{Symplectic structure of Weighted Blow-up of Symplectic Orbifolds}
Symplectic structure of blow-up of $\C^{n}$ at the origin is constructed as removing the ball $\{z\in \C^{n}:||z||^{2}<\epsilon \}$ and collapsing the fibers of the Hopf fibration of $\{z\in \C^{n}:||z||^{2}=\epsilon \}$ by McDuff in \cite{mcduff91}. The symplectic cut is a construction of a symplectic structure on the reduced spaces of symplectic manifolds with Hamiltonian circle action, which is a generalization of the symplectic blow-up, by Lerman \cite{lerman95}. Godinho explained symplectic structure of weighted blow-up \cite{godinho99}, \cite{godinho01}.

\subsection{Symplectic Cohomology} (\cite{ritter10}, \cite{mclean12}, \cite{mclean-ritter18})
Hamiltonian Floer cohomology of a symplectic manifold $M$ with a Hamiltonian function on it is a Morse cohomology on the space of loops on $M$. The corresponding Morse function is the action functional on the loop space, which is defined by integration of symplectic form over a surface bounding the loop and evaluation of the Hamiltonian function on the loop. The critical loops are the perioid orbits for Hamiltonian vector fields and the gradient flow trajectories are the Floer cylinders which satisfy elliptic PDE, called the Floer equation, connecting two critical loops. Symplectic cohomology is the direct limit of the Hamiltonian Floer cohomology under the system of admissible Hamiltonian functions on it by the Viterbo transfer maps. Exactness of Liouville domain may not be preserved under the symplectic blow-ups or resolution. For non-exact convex symplectic manifolds with weak monotonicity assumptions, the action functional is multivalued but the differential map is well-defined and the maximum principle still holds, therefore symplectic cohomology is well-defined. 

\subsubsection{Symplectic Cohomology of Liouville Domains} \label{def2}
Let $M$ be a Liouville domain, $\widehat{M}$ be its completion. A family of Hamiltonians $H:S^{1}\times \widehat{M}\to \R$ is called to be admissible if $H(t,x)=\lambda r_{M}(x)$ near infinity, where $r_{M}$ is the cylindrical coordinate of $\widehat{M}$ and $\lambda$ is some positive constant. We choose positive constant $\lambda$ away from the set of lengths of Reeb orbits of $\partial M$, $\theta_{M}$ (The Reeb vector field $R$ on $\partial M$ is defined by $\iota_{R}d\theta_{M}=0$ and $\theta_{M}(R)=1$). Let $J_{t}$ be an $S^{1}$ family $\omega$-compatible almost complex structure on $\widehat{M}$, (i.e., family of metrics, $g_{t}(u,v)=\omega(u,J_{t}v)$), which is convex (or, of contact type on the collar; $J_{t}^{*}\theta=e^{r}dr \iff J_{t}\partial_{r}=R$). For each $t\in S^{1}$, we define the Hamiltonian vector field $X_{H_{t}}$ by $\omega(X_{H_{t}},\cdot)=-dH_{t}$, where $\omega$ is the symplectic form on $M$. A $1$-periodic Hamiltonian orbit is a map $x:S^{1}\to \widehat{M}$ with $\dot{x}(t)=X_{H_{t}}(x(t))$. We say a $1$-periodic Hamiltomian orbit $x$ is non-degenerate if the linearized return map of $x$ has no eigenvalue equal to 1. By [Lemma $2.2$, \cite{mclean12}], near small neighborhood $U$ of $1$-periodic Hamiltonian orbits, we can perturb $H$ by a $C^{\infty}$ small to $\tilde{H}$ so that all the $1$-periodic Hamiltonian orbits are non-degenerate and $\tilde{H}=H$ outside $U$, (i.e., $\tilde{H}$ is still admissible). Choose a trivialization of the canonical bundle of $M$ and then get a canonical trivialization of $TM|_{x}$. Under the trivialization along the orbit, the linearlization of Hamiltonian flow become a smooth path of symplectic matrices. We embed such a path of symplectic matrices to a path of Lagrangians in the product symplectic vector space. We can define the Conley-Zehnder index of the $1$-periodic Hamiltonian orbit $x$, denote $\mu_{CZ}(x)$, by the Maslov index (roughly speaking, the winding number) of the path of Lagrangians with respect to the diagonal Lagrangian in the product symplectic vector space \cite{robbin.salamon92}. Then we grade the Floer cohomology by the index $\mu(x)=\text{dim}_{\R}M-\mu _{CZ}(x)$. The choice of index is based on the fact that it agrees with the Morse index of the critical points $x$ when $H$ is a $C^{2}$-small Morse Hamiltonian and that the degree of the identity element and of pants product is zero, therefore $SH^{*}(M)$ becomes a graded ring. For a $1$-periodic orbit, we define the action functional by,
$$\mathcal{A}_{H}(x):=-\int_{S^{1}} x^{*}\theta+\int_{0}^{1}H(x(t))dt=-\int_{\Sigma^{2}}u^{*}d\theta++\int_{0}^{1}H(x(t))dt.$$
Let $\mathcal{L}\widehat{M}=C^{\infty}(S^{1},\widehat{M})$ be the space of free loops in $\widehat{M}$. The differential of $\mathcal{A}_{H}$ at $x\in \mathcal{L}\widehat{M}$ in the direction $\eta\in T_{x}\mathcal{L}\widehat{M}=C^{\infty}(S^{1},x^{*}T\widehat{M})$ is,
$$d\mathcal{A}_{H}\cdot \eta=-\int_{0}^{1}\omega(\eta, \dot{x}-X_{H}).$$
So, the critical points of the action functional $\mathcal{A}_{H}$ are the $1$-periodic Hamiltonian orbits. By,
$$\int_{0}^{1}g_{t}(\eta, (\nabla \mathcal{A}_{H})_{x})=(d\mathcal{A})_{x}(\eta)=-\int_{0}^{1}\omega(\eta, \dot{x}-X_{H_{t}}),$$
The gradient vector of $\mathcal{A}_{H}$ with respect to $L^{2}$-metric, $\int_{0}^{1}g_{t}(\cdot,\cdot)dt$, is $(\nabla^{g_{t}} \mathcal{A}_{H})_{x}=J_{t}(\dot{x}-X_{H})$. The negative gradient flow is the map $u:\R\times S^{1}\to M$ satisfying, 
$$\partial_{s}u=-\nabla^{g_{t}} \mathcal{A}_{H_{t}}(u)\iff \partial_{s}u+J_{t}(\partial_{t}u-X_{H_{t}})=0, \text{ for }(s,t)\in \R\times S^{1}.$$

The action functional $\mathcal{A}_{H_{t}}(u(s,t))$ decreases in $s$ along the negative gradient flow, called Floer cylinder, since $\partial_{s}(\mathcal{A}_{H_{t}}(u(s,t)))=d\mathcal{A}_{H_{t}}\cdot\partial_{s}u=-\int_{0}^{1}\omega(\partial_{s}u, \partial_{t}u-X_{H_{t}})dt=-\int_{0}^{1}|\partial_{s}u|_{g_{t}}^{2}dt<0.$

Let us define Hamiltonian Floer cohomology as a Morse cohomology. Choose a coefficient field $\K$. Define the Floer chain complex for a pair $(H, J)$ of an admissible Hamiltonian $H\in C^{\infty}(\widehat{M},\R)$ and convex almost complex structure $J$, 
\footnote{
\begin{itemize}
    \item An almost complex structure $J$ on a Liouville domain $(M,\theta_{M})$ is called convex if there is a function $f:M\to \R$ such that $\partial M$ is a regular level set of $f$, $f$ has its maximum on $\partial M$ and $\theta_{M}\circ J=df$. Every Liouville domain $(\underline{M},\theta_{\underline{M}})$ has a convex $d\theta_{\underline{M}}$-compatible almost complex structure \label{lem1} [Lemma $2.1$, \cite{mclean14}].
    \item An almost complex structure on an symplectic manifold $(M,\omega)$, $J:TM\to TM$ with $J^{2}=-id$, is called $\omega$-tame if $\omega(v,Jv)>0$ for $v\in TM$, $\omega$-compatible if $\omega(J\cdot,J\cdot)=\omega(\cdot,\cdot)$ and $\omega(v,Jv)>0$. Both the space of $\omega$-tamed almost complex structures and the space of $\omega$-compatible almost complex structures are contractible. But the tameness is an open condition.
\end{itemize}}

$$CF_{d<}^{k}(H,J):=\bigoplus _{x\in Crit(\mathcal{A}_{H})}\K\cdot x=\bigoplus \{\K x:x\in \mathcal{L}\widehat{M},\dot{x}(t)=X_{H}(x(t)), \text{nondeg}, \mathcal{A}_{H}(x)> d, \mu(x)=k\}.$$

Define the differential map $\partial: CF_{d<}^{k}(H,J)\to CF_{d<}^{k+1}(H,J)$ by counting the negative gradient flow of $\mathcal{A}_{H}$, i.e., the Floer cylinders connecting two Hamiltonian periodic orbits. Denote 
$\mathcal{M}_{\R_{\text{param}}}(x_{-},x_{+}):=\{u:\R\times S^{1}\to M:\partial_{s}u+J_{t}(\partial_{t}u-X_{H_{t}})=0, \lim_{s\to \pm \infty}u(s,t)=x_{\pm}(t)\}$, $\mathcal{M}(x_{-},x_{+}):=\mathcal{M}_{\R_{\text{param}}}(x_{-},x_{+})/\R$.

The energy of a Floer cylinder $u\in \mathcal{M}_{\R_{\text{param}}}(x_{-},x_{+})$ is defined by
$$E(u):=\int |\partial_{s}u|^{2}dsdt=\int \omega(\partial_{s}u,\partial_{t}u-X_{H})dsdt=-\int\partial(\mathcal{A}_{H}(u))ds=\mathcal{A}_{H}(x_{-})-\mathcal{A}_{H}(x_{+}).$$

On the collar near the boundary of $M$, $H=h(e^{r})$, $X_{H}=h'(e^{r})R$ and $A_{H}(x)=-e^{r}h'(e^{r})+h(e^{r})$. By applying the maximum principle on the function $e^{r}\circ u$ on the collar, if the limiting Hamiltonian periodic orbits $x_{\pm}$ sit inside a compact subset $M\cup_{\partial M}([0,R]\times \partial M)$, then all the Floer trajectoris connecting them sit in there. Therefore, $\mathcal{M}(x_{-},x_{+})$ can be compactified to $\overline{\mathcal{M}(x_{-},x_{+})}$ by the broken Floer cylinders. After perturbing $(H_{t},J_{t})$, we get $\overline{\mathcal{M}(x_{-},x_{+})}$ a smooth manifold with dimension $\mu(x_{-})-\mu(x_{+})$. If $\mu(x_{-})-\mu(x_{+})=1$, $\overline{\mathcal{M}(x_{-},x_{+})}$ is a compact zero dimensional manifold, so we can count the Floer cylinders, up to orientation.

$$\partial: CF_{d<}^{k}(H,J)\to CF_{d<}^{k+1}(H,J)\text{ by }\partial (x_{-}):=\sum_{\mu(x_{-})-\mu({x_{+})=1}}\sharp \overline{\mathcal{M}(x_{-},x_{+})} x_{+}.$$

$\partial^{2}=0$ is followed by considering compactifying broken trajectories in ${\mathcal{M}(x_{-},x_{+})}$ for $\mu (x_{-})-\mu (x_{+})=2$.

\begin{remark}
$CF_{d<}^{k}(H,J)$ is independent of the choice of $J$ but its boundary operator does. The cohomology $HF_{d<}^{k}(H,J)$ depends on the choice of $H$ but not on $J$ up to canonical isomorphism.
\end{remark}

Let us explain the action filtration. Increasing $d$ gives a subcomplex because the action functional decreases along Floer trajectories and define the quotient complex,
$$CF_{(c,d]}^{k}(H,J):=CF_{c<}^{k}(H,J)/CF_{d<}^{k}(H,J)$$ 
For $a<a'<b$, $CF_{(a',b]}^{k}(H,J)\hookrightarrow CF_{(a,b]}^{k}(H,J)$ and, for $a<b'<b$, $CF_{(a,b]}^{k}(H,J)\to CF_{(a,b']}^{k}(H,J)$. Therefore, given $-\infty\leq a\leq b\leq c<\infty$, there is a short exact sequence,
$$0\to CF^{k}_{(b,c]}(H,J)\to CF^{k}_{(a,c]}(H,J)\to CF^{k}_{(a,b]}(H,J)\to 0, $$
$$0\to CF^{k}_{(-\epsilon,\epsilon]}(H,J)\to CF^{k}_{(-\infty,\epsilon]}(H,J)\to CF^{k}_{(-\infty,-\epsilon]}(H,J)\to 0.$$ 
Define $HF_{(c,d]}^{k}(H,J)$ as the cohomology of the chain complex $CF_{(c,d]}^{k}(H,J)$.

Let us relate two Hamiltonian Floer cohomologies of $(M,H_{1})$ and $(M,H_{2})$ with two non-degenerate admissible Hamiltonians $H_{1}<H_{2}$. Define the Viterbo's continuatin map,
$$C:CF_{d<}^{k}(H_{1},J_{1})\to CF_{d<}^{k}(H_{d},J_{2})\text{ by }\partial (x_{-}):=\sum_{\mu(x_{-})-\mu({x_{+})=0}}\sharp \mathcal{P}(x_{-},x_{+}) x_{+},$$
where $\mathcal{P}(x_{-},x_{+})$ is the set of solutions to the parametrized Floer equations connecting $H_{1}$ and $H_{2}$ in the following way: Let $K_{s}$ be a smooth non-degenerating family of admissible Hamiltonians equal to $H_{1}$ for $s\ll 0$ and $H_{2}\gg 0$ and $J_{s}$ a smooth family of admissible almost complex structures joining $J_{1}$ and $J_{2}$. Define 
$$\mathcal{P}(x_{-},x_{+}):=\{u: \partial_{s}u+J_{s,t}(u(s,t))\partial_{t}u=\nabla^{g_{t}}K_{s,t}, \lim_{s\to\pm\infty}u(s,\cdot)=x_{\pm}\}.$$

The action functional $\mathcal{A}_{H_{s}}(u(s,\cdot))$ along $u\in \mathcal{P}(x_{-},x_{+})$, 
$$\partial_{s}(\mathcal{A}_{H_{s}}(u(s,\cdot))):=-\int_{0}^{1}|\partial_{s}u|^{2}dt+\int_{0}^{1}(\partial_{s}H_{s})(u)dt,$$
The energy is 
$$E(u)=\int|\partial_{s}u|^{2}ds\wedge dt=\mathcal{A}_{H_{1}}(x_{-})-\mathcal{A}_{H_{+}}(x_{+})+\int (\partial_{s}H_{s})(u)ds\wedge dt.$$
The action functional decreases if $\partial_{s}H_{s}\leq0$ and the energy is bounded if $\partial_{s}H_{s}\leq0$ outside of a compact set in $\widehat{M}$. Therefore, by the maximum principle, we can compactify $\mathcal{P}(x_{-},x_{+})$.

We define the symplectic cohomology of $\widehat{M}$ by the direct limit of $HF_{d<}^{k}(H_{I},J_{I})$ under the Viterbo's continuation maps on admissible Hamiltonians $(H_{I},<)$ such that $H|_{\text{interior }M}<0$ and the slopes at cylindrical ends go to $\infty$ as the cylindrical coordinate $r$ goes to $\infty$.
$$SH^{*}(M):=\lim_{\longrightarrow}HF^{*}(H).$$
The induced long exact sequence commutes with the Viterbo's continuation maps, so, after taking direct limits, we get, for small $\epsilon>0$,
$$\begin{tikzcd} [cramped, sep=small]
\cdots \to SH^{*-1}_{(-\epsilon, \infty)}(M)\arrow[r] & SH^{*-1}_{(\epsilon, \infty)}(M)\arrow[r,"\delta"] & SH^{*}_{(-\epsilon, \epsilon]}(M)\arrow[r]\arrow[d,"\cong"] & SH^{*}_{(-\epsilon, \infty)}(M)\to \cdots\\
                                                                                                &         & H^{*}(M)
\end{tikzcd}$$
Let us denote it by, 
$$\begin{tikzcd} [cramped, sep=small]
\cdots \to SH^{*-1}(M)\arrow[r] & SH^{*-1}_{+}(M)\arrow[r,"\delta"] & H^{*}(M)\arrow[r,"\iota"] & SH^{*}(M)\to \cdots.
\end{tikzcd}
$$
where, $\delta: SH^{*}_{+}(M)\to H^{*}(M)$ is the connecting map.

\subsubsection{Symplectic Cohomology of Symplectic Convex Lefschetz Fibrations}
Inspired by Oancea's K\"unneth formula for symplectic cohomology \cite{oancea06}, McLean defined a Lefschetz-admissible Hamiltonian on a symplectic convex Lefschetz fibration $M$ [Definition 2.21, \cite{mclean09}] and showed that Lefschetz-admissible symplectic cohomology is isomorphic to the (original) symplectic cohomology [Theorem 2.2, 2.4, \cite{mclean09}] and also isomorphic to Monodromy-Floer cohomology of the fiber [Theorem 1.2, 1.3, \cite{mclean.monodromy12}]. The chain complexes of Lefschetz-admissible Hamiltonian Floer cohomology is generated by 
\begin{enumerate}
\item Critical points of Morse function on $M$, and
\item Two copies of fixed points of iterates of the monodromy map around a large circle on the base. 
\end{enumerate}
One of main results in \cite{mclean09} is isomorphism of symplectic cohomology of convex Lefschetz fibration and that of certain subfibration.

\begin{lemma} [Theorem 2.25, \cite{mclean09}] \label{lefSHisom}
Let $M'$ be a convex symplectic Lefschetz fibration with fiber $F'$. Suppose that $M$ is a subfibration of $M$ with fiber $F$ over the same base satisfying the following.
\begin{enumerate}
    \item The support of all the monodromy maps of $M'$ are contained in the interior of $M$.
    \item For any holomorphic curves $u$ in $F'$ with boundary $\partial u\subset F$, $u\subset F$.
\end{enumerate}
Then $SH^{*}_{lef}(M')\cong SH^{*}_{lef}(M)$. Therefore, $SH^{*}(M')\cong SH^{*}(M)$.
\end{lemma}

\subsubsection{Symplectic Cohomology of Non-Exact Convex Symplectic manifolds}
Many K\"ahler manifolds including the blow-up of $\C P^{1}$ at a point, and the crepant resolutions of simple isolated singularities, have non-exact symplectic structure. When open non-exact symplectic manifolds have convex contact boundary condition, symplectic cohomology still can be defined by Ritter \cite{ritter10}. Admissible Hamiltonians are radial at infinity and almost-complex structure are of contact type on the cylindrical ends. Using one-to-one correspondence between Hamiltonian orbits and Reeb orbits, we have the maximum principle on the radial coordinate of the image of Floer cylinder so that the moduli space of Floer cylinders with limiting orbits can be compactified. Weak-monotonicity condition\footnote{$M$ is a $2n$-dimensional symplectic manifold such that for every $A \in \pi_{2}(M)$, $3-n\leq c_{1}(A)<0$ implies $\omega(A)\leq 0$. Equivalently, one of the following holds 
\begin{itemize}
    \item $\omega=\lambda\cdot c_{1}(A)$ for every $A\in\pi_{2}(M)$ for some $\lambda \geq 0$.
    \item $c_{1}(A)=0$ for every $A\in \pi_{2}(M)$.
    \item The minimal Chern number $N\geq 0$ defined by $c_{1}(\pi_{2}(M))=N\Z$ is greater than or equal to $n-2$.
\end{itemize}} is needed to get rid of bubbling when we compactify the moduli spaces of Floer trajectories. Even though the action functional is multi-valued, the cut-offed radial filtration works and gives long exact sequence for non-exact convex symplectic manifolds [Theorem 6.2, \cite{mclean-ritter18}].

\begin{remark} Abouzaid-Blumberg constructed Floer homology with coefficients in Morava's K-theories \cite{abouzaid.blumberg.morava21}. Using it, they proved the Arnold conjecture in characteristic $p$; the rank of the cohomology of a closed symplectic manifold with coefficients in a field of characteristic $p$ is bounded above by the number of non-degenerate Hamiltonian periodic orbits. Extending our symplectic criteria of uniruledness in the setting of Morava K-theory seems interesting, see \cite{balwe.hogadi.sawant21} as well. The relation between fixed points of the Frobenius map and the monodromy Floer cohomology of projective varieties would be also interesting.
\end{remark}

\subsection{Morse-Bott Symplectic Cohomology}
Morse-Bott approach to the Floer theory was developed by Fukaya \cite{fukaya96}, Frauenfelder \cite{frauenfelder03}, Bourgeois-Oancea \cite{bourgeois.oancea09}, and Diogo-Lisi\cite{diogolisi18m}. The idea relating Floer trajectories to $J$-holomorphic curves followed by Morse trajectories was developed by Piunikhin-Salamon-Schwarz \cite{pss96} and Oh-Zhu \cite{oh.zhu11}, \cite{oh.zhu12}. Let us review Bourgeois-Oancea's correspondence theorem between symplectic cohomology defined in the previous section and Morse-Bott symplectic cohomology where the differential maps are defined by counting cascades. 

Let $H:M\to \R$ be a time-independent Hamiltonian on a symplectic manifold $(M,\omega)$ satisfying that $H|_{\underline{M}}$ is $C^{2}$-small and, for the cylindrical coordinate $r$ in the cylindrical end of $M$, $H(\cdot,r)=h(r)=\lambda\cdot e^{r}+b$, for $\lambda, b\in \mathbb{R}$, $\lambda \notin \text{Spec}(X_{H})$, $t>t_{0}$, otherwise $h''-h'>0$. The set of Hamiltonian $1$-periodic orbits are constant orbits, corresponding to critical points of $H|_{\text{int}(\underline{M})}$ and non-isolated, Morse-Bott nondegenerate families $C$ of orbits. We can perturb time-independent Hamiltonian by perfect Morse functions $f_{C}:C\to \R$ on each family $C$ of orbits so that $C$ perturbed into nondegenerate orbits corresponding to critial points of $f_{C}$. Let us construct such a time-dependent Hamiltonian $H_{\delta}: S^{1}\times M \to \R$. For small $\delta>0$, define $K_{\delta}:S^{1}\times M \to \R$ by $K(\theta,\cdot):=H(\cdot)+\delta\cdot\sum_{C}\rho_{C}\cdot f_{C}$, where $\{\rho_{C}\}$ is a smooth cutoff function supported in a small neighborhood of $C$. Then $K_{\delta}$ is Hamiltonian whose non-constant Morse-Bott families $G$ of orbits are of dimension 1. Denote $g_{G}:G\to \mathbb{R}$ the perfect Morse function on $G$ with exactly one maximum and one minimum. Then we perturb $K_{\delta}$ using $g$ to get $H_{\delta}$. Choose the maximal positive integer $t_{G}$ so that $G(\theta+1/t_{G})=G(\theta)$ for all $\theta\in S^{1}$. Choose a symplectic trivialization $\psi_{G}:=(\psi_{G,1},\psi_{G,2}):\text{nbhd}(G)\to V\subset S^{1}\times \mathbb{R}^{2n-1}$ of the open neighborhood of $G$ so that $\psi_{G,1}(G(\theta))=t_{G}\cdot \theta$. Define $H_{\delta}:S^{1}\times M\to \mathbb{R}$ by $H_{\delta}(\theta, \cdot):=K+\delta\cdot \sum_{G} \rho_{G}(\psi_{G}(\cdot))\cdot g_{G}(\psi_{G,1}(\cdot)-t_{G}\cdot\theta)$, where $\rho_{G}:S^{1}\times \mathbb{R}^{2n-1}\to [0,1]$ is a smooth cutoff function supported on a small neighborhood of $S^{1}\times \{0\}$ inside $S^{1}\times \mathbb{R}^{2n-1}$ and $\rho|_{S^{1}\times \{0\}}\equiv 1$ [Equation (25), \cite{bourgeois.oancea09}].
Let us define some moduli spaces as follows. Let $J$ be a generic time-dependent almost complex structure on $W$. Given $p_{-}\in \text{Crit}(f_{C^{-}})$, $p_{+}\in \text{Crit}(f_{C^{+}})$, $x_{-},x_{+}\in \text{Crit}(H)$, $A\in H_{2}(M;\mathbb{Z})$, define 
\begin{itemize}
    \item $\mathcal{M}^{A}(p_{-},p_{+},;H_{\delta},J)$: the moduli space of Floer trajectories for the pair $(H_{\delta}, J)$ with negative asymptote $p$ and positive asymptote $q$ having the same homology class with $A$ modulo reparmetrization in $s$.
    \item $\mathcal{M}^{A}(x_{-},C_{+},;H,J)$: the moduli space of Floer trajectories for the pair $(H, J)$ with negative asymptote $x_{-}$ and positive asymptote in $C_{+}$, having the same homology class with $A$ modulo reparmetrization in $s$. 
    \item $\mathcal{M}^{A}(C_{-}, C_{+};H,J)$: the moduli space of Floer trajectories for the pair $(H,J)$ with negative asymptote in $C_{-}$ and positive asymptote in $C_{+}$, having the same homology class with $A$ modulo reparametrization.
    \item Define evaluation maps $\overline{ev}:\mathcal{M}^{A}(C_{-},C_{+};H,J)\to C_{-}$ by $\overline{ev}([u]):=\lim_{s\to -\infty}u(s,\cdot)$,\\ $\underline{ev}:\mathcal{M}^{A}(C_{-},C_{+};H,J)\to C_{+}$ by $\underline{ev}([u]):=\lim_{s\to \infty}u(s,\cdot)$, and 
    let $\phi_{s}^{f_{C}}$ denote the gradient flow of $f_{C}$. 
    \item $\mathcal{M}^{A}_{m}(p_{-},p_{+};H,\{f_{x}\},J)=W^{u}(p_{-})\times _{\overline{ev}}(\mathcal{M}^{A_{1}}(C_{-},C_{1})\times \R^{+})_{\phi_{s}^{f_{C_{1}}}\circ \underline{ev}}\times  _{\overline{ev}} (\mathcal{M}^{A_{2}}(C_{1},C_{2})\times \R^{+})\times \cdots _{\phi_{s}^{f_{C_{m-1}}}}\circ \underline{ev}\times _{\overline{ev}}  \mathcal{M}^{A_{m}}(C_{m-1},C_{+})_{\underline{ev}}\times W^{s}(p_{+})$: the moduli space of Floer trajectories for the pair $(H,J)$ with intermediate gradient fragments with $\sum_{i}A_{i}=A$.
    \item $\mathcal{M}^{A}(p_{-},p_{+};H,\{f_{x}\},J):=\cup_{m\geq 0}\mathcal{M}^{A}_{m}(p_{-},p_{+};H,\{f_{x}\},J)$.
    \item $\mathcal{M}^{A}_{]0, \delta_{0}[}(p_{-},q_{+};H,\{f_{\gamma}\},J):=\bigcup_{0<\delta<\delta_{0}}\{\delta\}\times \mathcal{M}^{A}(p_{-}, q_{+};H_{\delta},J)$, where $q_{+}$ is either $p_{+}$ or $x_{+}$.
    \item A sequence $[u]\in \mathcal{M}^{A}(x^{-},x^{+};H,J)$ is called to converge to the broken trajectories $([u_{k}],[u_{k-1}]$, $\cdots$, $[u_{1}])$, $[u_{i}]\in \mathcal{M}^{A_{i}}(x_{i}^{-},x_{i}^{+};H,J)$, $\sum_{i}A_{i}=A$, if there exist sequences $s_{i}\in \R, 1\geq i\geq k$, such that $s_{i}\cdot u(s,\cdot)=u(s+s_{i},\cdot)$ converges uniformly on compact sets to $u_{i}$, \cite{bourgeois.oancea09}.
\end{itemize}

\begin{lemma} [Theorem 3.7, \cite{bourgeois.oancea09}] \label{correspondence}
\begin{enumerate}
    \item Any sequence $[v_{\delta_{n}}]\in \mathcal{M}(x^{-}_{p},x^{+}_{p};H_{\delta_{n}},J)$, $\delta_{n}\to 0$, converges to an element $[\boldsymbol{\nu}]$ of $\overline{\mathcal{M}(p_{-},q_{+};H,\{f_{x}\},J)}$.
    \item Any element of $\overline{\mathcal{M}(p_{-},q_{+};H,\{f_{x}\},J)}$ can be obtained as such a limit in a unique way. Moreover, if $dim \mathcal{M}(p_{-},q_{+};H,\{f_{x}\},J)=0$, then the intermediate gradient fragments in $[\boldsymbol{\nu}]$ are non-constant.
    \item There is a bijective correspondence between elements of $\mathcal{M}(p_{-},p_{+};H_{\delta_{n}},J)$ and elements of $\overline{\mathcal{M}(p_{-},q_{+};H,\{f_{x}\},J)}$ if the moduli spaces have dimension zero: More precisely quoting,\\
    
    Let $H\in \mathcal{H}'$ be fixed, and let $\alpha:=\lim_{t\to \infty}e^{-t}H(p,t)$ be the maximal slope of $H$. Let $J\in \mathcal{J}_{reg}(H)$, and let $\{f_{x}\}\in \mathcal{F}_{reg}(H,J)$. There exists $\delta_{1}:=\delta_{1}(H,J)\in ]0,\delta_{0}[$ such that for any $p_{-}\in \text{Crit}(f_{C_{-}})$, $p_{+}\in \text{Crit}(f_{C_{+}})$, and $x_{+}\in \text{Crit}(H)$, with $\mu(p_{+})-\mu(p_{-})=1$ or $\mu(x_{+})-\mu(p_{-})=1$, and any $A\in H_{2}(W;\Z)$, the following hold ($q_{+}$ means $p_{+}$ or $x_{+}$): 
    \begin{enumerate}
        \item $J$ is regular for $\mathcal{M}^{A}(p_{-},q_{+};H_{\delta},J)$ for all $\delta\in ]0,\delta_{1}[$;
        \item The moduli space $\mathcal{M}^{A}_{]0,\delta_{1}[}(p_{-},q_{+};H,\{f_{x}\},J)$ is a one-dimensional manifold having a finite number of components that are graphs over $]0,\delta_{1}[$;
        \item There is a bijective correspondence between points in $\mathcal{M}^{A}(p_{-},q_{+};H,\{f_{x}\},J)$ and connected components of $\mathcal{M}^{A}_{]0,\delta_{1}[}(p_{-},q_{+},H,\{f_{x}\},J)$.
    \end{enumerate}
\end{enumerate}
\end{lemma}

For simplicity, assume that $c_{1}(TW)=0$ and refer \cite{bourgeois.oancea09} for definitions in more general setting. Define the symplectic chain group $CF^{*}(H)$ by the free $\mathbb{Z}$-module generated by Hamiltonian 1-periodic orbits $x$, with grading $|x|=\mu_{CZ}$. Define the differential map $\partial: CF^{*}(H)\to CF^{*+1}(H)$ by 
$$\partial x^{-}:=\sum_{\mu_{CZ}(x^{+})-\mu_{CZ}(x^{-})=1,}\sum_{A\in H_{2}(M;\mathbb{Z}) }\sharp \mathcal{M}^{A}(x^{-},x^{+};H,J)\cdot x^{+}.$$

Since $\partial^{2}=0$, define $HF^{*}(H,J):=H^{*}(CF^{*}(H),\partial)$ and $SH^{*}(M):=\varinjlim_{H}HF^{*}(H)$.
Define the Morse-Bott Hamiltonian Floer chain groups by
$$CB^{*}(H):=\bigoplus \mathbb{Z}<\text{Crit}(f_{C})>\oplus \mathbb{Z}<\text{Crit}_{\text{nondeg}}(H)>.$$
The grading is defined using the symplectic shear axiom for the Robbin-Salamon index \cite{bourgeois.oancea09}. There are several conventions to define symplectic (co)homology and we compare them in Appendix \ref{several conventions} below. Define the Morse-Bott Hamiltonian Floer differential $\partial: CB^{*}(H)\to CB^{*+1}(H)$ by,
\begin{itemize}
    \item $\partial x_{-}:=\sum_{x_{+}\in\text{Crit}(H), |x_{+}|-|x_{-}|=1} \sharp \mathcal{M}^{0}(x_{-},x_{+};H,\{f_{\gamma}\},J)\cdot x_{+}$, for $x_{+}\in \text{Crit}(H)$.
    \item $\partial p_{-}:=\sum_{p_{+}\in \text{Crit}(f_{C_{+}}), |p_{+}|-|p_{-}|=1}\sum_{A\in H_{2}(M;\mathbb{Z})}\sharp \mathcal{M}^{A}(p_{-},p_{+};H,\{f_{\lambda}\},J)\cdot p_{+}$,\\
    $+\sum_{x_{+}\in\text{Crit}(H), |\gamma_{x_{+}}|-|p_{-}|=1}\sum_{A\in H_{2}(M;\mathbb{Z})}\sharp \mathcal{M}^{A}(p_{-},x_{+};H,\{f_{\gamma}\},J)\cdot x_{+}$, for $p_{-}\in \text{Crit}(f_{C_{-}})$.
\end{itemize}
where, all the counting is considered with a sign of the corresponding connected component.
By definition, $CF^{*}(H_{\delta})\cong CB^{*}(H), \delta\in ]0,\delta_{1}[$ as free $\mathbb{Z}$-modules.
And, by the correspondence theorem \ref{correspondence} [Theorem 3.7, \cite{bourgeois.oancea09}], the differentials counts the same. Therefore, $H^{*}(BC^{*}(H),\partial)=SH^{*}(H_{\delta},J)$. In this paper, we take Hamiltonian to be constantly zero on the interior of the Liouville domain $\underline{M}$, so that the Morse-Bott family of orbits is codimension $0$, Bourgeois-Oancea's argument on Morse-Bott moduli spaces still applies. We note that the part of Floer cylinder which degenerates to Morse trajectory are on the descending manifold of the critical point under the Morse flow, the the degenerating locus is under the boundary of $M$. 

\subsection{Twisted/Bulk Deformed Symplectic Cohomology}
We can extend the coefficient ring of the symplectic cohomology. Let us recall the definitions by Ritter [Section 4, \cite{ritter10}]. Let $ev:\mathcal{L}M\times S^{1}\to \mathcal{L}M$ be the evaluation map. Define $\tau:=\pi\circ ev^{*}:H^{2}(M;\R)\rightarrow H^{2}(\mathcal{L}M\times S^{1};\R)\to H^{1}(\mathcal{L}M;\R)$. For $\xi\in H^{2}(M;\R)$, a smooth path $u$ in $\mathcal{L}M$, $\tau\xi(u)=\int \xi(\partial_{s}u,\partial_{t}u)ds\wedge dt$. Let $\beta=\tau(\xi)$ be a singular cocycle representing a class in $H^{1}(\mathcal{L}M;\R)$ for $\xi\in H^{2}(M;\R)$. The twisted symplectic chain complex for $(M,H)$ with twisted coefficients in $\Lambda$ is the $\Lambda$-module generated by the periodic orbits of $X_{H}$, the differential $\partial$ is twisted by $t^{-\int_{[u]}\beta}=t^{-\int \eta(\partial_{s}u,\partial_{t}u)ds\wedge dt}$, where $u$ is a Floer cylinder connecting two generators, 
\begin{align*}
    \partial x_{-}:&=\sum_{u\in \mathcal{M}(x_{-},x_{+};H,J)}\epsilon (u)t^{\mathcal{E}_{H}(u)-\int_{[u]}\beta}x_{+}\\
    &=\sum_{u\in \mathcal{M}(x_{-},x_{+};H,J)}\epsilon(u)t^{\int_{\R\times\R/\Z}u^{*}\omega}t^{\int_{0}^{1}(H(x_{+}(t))-H(x_{-}(t)))dt}\cdot t^{-\int \eta(\partial_{s}u,\partial_{t}u)ds\wedge dt}x_{+},
\end{align*}

All continuation maps are also twisted by $t^{-\int_{[u]}\beta}=t^{-\int \eta(\partial_{s}u,\partial_{t}u)ds\wedge dt}$. Denote $\beta=\tau(\xi)$-twisted symplectic cohomology of $M$ by $SH^{*}(M;\underline{\Lambda}_{\beta})$.

\subsection{Computing Symplectic Cohomology}

Symplectic cohomology of an affine variety $X\setminus D$, where $X$ is a projective variety and $D$ is simple normal crossing divisors, can be computed in several ways: For admissible pairs $(X,D)$, spectral sequences which converges to symplectic cohomology was developed by Seidel \cite{seidelbiased}, McLean \cite{mcleancomputation} and Ganatra-Pomerleano \cite{ganatrapomerleano}. Ritter computed symplectic cohomology of some non-exact monotone convex symplectic manifold with Hamiltonian circle action that generate Reeb flow on the boundary using Seidel representation \cite{ritter14}, \cite{ritter16}. Using the Legendrian surgery exact sequence by Bourgeois-Ekholm-Eliashberg \cite{BEE12}, symplectic cohomology of Weinstein manifold, constructed from Legendrian surgery on a subcritical Weinstein manifold, is isomorphic to Legendrian contact homology. For any $4$-dimensional Weinstein manifold, symplectic cohomology can be computed combinatorially by Ekholm-Ng \cite{ekholm.ng15}. Mayer-Vietoris Sequence is developed by Cieliebak-Oancea for Liouville cobordisms and by Varolgunes for compact submanifolds of a closed symplectic manifold. Symplectic cohomology is also known to be isomorphic to Hochschild cohomology of wrapped Fukaya category. Various mirror symmetries predict that (wrapped) Fukaya categories are equivalent to derived category of coherent sheaves, possibly with superpotentials(the Landau-Ginzburg model). The Gross-Siebert program explains how to construct the mirrors by tropicalization. Pascaleff computed symplectic cohomology of log Calabi-Yau surfaces using tropicalization \cite{pascaleff}. Lekili-Ueda computed symplectic cohomlogy of Milnor fibers of ADE singularities by computing algebraic side of homological mirror symmetry \cite{lekili.ueda.20}.

\begin{example}[\cite{mcleancomputation}, See also Proposition 11.1, \cite{diogilisi18s}] Let $X$ be a projective variety with simple normal crossing divisor $D:=\{D_{j}\}_{j\in S}$. Choose a meromorphic section $\kappa$ of the canonical bundle of $X$ which is non-vanishing along $X\setminus D$. Define the discrepancy $a_{i}$ of $D_{i}$ as $a_{i}:=\text{ord}_{D_{i}}\kappa^{-1}(0)-\text{ord}_{D_{i}}\kappa^{-1}(\infty)$. Choose an ample line bundle $L$ on $X$ and a holomorphic section $s$ of $L$ which is non-vanishing along $X\setminus D$. Define a winding number $w_{i}$ of $D_{i}$ as $w_{i}:=-\text{ord}_{D_{i}}s^{-1}(0)$. Choose $ND_{I}$ a small tubular neighborhood of $D_{I}$ for $I\subset S$ so that $\overset{\circ}{N}D_{I}:=ND_{I}\setminus \cup_{j\in S}D_{j}$ is a bundle over $D_{I}\setminus \cup_{j\in S-I}D_{j}$ with fiber of a product of punctured disks. Using the Morse-Bott spectral sequence, a spectral sequence $(E_{r}^{p,q},d_{r}^{p,q}:=E_{r}^{p,q}\to E_{r}^{p+r,q-r+1})$ with $E^{1}$-page,\\
$$E_{1}^{p,q}:=\bigoplus_{(k_{i}\in \N^{S}:\sum_{i}k_{i}w_{i}=-p)} H^{p+q-2(\sum_{i}k_{i}(a_{i}+1))}(\overset{\circ}{N}D_{I_{(k_{i})}}), I_{(k_{i})}:=\{i\in S: k_{i}\neq 0\},$$
converges to $SH^{*}(X\setminus D)$.

Let $C_{(1,k)}$ be the complement of bidegree $(1,k)$-hypersufaces $D_{(1,k)}$ in $\P^{1}\times \P^{1}$. $C_{(1,1)}=\P^{1}\times \P^{1}\setminus D_{(1,1)}=T^{*}\P^{1}$. $\overset{\circ}{N}D_{(1,1)}=\R P^{3}$, considering the Hopf fibraion over $\P^{1}$ in $\P^{1}\times \P^{1}$. Using $a=\frac{-2}{1}+\frac{-2}{1}=-4$, $w=-2$, 
$$H^{*}(\R P^{3}) := 
\begin{cases}
\Z,&  \text{for } *=3,\\
\Z/2,&\text{for } *=2,\\
0  ,& \text{for } *=1,\\
\Z ,& \text{for } *=0,
\end{cases} \quad\quad H^{*}(C_{(1,k)})=
\begin{cases}
\Z  ,& \text{for } *=2,\\
\Z/2,& \text{for } *=1,\\
\Z  ,& \text{for } *=0,
\end{cases}$$

we get the following $E^{1}$-page, $E^{\infty}$-page,
\begin{center} \begin{tikzpicture}
  \matrix (m) [matrix of math nodes,
    nodes in empty cells,nodes={minimum width=3.5ex,
    minimum height=3.5ex,outer sep=-9pt},
    column sep=1ex,row sep=1ex]{
          8     &  \cdot& \cdot& \cdot&  \Z   \quad \quad&  8 & \cdot & \cdot& \cdot & \Z \\
          7     &  \cdot& \cdot& \cdot& \cdot \quad \quad&  7 & \cdot & \cdot& \cdot & \cdot \\
          6     &  \cdot& \cdot&  \Z  &  \Z   \quad \quad&  6 & \cdot & \cdot& \cdot & \cdot \\
          5     &  \cdot& \cdot& \cdot&  \Z   \quad \quad&  5 & \cdot & \cdot& \cdot & \cdot \\ 
          4     & \Z    & \Z   &  \Z  & \cdot \quad \quad&  4 & \cdot & \cdot& \Z    & \cdot \\
          3     & \Z    & \cdot&  \Z  & \cdot \quad \quad&  3 & \Z    & \cdot& \Z    & \cdot \\
          2     & \Z/2  & \Z/2 & \cdot& \cdot \quad \quad&  2 & \cdot & \cdot& \cdot & \cdot \\
          1     &  \cdot& \Z   & \cdot& \cdot \quad \quad&  1 & \cdot & \Z   & \cdot & \cdot \\
          0     & \Z    & \cdot& \cdot& \cdot \quad \quad&  0 & \Z    & \cdot& \cdot & \cdot \\
    \quad\strut &  0    & 1    &  2   &  3    \quad \quad&    & 0     & 1    & 2     & 3 \strut \\};
\end{tikzpicture} \end{center}
Therefore, we get $\ell(C_{(1,1)})=2$. However, $C_{(1,k)}$ is $\C$-uniruled for all $k\geq 1$.
\end{example}

\subsubsection{Relative Symplectic Cohomology of Compact Subsets of a Completed Liouville Domain and a Mayer-Vietoris Sequence}

We adapt two definitions of symplectic cohomology of a pair $(K\subset X)$ of a projective variety $X$ and its compact subset $K$ in [Definition 2.64, Definition 3.3 in \cite{mclean18}, \cite{varolgunes19}] to a pair $(K\subset M)$ of a completion $M$ of a Liouville domain $\underline{M}$ and a compact subset $K$ of $M$ following \cite{cieliebakhandle}, \cite{mclean18}, \cite{varolgunes19}. Varolgunes' Mayer-Vietoris property on relative symplectic cohomology still works. There are some aspects that our setting is simpler because we are given a Liouville domain with contact cylinder(exact symplectic form) so we are not worried about the index-bounded contact cylinder for a pair $(K\subset X)$ of a projective variety $X$ and its compact subset $K$. And the symplectic form is exact compatible with the contact boundary.

Denote $\text{Cube}^{n}:=\{(x_{1},\cdots,x_{n})|x_{j}\in [0,1]\}$, the standard unit cube. For $0\leq k\leq n$, a $k$-dimensional face of $\text{Cube}^{n}$ is a subset of $\text{Cube}^{n}$ having that $n-k$ of the coordinates are either $0$ or $1$. We call $0$-dimensional faces, vertices. A vertex of a face is called the initial(resp. terminal) vertex if it has the maximum number of $0$'s (resp. $1$'s) in its coordinates and denote $\nu_{s(F)}$ (resp. $\nu_{t(F)}$). Two faces $F'$ and $F''$ are adjacent if $\nu_{s(F'')}=\nu_{t(F')}$ and denote $F''<F'$. Let us summarize some notations as follows.
\begin{itemize}
    \item The faces of $Cube^{n}$ are in one-to-one correspondence with the set of $n$-tuples of elements of a set $\{0,1,-\}$. Denote $\mu(F)$ the $n$-tuple corresponding to a face $F$. 
    \item Given two adjacent faces $F'>F''$, choose the smallest face $F$ that contains both $F'$ and $F''$. Denote $v(F',F)$ the subtuple of $\mu(\nu_{t(F')})-\mu(\nu_{s(F')})$.
    \item For tuples $w,v$ of elements of a set $A$, $\#(w,v)$ is defined to be the number of subtuples of $v$ that are equal to $w$.
    \item $n(F):=\#(0-,\mu (F))+\#(0,\mu(F))$.
\end{itemize}

\begin{defn}[Definition 2.1.1,\cite{varolgunes19}]
An $n$-cube of chain complexes over a commutative ring $R$ consists of following data,
\begin{enumerate}
    \item To each vertex $\nu$ of $\text{Cube}^{n}$, we assign a $\Z/2$-graded $R$-module $C^{\nu}$.
    \item To each $k$-dimensional face $F$, we assign maps $f_{F}:C^{\nu_{s(F)}}\to C^{\nu_{t(F)}}$ from its initial vertex to its terminal vertex, of degree $dim(F)+1$ modulo $2$.
    \item For any boundary $F''<F'$ of a face $F$, $\sum_{F\prime\prime<F', \text{boundary }F}(-1)^{*_{F',F}}f_{F''}f_{F'}=0$, where $*_{F',F}=n(1,v(F',F))+n(01,v(F',F))$. 
\end{enumerate}
\end{defn} 
When the signs $(-1)^{*_{F',F}}$ are all positive, we call it $n$-cubes with positive signs. There is a canonical way to make an $n$-cube with positive signs [Lemma 2.1.5,\cite{varolgunes19}]. The cone of an $n$-cube with positive signs $(\{\mathcal{C}^{\nu}\},\{f_{F}\})$ in direction $i$ is an $(n-1)$-cube with positive signs constructed as:
\begin{itemize}
    \item For each vertex $w$ of Cube$^{n-1}, \mathcal{C}^{w}:=\mathcal{C}^{(w,i,0)}[1]\oplus \mathcal{C}^{(w,i,1)}$, where $(w,i,a)$ denotes the $n$-tuple with $a$ inserted as the $i$th entry to $w$, for example, $((0,2,2),3,1)=(0,2,1,2)$.
    \item For each face $F$, the map $f_{F}:\mathcal{C}^{s(F)}\to \mathcal{C}^{t(F)}$ is given by the matrix 
    $$\begin{pmatrix}
f_{(\mu(F),i,0)} &  0\\
f_{(\mu(F),i,-)} &  f_{(\mu(F),i,1)}
\end{pmatrix}.$$
\end{itemize}
The cone $cone^{i}(\mathcal{C})$ of an $n$-cube $C$ in direction $i$ is the cone construction of $n$-cube with positive signs conjugated with canonical maps to a general $n$-cubes. A map between two $n$-cubes $\mathcal{C}\to \mathcal{C}'$ is a filling of the partially defined $(n+1)$-cube with the $n$-dimensional faces $\{x_{n+1}=0\}=\mathcal{C}$, $\{x_{n+1}=1\}=\mathcal{C'}$. A homotopy of two maps $f_{0}, f_{1}:\mathcal{C}\to \mathcal{C}'$ of $n$-cubes is a filling of the partially defined $(n+2)$-cube where the faces $\{x_{n+1}=0\}=(f_{0}:\mathcal{C}\to \mathcal{C}')$, $\{x_{n+1}=1\}=(f_{1}:\mathcal{C}\to \mathcal{C}')$ and the faces $\{x_{n+2}=0\}$ and $\{x_{n+2}=1\}$ are the identity maps for the given $n-cubes$.
\begin{equation}
    \begin{tikzcd}
  \mathcal{C} \arrow[r,"f_{0}"]  & \mathcal{C}'  \\
  \mathcal{C} \arrow[r,"f_{1}"] & \mathcal{C}', \end{tikzcd} \hspace{40pt}%
  \begin{tikzcd}
  \mathcal{C} \arrow[r,"f_{0}"] \arrow[rd] \arrow[d,"id"] & \mathcal{C}' \arrow[d,"id"] \\
  \mathcal{C} \arrow[r,"f_{1}"] & \mathcal{C}'. \end{tikzcd}
\end{equation}

An $n$-ray is an infinite sequence of $n$-cubes $\{\mathcal{D_{i}}\}_{i\in \N}$ satisfying that $\mathcal{D}_{i}, \mathcal{D}_{i+1}$ glued in the $n$-th direction for all $i$, presented as a sequence of $(n-1)$-cubes \begin{tikzcd}
\mathcal{C}=C_{1}\arrow[r,"f_{1}"] & C_{2}\arrow[r,"f_{2}"]&\cdots
\end{tikzcd} so that $\mathcal{D}_{i}:\mathcal{C}_{i}\to \mathcal{C}_{i+1}$ is a map of $(n-1)$-cubes. Given a $1$-ray \begin{tikzcd}
\mathcal{C}=C_{1}\arrow[r,"f_{1}"] & C_{2}\arrow[r,"f_{2}"]&\cdots
\end{tikzcd}, the (mapping) telescope $\text{tel}(\mathcal{C})$ of $\mathcal{C}$ is the complex $\big(\bigoplus_{i\in \N}C_{i}[1]\oplus C_{i},\delta:=(x[-1]-dx+f_{i}(x),dx)\big)$, where $x[-1]$ denotes the copy of $x$ in $C_{i}$. The telescope $\text{tel}(\mathcal{C})$ of an $n$-ray $\mathcal{C}$ is an $(n-1)$-cube constructed as follows.
\begin{itemize}
    \item At each vertex $w$ of Cube$^{n-1}$ of tel$(\mathcal{C})$, we assign the $R$-module $\bigoplus_{i\in \N}(\text{cone}^{n}(\mathcal{D}_{i}))^{w}=\bigoplus_{i\in \N}\mathcal{C}_{i}^{w}[1]\oplus \mathcal{C}_{i}^{w}$, where $\text{cone}^{n}(\mathcal{D}_{i})$ is the cone in the last direction of the $n$-cube $\mathcal{D}_{i}:\mathcal{C}_{i}\to \mathcal{C}_{i+1}$. 
    \item The induced maps are defined, for each face $F$ of Cube$^{n-1}$,
 $$\delta_{F}(x)=\begin{cases*}
 f_{F}^{i}(x) & \text{if $x\in \mathcal{C}^{s(F)}_{i}$},\\
 f_{F}^{i}(x)+x[-1] & \text{if $x\in \mathcal{C}^{w}_{i}[1]$ and $F$ is of zero-dimension}, \\
 f_{F}^{i}(x) &\text{if $x\in \mathcal{C}^{w}_{i}[1]$ and $F$ is of positive-dimension}.
 \end{cases*}$$
\end{itemize}

An $n$-cube of admissible Hamiltonians is a smooth map $H:\text{Cube}^{n}\to C^{\infty}(M\times S^{1},\R)$ which is locally constant near the vertices and non-degenerate at the vertices. An $n$-cube family of admissible Hamiltonians is called to be monotone if the Hamiltonians are non-decreasing along all of the flow lines of $f$. By the energy inequality, a monotone $n$-cube of admissible Hamiltonians gives an $n$-cube defined over the ring $\Lambda_{\geq 0}$ [Definition 3.2.4, \cite{varolgunes19}]. Any two acceleration data for $K\subset M$ are homotopic because any partially defined Hamiltonian-Floer-homotopic (with monotonicity) $n$-cube has a filling to an Hamiltonian-Floer-homotopic $n$-cube, [Proposition 3.2.18, \cite{varolgunes19}, Lemma 3.5, \cite{cieliebakhandle}]. Here, an $n$-cube is Hamiltonian-Floer-homotopic $n$-cube means that if for any two points $x$ and $y$ in Cube$^{n}$ where Hamiltonians are defined and that there exists a possibly broken negative gradient flow line of $f$ from $x$ to $y$, then $H|_{x}\leq H|_{y}$.

\begin{defn}
Let $(K\subset M)$ be a pair of a completion $M$ of a Liouville domain $\underline{M}$ and a compact subset $K$ of $M$. We can assume that $K\subset \underline{M}$. A Hamiltonian $H_{K\subset M}$ on $M$ is $(K\subset M)$-admissible if 
\begin{enumerate}
    \item $H_{K\subset M}|_{K}<0$ and $C^{2}$-small.
    \item $H_{K\subset M}\equiv c$ outside of a compact set whose interior containing $K$ ($c\in \mathbb{R}_{>0}$).
\end{enumerate}
\end{defn}

\begin{remark}
We refer p. 511-512 around the Figure 12 in \cite{mclean18}.
\end{remark}
Topological energy of $u:\R\times S^{1}\to (M,\omega, H:\R\times S^{1}\times M\to \R)$ is defined by,
$$\mathcal{E}_{H}(u):=\int u^{*}\omega+\int \partial_{s}(H(s,t,u(s,t))).$$

It defines the complex of Hamiltonian Floer cohomology as follows,
$$CF^{*}(H,J):=\oplus_{x\in\mathcal{P}(H)}\mathbb{Z}\cdot x,  \partial x_{-}:=\sum_{\mu(x_{-})-\mu({x_{+})=1}}\sharp \overline{\mathcal{M}(x_{-},x_{+})} x_{+}.$$

For two $(K\subset M)$-admissible Hamiltonians $H, \widetilde{H}$, define $H\leq \widetilde{H}$ if $H(x)\leq \widetilde{H}(x)$ for any $x\in M$. We consider a cofinal family of $(K\subset M)$-admissible Hamiltonians $\{H_{K\subset M, c_{i}},\leq\}_{i\in\N}$ converging to $\overline{H}_{K}$, defined by $\overline{H}_{K}|_{K}\equiv 0$, $\overline{H}_{K}|_{M\setminus K}\equiv \infty$. An acceleration data for $K$ is a cofinal family of $(K\subset M)$-admissible Hamiltonians on a pair with a choice of interpolating Hamiltonians $\{H_{K\subset M, c_{s}}\}_{s\in [i,i+1]}$ for all $i$. Define a filtration $\mathcal{F}_{\bullet}$ on $(CF^{*}(\mathcal{H}_{K\subset M}),\delta)$ by the action, where $\mathcal{F}_{a}(CF^{*}(\mathcal{H}_{K\subset M}),\delta)$ is the subcomplex generated by orbits with the action greater than $a$ ($\partial \mathcal{F}_{a}\subset \mathcal{F}_{a}$). Denote that $\mathcal{F}_{a, b}:=\mathcal{F}_{a}/\mathcal{F}_{b}$. Then for $a_{1}<a_{2}<b_{1}<b_{2}$, $c_{1}<c_{2}$, we have a Hamiltonian-Floer-homotopic $3$-cube, a $3$-cube where Hamiltonians are mononotonous along the arrows of the following diagram with the connecting 1-parameter family of Hamiltonians.

$$\begin{tikzcd}[back line/.style={densely dotted}, row sep=1em, column sep=1em]
& \mathcal{F}_{a_{1},b_{2}}CF^{*}(H_{K\subset M,c_{1}}) \ar{dl} \ar{rr} \ar[back line]{dd} 
  & & \mathcal{F}_{a_{1},b_{2}}CF^{*}(H_{K\subset M,c_{2}}) \ar{dd} \ar{dl} \\
 \mathcal{F}_{a_{1},b_{1}}CF^{*}(H_{K\subset M,c_{1}}) \ar[crossing over]{rr} \ar{dd}
  & &  \mathcal{F}_{a_{1},b_{1}}CF^{*}(H_{K\subset M,c_{2}}) \\
&  \mathcal{F}_{a_{2},b_{2}}CF^{*}(H_{K\subset M,c_{1}}) \ar[back line]{rr} \ar[back line]{dl} 
  & &  \mathcal{F}_{a_{2},b_{2}}CF^{*}(H_{K\subset M,c_{2}}) \ar{dl} \\
 \mathcal{F}_{a_{2},b_{1}}CF^{*}(H_{K\subset M,c_{1}})  \ar{rr} & &  \mathcal{F}_{a_{2},b_{1}}CF^{*}(H_{K\subset M,c_{2}})  \ar[crossing over, leftarrow]{uu}
\end{tikzcd}$$

We also assume that for any two non-degenerate Hamiltonians $f$ and $g$ in the cofinal family, $\overline{\{f<g\}}\cap \overline{\{f>g\}}=\emptyset$ so that max$\{f,g\}$ and min$\{f,g\}$ are smooth [Proposition 4.1.1,\cite{varolgunes19}]. From the cofinal family of $(K\subset M)$-admissible Hamiltonians, we define relative symplectic cohomology of a pair $(K,M)$ as follows.

\begin{defn} Let $M$ be the completion of a Liouville domain $\underline{M}$, and $K$ be a compact subset of $\underline{M}$. Let $a,b\in [-\infty,\infty]$ with $a<b$. 
\begin{align*}
SC^{*}_{a,b}(K\subset M;\mathcal{H}_{K\subset M},\mathcal{J}):&=\varinjlim_{\substack{a'\searrow a}}\circ \varprojlim_{\substack{b'\nearrow b}} \circ \text{tel}_{c}[\mathcal{F}_{a', b'}CF^{\bullet}(\mathcal{H}_{K\subset M, c})],\\
&=\varinjlim_{\substack{a'\searrow a }}\circ \varprojlim_{\substack{b'\nearrow b}} \circ \varinjlim_{c\nearrow \infty}[\mathcal{F}_{a', b'}CF^{\bullet}(\mathcal{H}_{K\subset M, c})],
\end{align*}
where $\varinjlim_{\substack{a}}\circ \varprojlim_{\substack{b}}$ is completion with respect to action filtration, where $\text{tel}_{c}$ denotes telescope of a $3$-ray in the direction of the height $c$ of admissible Hamiltonians, with choices of continuation map data. For a finite number $a\in \R$, taking direct limit $\varinjlim_{\substack{a\leq}}$ does not necessary. 
$$SC^{*}_{a,b}(K\subset M;\mathcal{H}_{K\subset M},\mathcal{J}):=\varprojlim_{\substack{b'\nearrow b}} \circ \varinjlim_{c\nearrow \infty}[\mathcal{F}_{a, b'}CF^{\bullet}(\mathcal{H}_{K\subset M, c})].$$
\end{defn}

By \cite{varolgunes19}, we can choose a homotopy interpolating two cofinal Hamiltonian data. 
\begin{lemma} \cite{varolgunes19}
For any two cofinal family of Hamiltonians $(\mathcal{H}_{K\subset M},\mathcal{J}), (\tilde{\mathcal{H}}_{K\subset M},\widetilde{\mathcal{J}})$, $SC^{*}_{a,b}(K\subset M;\mathcal{H}_{K\subset M},\mathcal{J})$ and $SC^{*}_{a,b}(K\subset M;\tilde{\mathcal{H}}_{K\subset M},\widetilde{\mathcal{J}})$ are quasi-isomorphic. 
\end{lemma} 

Therefore, we have well-defined cohomology,
$$SH^{*}_{a,b}(K\subset M):=H^{*}\Big(\varinjlim_{\substack{a'\searrow a }}\circ \varprojlim_{\substack{b'\nearrow b}} \circ \varinjlim_{c\nearrow \infty}[\mathcal{F}_{a, b'}CF^{\bullet}(\mathcal{H}_{K\subset M, c})]\Big).$$

Similar to the usual symplectic cohomology, we have the following lemma.
\begin{lemma} \label{relativelemma1}
Let $M$ be the completion of a Liouville domain $\underline{M}$. Let $SH^{*}(\underline{M}\subset M)$ denote $SH^{*}_{-\epsilon,\infty}(\underline{M}\subset M)$ and $SH^{*}_{+}(\underline{M}\subset M)$ denote $SH^{*}_{\epsilon,\infty}(\underline{M}\subset M)$. Then
\begin{enumerate}
    \item $SH^{*}_{-\epsilon,\epsilon}(\underline{M}\subset M)\cong H^{*}(M)$.
    \item $\cdots\to SH^{*}_{+}(\underline{M}\subset M)\xrightarrow{\hat{\delta}} H^{*}(M)\to SH^{*}(\underline{M}\subset M)\to SH^{*}_{+}(\underline{M}\subset M)\to \cdots.$
    \item $SH^{*}(\underline{M}\subset M)=\widehat{SH^{*}}(M)$, where $\widehat{SH^{*}}(M)$ is cohomology of action-completed $(a,b)$-truncated symplectic chains, \cite{venkatesh18}.
\end{enumerate}
\end{lemma}

By the connecting map $\widehat{\delta}$ in the long exact sequence \ref{relativelemma1} (2), we can define 
$\widehat{\ell}(M)$ as well. 

\begin{prop} \cite{varolgunes19} \label{mv}
Let $K_{1},K_{2}$ be compact subsets of a completion $M$ of a Liouville domain $\underline{M}$ satisfying that $\partial(K_{1}\cap \underline{M})\cap \partial(K_{2}\cap \underline{M})\neq \emptyset$.
The Mayer-Vietoris sequence for such subsets $K_{1},K_{2}$ of $M$ holds,
$$\lra{\delta} SH^{*}(K_{1}\cup K_{2}\subset M)\to SH^{*}(K_{1}\subset M)\oplus SH^{*}(K_{2}\subset M)\to SH^{*}(K_{1}\cap K_{2}\subset M) \to.$$
\end{prop}

\begin{proof}
Exactly the same idea in \cite{varolgunes19} works. Since $\partial(K_{1}\cap \underline{M})\cap \partial(K_{2}\cap \underline{M})\neq \emptyset$, the slices of any $3$-ray that is compatible with any acceleration datum for $(K_{1}\subset M)$, $(K_{2}\subset M)$, $(K_{1}\cap K_{2}\subset M)$, and $(K_{1}\cup K_{2}\subset M)$, we get an acyclic $2$-cube of the cofinal families $\mathcal{H}_{1}$, $\mathcal{H}_{2}$, $max\{\mathcal{H}_{1},\mathcal{H}_{2}\}$, and $min\{\mathcal{H}_{1},\mathcal{H}_{2}\}$ respectively, as follows. Here, $n$-cube $\mathcal{C}$ is called acyclic if $\text{Cone}^{n-1}\mathcal{C}$ is an acyclic chain complex.  
$$\begin{tikzcd} [cramped, sep=small]
  C_{00} \arrow[r] \arrow[rd] \arrow[d] & C_{10} \arrow[d] \\
  C_{01} \arrow[r] & C_{11}. \end{tikzcd}$$
  If a $2$-cube is acyclic, then the map $C_{00}\to Cone(C_{10}\oplus C_{01}\to C_{11})$ is a quasi-isomorphism. Therefore, it induces an exact sequence 
  $$\cdots \to H^{*}(C_{00})\to H^{*}(C_{10})\oplus H^{*}(C_{01})\to H^{*}(C_{11})\to H^{*+1}(C_{00}) \to \cdots.$$
  Hence, the acyclic $2$-cube
  $$\begin{tikzcd} [cramped, sep=small]
  CF^{*}(min\{\mathcal{H}_{1}, \mathcal{H}_{2}\}) \arrow[r] \arrow[rd] \arrow[d] & CF^{*}(\mathcal{H}_{1}) \arrow[d] \\
  CF^{*}(\mathcal{H}_{2}) \arrow[r] & CF^{*}(max\{\mathcal{H}_{1},\mathcal{H}_{2}\}),  \end{tikzcd}$$
  gives
  $$\begin{tikzcd} 
  \varprojlim_{\substack{b'\nearrow \infty}} \circ \varinjlim_{c\nearrow \infty}[\mathcal{F}_{\epsilon, b'}CF^{*}(min\{\mathcal{H}_{1}, \mathcal{H}_{2}\})] \arrow[r] \arrow[rd] \arrow[d] & \varprojlim_{\substack{b'\nearrow \infty}} \circ \varinjlim_{c\nearrow \infty}[\mathcal{F}_{\epsilon, b'}CF^{*}(\mathcal{H}_{1})] \arrow[d] \\
  \varprojlim_{\substack{b'\nearrow \infty}} \circ \varinjlim_{c\nearrow \infty}[\mathcal{F}_{\epsilon, b'}CF^{*}(\mathcal{H}_{2})] \arrow[r] & \varprojlim_{\substack{b'\nearrow \infty}} \circ \varinjlim_{c\nearrow \infty}[\mathcal{F}_{\epsilon, b'}CF^{*}(max\{\mathcal{H}_{1},\mathcal{H}_{2}\})], \end{tikzcd}$$
  that is still an acyclic $2$-cube. Hence, it induces the Mayer-Vietoris sequence.
  \end{proof}

Given a Liouville domain $(M,d\theta)$, the skeleton or the core of the Liouville structure is the attractor of the negative flow of the Liouville vector field $Z$, which is a compact set \cite{cieliebakeliashberg12}. For example, consider ribbon graphs of certain punctured Riemann surfaces. The following theorem is a direct corollary of Proposition 5.32 in \cite{mclean18}. We also refer Theorem D in \cite{borman.sheridan.varolgunes21}.

\begin{prop} \label{Liouvilleflow}
For any compact subset $K$ of a Liouville domain $(\underline{M},d\theta)$ containing skeleton of the Liouville flow, the relative symplectic cohomology $SH^{*}(K\subset M)$ is isomorphic to $SH^{*}(\underline{M}\subset M)$.
\end{prop}

\proof
The same idea of the proof of lemmas for proposition 5.32 in \cite{mclean18} works for the completion of a Liouville domain, where symplectic form is exact and where we are free from rearranging index-bounded contact cylinders. Let $\phi^{-Z_{\theta}}_{t}:\underline{M}\to \underline{M}$ be the time $t$ flow of the Liouville vector field $-Z_{\theta}$ with opposite direction. Since $K$ contains the skeleta of $\phi^{-Z_{\theta}}_{t}$, there exists $T>0$ such that $\phi^{-Z_{\theta}}_{t}(\underline{M})\subset K$ for all $t>T$. By independence of symplectomorphism and continuation maps, we get the isomorphism,
\begin{center}$SH^{*}(\underline{M}\subset M)\hookrightarrow SH^{*}(K\subset M)\hookrightarrow SH^{*}(\phi^{-Z_{\theta}}_{t}(\underline{M})\subset M)\cong SH^{*}(\underline{M}\subset M).$ \qedsymbol \end{center}

\subsubsection{Relative Symplectic Cohomology with respect to Symplectic Structures}

Let $X$ be a projective variety with effective ample divisors $D_{1}, D_{2}$. Choose an effective ample divisor $D$ whose support is exactly $D_{1}\cup D_{2}$ and which is linearly equivalent to a large multiple $m$ of $D_{1}$. Let $s$ be a section of a line bundle $\mathcal{O}_{X}(D)$ with $s^{-1}(0)=D$ and $\omega_{D}:=dd^{c}(-log||s||)$ be the induced exact symplectic form on $X\setminus D$. Then $\omega_{D}=m\cdot\omega_{D_{1}}$. Since the set of critical point of $-log||s||$ is a compact subset of $X\setminus (D_{1}\cup D_{2})$ by Lemma 4.3 in \cite{seidelbiased}, $([X\setminus (D_{1}\cup D_{2})]^{\wedge}_{\omega_{D}},\omega_{D})$ is a finite-type Weinstein manifold, where $[X\setminus (D_{1}\cup D_{2})]^{\wedge}_{\omega_{D}}$ is the union of $X\setminus (D_{1}\cup D_{2})$ and the trajectory of the Liouville flow of $\omega_{D}$. Scaling by $\frac{1}{m}$, $([X\setminus (D_{1}\cup D_{2})]^{\wedge}_{\omega_{D_{1}}},\omega_{D_{1}})$ is a finite-type Weinstein manifold.

Let $K$ be a compact subset of $X\setminus D_{1}$. We define a relative symplectic cohomology $SH^{*}_{D_{1}}(K\subset X\setminus (D_{1}\cup D_{2}))$ of a pair $(K\subset [X\setminus (D_{1}\cup D_{2})]^{\wedge}_{\omega_{D_{1}}})$ with respect to $\omega_{D_{1}}|_{[X\setminus (D_{1}\cup D_{2})]^{\wedge}_{\omega_{D_{1}}}}$ by a cofinal family of $(K\subset [X\setminus (D_{1}\cup D_{2})]^{\wedge}_{\omega_{D_{1}}})$-admissible Hamiltonians converging to $\overline{H_{K}}$, where
\begin{enumerate}
    \item $H|_{K}<0$ and $C^{2}$-small.
    \item $H=$ constant outside of a compact set whose interior containing $K$
    \item $\overline{H}_{K}|_{K}\equiv 0$, and $\overline{H}_{K}|_{[X\setminus (D_{1}\cup D_{2})]^{\wedge}_{\omega_{D_{1}}}\setminus K}\equiv \infty$
\end{enumerate}

\begin{prop} \label{samecompletion}
Let $X$ be a projective variety with effective ample divisors $D_{1}, D_{2}$. Let $K$ be a compact subset of $X\setminus(D_{1}\cup D_{2})$. Then the following relative symplectic cohomologies are isomorphic,
$$SH^{*}_{D_{1}\cup D_{2}}(K\subset X\setminus (D_{1}\cup D_{2}))\cong SH^{*}_{D_{1}}(K\subset X\setminus (D_{1}\cup D_{2})).$$
\end{prop}

\begin{proof} of \ref{samecompletion} We get the isomorphism after interpolating symplectic cohomologies of different K\"aher structures on $X$ inducing exact symplectic structures on the complement of ample divisors in \cite{mclean18} and by relabeling isomorphism \cite{varolgunes19}. Denote $(M,\omega_{M}):=(X\setminus (D_{1}\cup D_{2}),\omega_{D_{1}\cup D_{2}})$, $(W,\omega_{W}):=\big([X\setminus (D_{1}\cup D_{2})]^{\wedge}_{\omega_{D_{1}}},\omega_{D_{1}}|_{[X\setminus (D_{1}\cup D_{2})]^{\wedge}_{\omega_{D_{1}}}}\big)$.

Consider two pluri-subharmonic functions $f:=-log||s||$, $f_{1}:=-log||s_{1}||$ on $M$, where $s, s_{1}$ are sections of line bundles $\mathcal{O}_{X}(D)$, $\mathcal{O}_{X}(D_{1})$ with $s^{-1}(0)=D$, $s_{1}^{-1}(0)=D_{1}$, respectively. Let $U_{1}$ be a small neighborhood of $K$ and $U_{2}:=M\setminus K$. Consider a smooth partition of unity $(\rho_{1}, \rho_{2})$ subordinate to $(U_{1},U_{2})$. Since admissible Hamiltonians are $\mathcal{C}^{2}$-small and approach to zero on $K$,
$$SH^{*}_{\omega_{D_{1}\cup D_{2}}}(K\subset M)=SH^{*}_{dd^{c}f}(K\subset M)\cong SH^{*}_{dd^{c}(f+\rho_{1}\cdot f_{1})}(K\subset M).$$ 

Two exhausting plurisubharmonic functions $f_{0}:=f+\rho_{1}\cdot f_{1}$, $f_{1}$ on a complex manifold $(M,J)$ give a finite type (noncomplete) Weinstein homotopy $(M,dd^{c}f_{s},f_{s})_{s\in [0,1]}$ between $(M,dd^{c}f_{0}, f_{0})$, $(M, dd^{c}f_{1},f_{1})$ satisfying that $\overline{\cup_{s\in [0,1]}\text{Skeleton}(M,dd^{c}f_{s})}$ is compact, by the idea of proof of Proposition 11.22 in \cite{cieliebakeliashberg12}, and \cite{eliashberg.gromov91}. By the idea of proof of Proposition 11.8 in \cite{cieliebakeliashberg12}, there exists a diffeotopy $h_{s}:M\to M$ so that $h_{0}=Id_{M}$, $h_{s}^{*}df_{s}-df_{0}$ is exact for all $s\in [0,1]$ and that $h_{s}^{*}df_{s}=df_{0}$ outside a compact set. After completion under the Liouville flow, we get an exact symplectomorphism $\phi:(M,dd^{c}(f+\rho_{1}\cdot f_{1}))\to (W, \omega_{W})$ between two finite type Weinstein manifolds with $\phi(K)= K$. The relabeling all the choices by the symplectomorphism $\phi$ gives an isomorphism of relative symplectic cohomology, $SH^{*}_{dd^{c}(f+\rho_{1}\cdot f_{1})}(K\subset M)\cong SH^{*}_{\omega_{D_{1}}}(K\subset W)$, [Proposition 3.3.3(3), \cite{varolgunes19}]. In sum, $$SH^{*}_{\omega_{D_{1}\cup D_{2}}}(K\subset M)=SH^{*}_{\omega_{D_{1}}}(K\subset W).$$
\end{proof}

\subsubsection{Spectral Sequence Relating Different Choices of Polarization}
Let us recall a spectral sequence associated to a filtered complex. Let $(C^{*}_{\lambda},F^{\bullet})$ is a filtered complex of differential graded module over $\Lambda$. There is a spectral sequence $(E_{r}=\oplus E_{r}^{p,q}, d_{r}^{p,q}:E_{r}^{p,q}\to E_{r}^{p+r,q-r+1})$ whose $E_{0}$-page consists of an associated graded ring $(gr^{p}(C^{p+q}),d_{0}^{p,q}=gr^{p}(d^{p+q}))$ and whose $E_{1}$-page consists of $(E_{1}^{p,q}=H^{p+q}(gr^{p}(C^{*})))$, where $gr^{p}(C^{*}):=F^{p}(C^{*})$ [Lemma 12.24.2, \cite{stacksproj}]. If the filtration is finite (i.e., there exist $n,m$ such that $F^{n}C^{*}=C^{*}$ and $F^{m}C^{*}=0$), then the spectral sequence associated to $(C^{*},F^{\bullet})$ converges to $H^{*}(C^{*})$, 
[Lemma 12.24.11(3), \cite{stacksproj}].
Using the filtration by the winding numbers, we can construct a spectral sequence explaining how relative symplectic cohomology changes when we delete an ample divisor. 

\begin{prop} \label{spectral1}
Let $X$ be a projective variety with ample divisors $D_{1}, D_{2}$. Let $K$ be a compact subset of $X\setminus D_{1}$ containing skeletons of both $(X\setminus D_{1}, \omega_{D_{1}})$ and $(X\setminus (D_{1}\cup D_{2}), \omega_{D_{1}\cup D_{2}})$. Then there exists a spectral sequence converging to $SH^{*}_{D_{1}}(K\subset X\setminus D_{1})$ whose $E_{1}$-page is $SH^{*}_{D_{1}}(K\setminus \mathbb{D}_{2}\subset X\setminus (D_{1}\cup D_{2}))$, where $\mathbb{D}_{2}$ is a small tubular neighborhood of $D_{2}$.
\end{prop}

\begin{proof} of \ref{spectral1}
By the stably-displaceability of a divisor $D_{2}$ and the vanishing result in \cite{mclean18}, we can choose a small neighborhood $\mathbb{D}_{2}$ of $D_{2}$ so that  $SH^{*}_{D_{1}\cup D_{2}}(K\subset X\setminus D_{1})\cong SH^{*}_{D_{1}\cup D_{2}}(K\setminus \mathbb{D}_{2}\subset X\setminus D_{1})$. Therefore, we can consider $K$ as a compact subset of $X\setminus (D_{1}\cup D_{2})$.

Define a filtration $W^{\bullet}$ on $(SC^{*}(K\setminus \mathbb{D}_{2}\subset X\setminus D_{1}),\delta_{K\subset X\setminus D_{1}})$ (resp., $(SC^{*}(K\setminus \mathbb{D}_{2}\subset X\setminus (D_{1}\cup D_{2}))$, $\delta_{K\subset X\setminus (D_{1}\cup D_{2})})$) by the winding number of orbits along $D_{2}$ \footnote{Since $D_{2}$ is effective ample, there exists $f\in \mathcal{O}(X\setminus D_{1})$ so that $D_{2}\setminus D_{1} =f^{-1}(0)$. Define $ev_{f}:X\setminus D_{1}\to \mathbb{C}$ by evaluating $f$. The winding number of orbits is defined by the winding number of the image of each orbit on $\mathbb{C}$ under the map $ev_{f}$ along the origin or the intersection number of $D_{2}\D_{1}$ with a holomorphic capping of an orbit.} is greater than or equal to $w$. Since our admissible Hamiltonians are constant near the divisors, Floer trajectories are holomorphic near the divisors. Since holomorphic curves are positively intersect with any ample divisors, the winding number of an orbit along $D_{2}$ is strictly increasing under the differential map if the connecting Floer trajectory intersect with $D_{2}$, otherwise it preserve the winding number.\footnote{In \cite{borman.sheridan.varolgunes21}, the integrated maximum principle in \cite{abouzaid.seidel10} was used to prove the non-negativity of a Floer trajectory (perturbed $J$-holomorphic curve) and a divisor $D_{2}$ ($J$-holomorphic hypersurface) and their special Hamiltonian is constant on the divisors [Proposition 3.9, \cite{borman.sheridan.varolgunes21}; Section 2.9, \cite{wendl15}].} Therefore, $E_{1}$-page of the spectral sequence associated to $(SC^{*}_{D_{1}}(K\setminus \mathbb{D}_{2}\subset X\setminus D_{1}),\delta_{K\subset X\setminus D_{1}})$ isomorphic to $\bigoplus_{n}W^{n}(SH^{*}(K\setminus \mathbb{D}_{2}\subset X\setminus 
(D_{1}\cup D_{2})))\cong SH^{*}(K\setminus \mathbb{D}_{2} \subset X\setminus (D_{1}\cap D_{2}))$. Remark that the Floer trajectories of $(SC^{*}(K\setminus \mathbb{D}_{2}\subset X\setminus D_{1}),\delta_{K\subset X\setminus D_{1}})$ are allowed to pass through $D_{2}$, on the other hand, those of $(SC^{*}(K\setminus \mathbb{D}_{2}\subset X\setminus (D_{1}\cup D_{2})),\delta_{K\subset X\setminus (D_{1}\cup D_{2})})$ are not.

Since the symplectic structure on the completion of a Liouville domain is exact, the spectral sequence associated to $(C^{*},W^{\bullet})$ is finite. Hence the spectral sequence converges to $H^{*}(SC^{*}(K\setminus \mathbb{D}_{2}\subset X\setminus D_{1}),\delta_{K\subset X\setminus D_{1}})=SH^{*}(K\setminus \mathbb{D}_{2}\subset X\setminus D_{1})\cong SH^{*}(K\subset X\setminus D_{1})$ whose $E_{1}$-page isomorphic to $\bigoplus_{n}W^{n}(SH^{*}(K\setminus \mathbb{D}_{2}\subset X\setminus 
(D_{1}\cup D_{2})))\cong SH^{*}(K\setminus \mathbb{D}_{2} \subset X\setminus (D_{1}\cap D_{2}))$.
\end{proof}

\begin{remark}
When $D_{1}=\emptyset$, then the statement of proposition \ref{spectral1} recover that of the Theorem C in \cite{borman.sheridan.varolgunes21} where a spectral sequence associated to the filtered complex $\big(SC^{*},\partial, F^{*}\big)$ converging to $H^{*}(SC^{*},\partial)=SH^{*}_{D}\big(\underline{X\setminus D}\subset X\big)\cong QH^{*}(X,\Lambda)$ whose $E_{1}$-page consists of $SH^{*}_{D}\big(\underline{X\setminus D}\subset X\setminus D\big)$. But our proof is not enough to prove Theorem C so we refer \cite{borman.sheridan.varolgunes21}.
\end{remark}

Combining propositions \ref{Liouvilleflow}, \ref{samecompletion} and \ref{spectral1}, 
$$SH^{*}_{D_{1}\cup D_{2}}\big(K\setminus \mathbb{D}_{2}\subset X\setminus (D_{1}\cup D_{2})\big)\cong SH^{*}_{D_{1}}(K\setminus \mathbb{D}_{2}\subset X\setminus (D_{1}\cup D_{2}))\Rightarrow SH^{*}_{D_{1}}(K\subset X\setminus D_{1}),$$ 
$$\Big(\text{or }SH^{*}_{D_{1}\cup D_{2}}\big(K\setminus \mathbb{D}_{2}\subset X\setminus (D_{1}\cup D_{2})\big)\Rightarrow SH^{*}_{D_{1}\cup D_{2}}\big(K\setminus \mathbb{D}_{2}\subset X\setminus D_{1}\big)\cong SH^{*}_{D_{1}}(K\setminus \mathbb{D}_{2}\subset X\setminus D_{1}).\Big)$$

\begin{theorem} \label{spectralseqconj}
Let $X$ be a projective variety with effective ample divisors $D_{1}, D_{2}$. Let $K$ be a compact subset of $X\setminus D_{1}$. Then there exists a spectral sequence converging to $SH^{*}_{D_{1}}(K\subset X\setminus D_{1})$ whose $E_{1}$-page is $SH^{*}_{D_{1}\cup D_{2}}\big(K\setminus \mathbb{D}_{2}\subset X\setminus (D_{1}\cup D_{2})\big)$, where $\mathbb{D}_{2}$ is a small tubular neighborhood of $D_{2}$.
\end{theorem}

The following corollary of theorem \ref{spectralseqconj} will be used to prove theorem \ref{cylindrical}.

\begin{corollary} \label{spectralseqconj1}
Let $M$ be an affine variety and $Y\subset M$ a smooth hypersurface. Choose a small tubular neighborhood $\mathbb{Y}$ of $Y\cap \underline{M}$ inside $\underline{M}$. Then there exists a spectral sequence converging to  $SH^{*}(\mathbb{Y}^{c}\subset M)\cong SH^{*} \big(\underline{M\setminus Y}\subset M\big)$ whose $E_{1}$-page is $SH^{*}(\mathbb{Y}^{c}\subset M\setminus Y)\cong SH^{*}\big(\underline{M\setminus Y}\subset M\setminus Y\big)$.
\end{corollary}

\subsection{Vanishing or Invariance Theorems of Symplectic Cohomology}\label{h-principle} 
In this section, we review some results on vanishing theorems and invariance theorems of symplectic cohomology, which we use when we apply the main theorems. First of all, a non-trivial fact is that $SH^{*}(\C^{n})=0$ for all $n$, [\cite{floer.hofer.wysocki94}, Chapter 3, \cite{oancea04}] by computing Conley-Zehnder indices and taking the direct limit. 

\subsubsection{Symplectic Cohomology under Attaching Weinstein Handles}
As Morse theory let us understand topology of smooth manifolds and construct them by attaching handles of critical points, pseudo-convex function theory and h-principle let us understand symplectic topology of Liouville domains and construct some of them by attaching Weinstein handles, \cite{eliashberg90}, \cite{gromov86}. In  \cite{gromov85}, Gromov developed the pseudo-holomorphic curve theory and showed that a Liouville domain $\underline{M}$ whose boundary $\partial \underline{M}$ is contactomorpic to the standard contact structure on $S^{3}$ is isomorphic to $(\R^{4},\omega_{\text{standard}})$ as Liouville domains. The behavior of symplectic cohomology under Weinstein handle attaching was studied by Cieliebak in \cite{cieliebakhandle}. An important fact is that subcritical handle attaching does not change the symplectic cohomology before and after. By attaching subcritical handles and the h-cobordism, Seidel-Smith showed that Liouville domains of dimension $2n\geq 4$, with sphere boundary with the standard contact structure, should have vanishing symplectic cohomology in \cite{seidelbiased}. Moreover, Seidel-Smith and McLean found exotic Stein manifolds. Casals-Murphy detected flexibility of affine varieties which leads vanishing symplectic cohomology by detecting a zig-zag on  Legendrian front projection \cite{casalsmurphy19}. 

Let us recall the Weinstein handle attachment as follows \cite{weinsteinhandle}.

\begin{itemize}
    \item A $1$-form $\alpha$ on a manifold $E$ of dimension $2n+1$ is called to be contact if $\alpha\wedge (d\alpha)^{n}\neq 0$. A hyperplane field $\xi$ of the tangent bundle $TE$ is called to be a contact structure if it is a kernel of some contact $1$-form (by definition, it is maximally non-integrable).
    \item A submanifold $L$ of a contact manifold $(E,\xi)$ is called to be isotropic if $T_{x}L\subset \xi_{x}$ for all $x\in L$ (i.e., $j^{*}_{L}\alpha=0 \iff j^{*}_{L}d\alpha=0 \iff T_{x}L\subset (\xi_{x},d\alpha|_{\xi_{x}})$ isotropic). A maximual isotropic submanifold is called a Legendrian submanifold. 
\end{itemize}

Let $(E,\xi)$ be a contact manifold and $j:S^{k-1}\to E$ be a isotropic sphere. Denote $(TS^{k-1})^{\perp}$ a sub-bundle of $(\xi|_{L}, d\alpha|_{\xi})$ that is orthogonal to $TS^{k-1}$ with respect to $d\alpha|_{\xi}$. $TS^{k-1}\subset (TS^{k-1})^{\perp}$ (isotropic) and the quotient bundle $(TS^{k-1})^{\perp}/TS^{k-1}\cong \C^{n-k}$ has a structure of the conformal symplectic normal bundle of $S^{k}$ induced from $d\alpha$. The standard chart is $\R^{2n}=T^{*}\R^{k}\times \C^{n-k}$ with the symplectic form $\omega=\sum_{i=0}^{k}dp_{i}\wedge dq_{i}+\sum_{j=0}^{n-k}dx_{j}\wedge dy_{j}$ and with the Liouville vector field $Z_{k}=\sum_{i=1}^{k}(2p_{i}\frac{\partial}{\partial p_{i}}-q_{i}\frac{\partial}{\partial q_{i}})+\frac{1}{2}(\sum_{j=1}^{n-k}x_{j}\frac{\partial}{\partial x_{j}}+y_{j}\frac{\partial}{\partial y_{j}})$ and $Z_{k}$ is a gradient vector field for the Morse function $f_{k}:=\sum(p_{i}^{2}-\frac{1}{2}q_{i}^{2})+\sum(\frac{1}{4}x_{j}^{2}+\frac{1}{4}y_{j}^{2})$ with respect to the standard Euclidean metric on $\R^{n}$. Define a $1$-form $\alpha_{k}:=\iota_{Z_{k}}\omega=\sum_{i}(2p_{i}dp_{i}+q_{i}dq_{i})+\frac{1}{2}\sum_{j}(x_{j}dx_{j}-y_{j}dy_{j})$. Along the flow of $Z_{k}$, the stable(resp. unstable) subspace is $W^{s}_{k}=\{q_{1}=\cdots=q_{k}=0\}\cong \R^{2n-k}$ (resp. $W^{u}_{k}=\{p_{i}=0,x_{j}=y_{j}=0\text{ for all} i,j\}\cong \R^{k}$). Since the form $\alpha_{i}$ pulls back to $0$ on $W^{s}_{k}$ along the flow, a sphere $S^{st}_{k}:=W^{s}_{k}\cap \{f_{k}=1\}$ is an isotropic submanifold of $\{f_{k}=1,\alpha_{k}\}$, call it the ascending sphere of index $k$ (resp. $S^{ut}_{k}:=W^{u}_{k}\cap \{f_{k}=-1\}$ is called the descending sphere). A standard Weinstein $k$ handle in $\R^{2n}$ is a bounded neighborhood of the descending sphere in $\{f_{k=-1}\}$ with a connectic manifold diffeomorphic to $S^{2n-k-1}\times D^{k}$. It is known that \textit{every Weinstein domain $M^{2n}$, $n\geq 3$, can be Weinstein homotoped to a subcritical domain $M_{subcrit}$ with handles attached to the Legendrian (may not be loose) link $\Lambda_{1}\sqcup \cdots \sqcup \Lambda_{k-1}\sqcup \Lambda_{k}\subset \partial M_{subcrit}$ such that $\Lambda_{1}\sqcup \cdots \sqcup \Lambda_{k-1}$ is a loose link and $\Lambda_{k}$ is a loose Legendrian where linking information is symplectically nontrivial} \cite{lazarev20}. 

Cieliebak defined relative symplectic cohomology of a pair $W\subset (M,\omega)$ of a codimension $0$ submanifold $W$ of $(M,\omega)$ with $\omega$-convex boundary $\partial W$, using Hamiltonians that are $C^{2}$-small Morse functions in the interior of $\underline{W}$ and increasing sharply near $\partial \underline{W}$ \cite{cieliebakhandle}. Let $M_{\Supset k}$ be a Liouville domain obtained by attaching a Weinstein $k$-handle to $\partial \underline{M}$. Considering that  the only generator of the complex of relative symplectic cohomology of a pair $M\subset (M_{\Supset k},\omega)$ is a critical point of the $k$-handle and that it disappears under the transfer functor in the direct limit (the Conley-Zehnder index can be increased under the transfer functor), he showed the following invariance property of symplectic cohomology under subcritical handle attachment.

\begin{prop} [Theorem 1.11(1), \cite{cieliebakhandle}]
Let $(\underline{M},d\theta_{\underline{M}})$ be a Liouville domain of dimension $2n$, and $S\subset \partial \underline{M}$ be a contact-isotropic(i.e.,$\theta_{\underline{M}}|_{S}$=0) sphere of dimension $k-1<n-1$. (The symplectic normal bundle of $S$ in $\underline{M}$ is trivial.) Let $\underline{M}_{\Supset k}:=\underline{M}\cup_{S}H_{k}$ be a Liouville domain obtained by attaching a Weinstein $k$-handle $H_{k}$ to $\partial \underline{M}$ near $S$. Then $SH^{*}(M_{\Supset k})\cong SH^{*}(M)$.
\end{prop}

A direct corollary is that for a subcritical Weinstein manifold $M$, $SH^{*}(M)=0$.

\begin{prop} [Corollary 6.5, \cite{seidelbiased}] Let $\underline{M}$ be a Liouville domain of dimension $2n\geq 4$, such that $\partial \underline{M}$ is contact isomorphic to the standard contact structure on $S^{2n-1}$. Then $SH^{*}(M)=0$.
\end{prop}

Let $\underline{M}$ be a Liouville domain. Let $\Lambda:=\cup_{i=1}^{k} \Lambda_{i}$ be a disjoint union of Legendrian spheres in the boundary $\partial \underline{M}$ and denote $M_{\Supset \Lambda}$ a convex Weinstein domain on which symplectic cohomology is well-defined, obtained by attaching Weinstein-handles to $\underline{M}$ along $\Lambda$. In the case that $\underline{M}$ is a Liouville subdomain of $\underline{M}'$ of dimension $2n$ and a Weinstein cobordism $(\mathfrak{S}:=\overline{\underline{M}'}\setminus \text{int}(\underline{M}), f:\mathfrak{S}\to \R)$ contains only Morse critical points $p_{1},\cdots, p_{k}$ of index $n$, we consider the Legendrian spheres $\Lambda_{j}$'s are the intersection of the stable manifold $L_{j}$ of $p_{j}$ and the boundary $\partial \underline{M}'$ of $\underline{M}'$. Then $M'\cong M_{\Supset \Lambda}$. \par
Let $LHA(\Lambda)$ be the Legendrian homology differential graded algebra(DGA), whose complex is generated by Reeb chords connecting $\Lambda_{j}$'s. Let $LHO(\Lambda)$ be the subalgebra of cyclically compaosable monomials. Let $LH^{H_{0}}_{+}$ be a homology of a mapping cylinder of the reduced Legendrian DGA $LHO^{+}(\Lambda)$ of $LHO(\Lambda)$. Denote $LH^{H_{0}}(\Lambda)$ the full Legendrian homology DGA whose complex consists of $LHO^{+}(\Lambda)$ and the Legendrian spheres $\Lambda$ [Section 4, \cite{BEE12}].

\begin{lemma}[Theorem 5.4, 5.6, \cite{BEE12}]
There are exact sequences,
$$\begin{tikzcd} [cramped, sep=small]
\cdots \to LH^{H_{0}}_{+}(\Lambda)\arrow[r] & SH^{*}_{+}(M_{\Supset \Lambda})\arrow[r] & SH^{*}_{+}(M) \arrow[r] & LH^{H_{o}}_{+}(\Lambda)[1]\to \cdots.
\end{tikzcd}
$$
$$\begin{tikzcd} [cramped, sep=small]
\cdots \to LH^{H_{0}}(\Lambda)\arrow[r] & SH^{*}(M_{\Supset \Lambda})\arrow[r] & SH^{*}(M)\arrow[r] & LH^{H_{o}}(\Lambda)[1]\to \cdots.
\end{tikzcd}
$$
\end{lemma}

\begin{lemma}[\cite{ekholm.etnyre.sullivan07}, \cite{murphy12}]
The Legendrian contact homology algebra of a loose Legendrian vanishes.
\end{lemma}

\begin{prop} [Theorem 3.2, \cite{murphysiegel18}] \label{flexible}
Let $M$ be a flexible Weinstein domain, then, for any closed two-form on $M$, the twisted symplectic cohomology $SH^{*}(M, \Lambda_{2})=0$.
\end{prop}

The following example is from Casals-Murphy's recipe on the Legendrian front projection of Weinstein Lefschetz bifibration of affine varieties in \cite{casalsmurphy19}. Define
$$L_{a,b}^{n}:=\{(x,y,z_{1},\cdots,z_{n-1}:x^{a}y^{b}+\sum_{i=1}^{n-1}z_{i}^{2}=1\}\subset \C^{n+1},$$
where $a,b\in \N$ are two coprime integers with $1\leq a<b$. By [Theorem 1.1., \cite{casalsmurphy19}], $L_{1,b}^{n}$ are flexible for all $b\geq 2$, therefore, $SH^{*}(L_{1,b}^{n})=0$. In particular, $L_{1,b}^{3}$ are diffeomorphic to $S^{3}\times \R^{3}$, so we can see $\ell(L_{1,b}^{3})=0$ as well. Inspired by it, we expect that finding flexible Weinstein cobordisms by Casals-Murphy's algorithm on the Lefschetz fibration of an affine variety $M$ would lead to compute $\ell(M)$ as well as to find uniruled subvarieties of $M$.

\subsubsection{Vanishing Symplectic Cohomology from $\C^{*}$-action} \label{ritter}
When a convex weak-monotone symplectic manifold has Hamiltonian circle action which generates Hamiltonian/Reeb orbits near at the infinity, the circle action make the Conley-Zehnder index of Hamiltonian orbits goes to infinity as the slope of Hamiltonian goes to infinity. As a direct limit of Hamiltonian Floer cohomology under the direct system by the increasing slope of Hamiltonian, symplectic cohomology vanishes. Let us review on Ritter's vanishing result here.

\begin{prop}
\begin{enumerate}
    \item Let $M$ be the total space of the line bundle $\mathcal{O}_{\P^{m}}(-n)$ over $\P^{m}$. Then,
\begin{itemize}
    \item for $n>2m$, $SH^{*}(M)=0$,
    \item for $2+m\leq n \leq 2m$, the weak monotonicity condition does not hold.
    \item for $n=1+m$, $SH^{*}(M)=0$,
    \item for $1+\frac{m}{2}\leq n<1+m$, $SH^{*}(M)\neq 0$ has rank a multiple of $1+m-n$,
    \item for $1\leq n<1+\frac{m}{2}$, $SH^{*}(M)=\Lambda[\omega_{Q}]/(\omega^{1+m-n}_{Q}-(-n)^{n}t)$ [Theorem 5, \cite{ritter14}].
\end{itemize}
    \item If $c_{1}(TM)(\pi_{2}(M))=0$, then $SH^{*}(M)=0$ [Theorem 6, \cite{ritter14}].
    \item Let $M$ be a negative line bundle over $B$ with some conditions we refer to \cite{ritter14}. Define the minimal Chern number $c_{min}$ of a manifold $M$ to be the number satisfying $c_{1}(TM)(\pi_{2}(M))=c_{min}\Z$. If the minimal Chern number of $M$ is greater or equal to the rank of $H^{*}(B)$, then $SH^{*}(M)=0$ [Corollary 7, \cite{ritter14}].
    \item $SH^{*}(M)=0$ if and only if $\pi^{*}_{M}c_{1}(L)$ is nilpoitent in the quantum cohomology $QH^{*}(M)$ of $M$. In particular, the quantum cup product reduces to the ordinary cup product, then $SH^{*}(M)=0$ [Corollary 2, \cite{ritter14}].
\item Let $E\to B$ be any complex vector bundle, $L\to B$ a negative line bundle and $M_{k}$ be the total space of $E\otimes L^{\otimes k}\to B$. Then for $k\gg 0$, $SH^{*}(M_{k})=0$ [Theorem 10, \cite{ritter14}].    
\end{enumerate}
\end{prop}

\begin{example}[Example p.7, \cite{ritter14}]
Let $M$ be the total space of the canonical bundle over Fano variety, then $SH^{*}(M)=0$.
\end{example}

\begin{prop}[Theorem p.3, \cite{ritter10}] Let $M$ be an almost locally Euclidean space, i.e. a minimal resolution of ADE singularities on $\C^{2}$, which is the quotient $\C^{2}/\Gamma$ of $\C^{2}$ by a finite subgroup $\Gamma\in SU(2)$ or plumbings of cotangent bundles $T^{*}\P^{1}$ according to ADE Dynkin diagrams(It is known to be a non-compact simply-connected hyperk\"ahler $4$-manifold). Then for generic $\beta\in H^{2}(M)$, $SH^{*}(M,d\theta;\underline{\Lambda}_{\tau\beta})=0$.
\end{prop}

\subsubsection{Invariance from Convex Symplectic Lefschetz fibration}
Understanding the convex symplectic Lefschetz fibration structure on Kaliman modification leads the fact that symplectic cohomology is invariant under Kaliman modification. The theorem of invariance will be used to prove theorem \ref{construction1} below.

\begin{prop}[Theorem 2.31, \cite{mclean09}] \label{marklef}
Let $M$ be an smooth affine variety of dim$_{\C}M\geq 3$ with compactification $X$ with ample divisor $D$. Let $Z$ be an irreducible divisor in $X$ and $p\in Z\cap M$ a smooth point. Assume $c_{1}\big(\mathcal{B}l_{p}M\setminus \widetilde{(Z\cap M)}\big)=c_{1}(M\setminus Z)=0$. Then $SH^{*}\big(\mathcal{B}l_{p}M\setminus \widetilde{(Z\cap M)}\big)\cong SH^{*}(M\setminus Z)$.
\end{prop}

\subsubsection{Vanishing Symplectic Cohomology from Stably-Displaceability}

The same idea of vanishing result on relative symplectic cohomology for stably-displaceable compact subsets of a closed symplectic manifold in [Theorem 5.12, \cite{mclean18}, \cite{varolgunes19}] works for compact subsets of a completed Liouville domain. 
A subset $K$ of a symplectic manifold $(M,\omega)$ is called to be Hamiltonian displaceable if there is a Hamiltonain symplectomorphism $\phi$ satisfying $\phi(K)\cap K=\emptyset$. A subset $K\subset M$ is called to be stably displaceable if $K\times S^{1}\subset (M\times T^{*}S^{1},\omega+d\sigma \wedge d\tau)$ is Hamiltonian displaceable, where $(\sigma,\tau)\in \R\times S^{1}\cong T^{*}S^{1}$.

\begin{prop} [\cite{mclean18}]\label{invariance3}
If a compact subset $\mathbb{Y}\subset M$ is stably-displaceable by Hamiltonians on $M\times T^{*}S^{1}$, then $SH^{*}(\mathbb{Y}\subset M)=0$.
\end{prop}

\begin{proof} [Proof of Proposition \ref{invariance3}] 
Let $\mathcal{K}_{t},c$ be a Hamiltonian on $M\times T^{*}S^{1}$ which displaces a small neighborhood of $V$ with height $c$. Let $\mathcal{H}_{t},c$ be an admissible time-dependent Hamiltonian with height $c$. Denote $\phi_{\mathcal{K}_{t}}$ (resp. $\phi_{\mathcal{H}_{t}}$) the one-parameter family of the Hamiltonian vector field $X_{\mathcal{K}_{t}}$ (resp. $X_{\mathcal{H}_{t}}$) of $\mathcal{K}_{t}$ (resp. $\mathcal{H}_{t}$). We need the following lemmas that we provide proofs below.

\begin{lemma}[Proposition 3.3.2, \cite{varolgunes19}] \label{layers of Hamiltonian}
 There exist an isomorphism of relative symplectic cohomologies after tensoring with the Novikov field,
$$H^{*}\Big(\varinjlim_{\substack{a'\searrow -\infty }}\circ \varprojlim_{\substack{b'\nearrow \infty}} \circ \text{tel}_{c}[\mathcal{F}_{a', b'}CF^{\bullet}(\mathcal{K}_{c}+\mathcal{H}_{t}\circ \phi_{\mathcal{K}_{t}, c})]\Big)\cong H^{*}\Big(\varinjlim_{\substack{a'\searrow -\infty }}\circ\varprojlim_{\substack{b'\nearrow \infty}} \circ \text{tel}_{c}[\mathcal{F}_{a', b'}CF^{\bullet}(\mathcal{H}_{t}\circ \phi_{\mathcal{K}_{t}, c})])\Big).$$
\end{lemma}

\begin{lemma} [Proposition 3.3.3, \cite{varolgunes19}] \label{invariance4}
For any time-dependent Hamiltonian diffeomorphism $\phi_{\mathcal{K}_{t}}$ on $M$,
$$H^{*}\Big(\varinjlim_{\substack{a'\searrow -\infty }}\circ \varprojlim_{\substack{b'\nearrow \infty}} \circ \text{tel}_{c}[\mathcal{F}_{a', b'}CF^{\bullet}(\mathcal{H}_{t}\circ \phi_{\mathcal{K}_{t}, c})]\Big)\cong H^{*}\Big(\varinjlim_{\substack{a'\searrow -\infty }}\circ \varprojlim_{\substack{b'\nearrow \infty}} \circ \text{tel}_{c}[\mathcal{F}_{a', b'}CF^{\bullet}(\mathcal{H}_{t, c})])\Big).$$
\end{lemma}

Since $\mathcal{K}_{t}$ stably-displaces small neighborhood of $V$, there is no fixed point of $\phi_{\mathcal{K}_{t}+\mathcal{H}_{t}\circ \phi_{\mathcal{K}_{t}}^{-1}}=\phi_{\mathcal{K}_{t}}\circ \phi_{\mathcal{H}_{t}}$. Therefore,
$H^{*}\Big(\varinjlim_{\substack{a'\searrow -\infty }}\circ\varprojlim_{\substack{b'\nearrow \infty}} \circ \text{tel}_{c}[\mathcal{F}_{a', b'}CF^{\bullet}(\mathcal{K}_{c}+\mathcal{H}_{t}\circ \phi_{\mathcal{K}_{t}, c})]\Big)=0.$
By lemmas \ref{layers of Hamiltonian}, \ref{invariance4},
\begin{align*} 
SH^{*}(K\times S^{1}\subset M\times T^{*}S^{1})&= H^{*}\Big(\varinjlim_{\substack{a'\searrow -\infty }}\circ\varprojlim_{\substack{b'\nearrow \infty}} \circ \text{tel}_{c}[\mathcal{F}_{a', b'}CF^{\bullet}(\mathcal{H}_{t, c})]\Big)  \\ 
&\cong H^{*}\Big(\varinjlim_{\substack{a'\searrow -\infty }}\circ\varprojlim_{\substack{b'\nearrow \infty}} \circ \text{tel}_{c}[\mathcal{F}_{a', b'}CF^{\bullet}(\mathcal{H}_{t}\circ \phi_{\mathcal{K}_{t}, c})]\Big) \\
&\cong H^{*}\Big(\varinjlim_{\substack{a'\searrow -\infty }}\circ\varprojlim_{\substack{b'\nearrow \infty}} \circ \text{tel}_{c}[\mathcal{F}_{a', b'}CF^{\bullet}(\mathcal{K}_{c}+\mathcal{H}_{t}\circ \phi_{\mathcal{K}_{t}, c})]\Big) =0.
\end{align*}

By the K\"unneth formula (Section 4.3, \cite{varolgunesthesis}), $SH^{*}(K\subset M)\subset SH^{*}(K\times S^{1}\subset M\times T^{*}S^{1})=0$.
\end{proof}

\begin{proof} [Proof of Lemma \ref{layers of Hamiltonian}]
Let us explain the idea of [Proposition 3.3.2, \cite{varolgunes19}] for readers' convenience. $\mathcal{K}+\mathcal{H}\circ \phi_{\mathcal{K}}^{-1}$ and $\mathcal{H}\circ \phi_{\mathcal{K}}^{-1}$ are two acceleration data for $K$ satisfying that $\mathcal{K}+\mathcal{H}\circ \phi_{\mathcal{K}}^{-1}\leq \mathcal{H}\circ \phi_{\mathcal{K}}^{-1}$ (We can assume that $\mathcal{K}\leq 0$). Denote $\mathcal{H}^{\vee}:=\mathcal{K}+\mathcal{H}\circ \phi_{\mathcal{K}}^{-1}=\{H^{\vee}_{\bullet}\}$ and $\mathcal{H}^{\blor}:=\mathcal{H}\circ \phi_{\mathcal{K}}^{-1}=\{H^{\blor}_{\bullet}\}$. Such two acceleration data give a map of $1$-rays
$$\begin{tikzcd} [cramped, sep=small]
  CF^{*}(H^{\vee}_{1}) \arrow[r] \arrow[rd] \arrow[d] & CF^{*}(H^{\vee}_{2}) \arrow[r] \arrow[rd] \arrow[d] & CF^{*}(H^{\vee}_{3}) \arrow[d]\arrow[r] & \cdots\\
  CF^{*}(H^{\blor}_{1}) \arrow[r] & CF^{*}(H^{\blor}_{2}) \arrow[r] & CF^{*}(H^{\blor}_{3})\arrow[r] & \cdots . \end{tikzcd}$$
  And we have an induced map, 
$$H^{*}\Big(\varinjlim_{\substack{a'\searrow -\infty }}\circ\varprojlim_{\substack{b'\nearrow \infty}} \circ \text{tel}_{c}[\mathcal{F}_{a', b'}CF^{\bullet}(\mathcal{H}^{\vee})]\Big) \to H^{*}\Big(\varinjlim_{\substack{a'\searrow -\infty }}\circ\varprojlim_{\substack{b'\nearrow \infty}} \circ \text{tel}_{c}[\mathcal{F}_{a', b'}CF^{\bullet}(\mathcal{H}^{\blor})]\Big).$$
Since $\big(\lim\mathcal{H}^{\vee}\big)^{-1}(\infty)=\big(\lim\mathcal{H}^{\blor}\big)^{-1}(\infty)$, we can choose infinite strictly monotone sequence $n(i)$ and $m(i)$ of positive integer in the indexing set of sequences so that $H^{\vee}_{i}\leq H^{\blor}_{i}\leq H^{\vee}_{n(i)}\leq H^{\blor}_{m(i)}$. We have maps of $1$-rays for each action bound $a',b'$,
$$\begin{tikzcd} [cramped, sep=small]
  \mathcal{F}_{a', b'}CF^{*}(H^{\vee}_{1}) \arrow[r] \arrow[rd] \arrow[d] & \mathcal{F}_{a', b'}CF^{*}(H^{\vee}_{2}) \arrow[r] \arrow[rd] \arrow[d] & \mathcal{F}_{a', b'}CF^{*}(H^{\vee}_{3}) \arrow[d]\arrow[r] & \cdots\\
  \mathcal{F}_{a', b'}CF^{*}(H^{\blor}_{1}) \arrow[r] \arrow[rd] \arrow[d] & \mathcal{F}_{a', b'}CF^{*}(H^{\blor}_{2}) \arrow[r] \arrow[rd] \arrow[d] & \mathcal{F}_{a', b'}CF^{*}(H^{\blor}_{3})\arrow[r]\arrow[d] & \cdots\\
  \mathcal{F}_{a', b'}CF^{*}(H^{\vee}_{n_(1)}) \arrow[r] \arrow[rd] \arrow[d] & \mathcal{F}_{a', b'}CF^{*}(H^{\vee}_{n(2)}) \arrow[r] \arrow[rd] \arrow[d] & \mathcal{F}_{a', b'}CF^{*}(H^{\vee}_{n(3)}) \arrow[d]\arrow[r] & \cdots\\
  \mathcal{F}_{a', b'}CF^{*}(H^{\blor}_{m(1)}) \arrow[r] & \mathcal{F}_{a', b'}CF^{*}(H^{\blor}_{m(2)}) \arrow[r] & \mathcal{F}_{a', b'}CF^{*}(H^{\blor}_{m(3)})\arrow[r] & \cdots, \end{tikzcd}$$
so we have the induced canonical map, $H^{*}\Big(\varinjlim_{\substack{a'\searrow -\infty }}\circ\varprojlim_{\substack{b'\nearrow \infty}} \circ \text{tel}_{c}[\mathcal{F}_{a', b'}CF^{\bullet}(\mathcal{H}^{\vee})]\Big) \to$\\ $\to H^{*}\Big(\varinjlim_{\substack{a'\searrow -\infty }}\circ\varprojlim_{\substack{b'\nearrow \infty}} \circ \text{tel}_{c}[\mathcal{F}_{a', b'}CF^{\bullet}(\mathcal{H}^{\blor})]\Big)\to H^{*}\Big(\varinjlim_{\substack{a'\searrow -\infty }}\circ\varprojlim_{\substack{b'\nearrow \infty}} \circ \text{tel}_{c}[\mathcal{F}_{a', b'}CF^{\bullet}(\mathcal{H}^{\vee})]\Big).$ Therefore, we get an isomorphism.
\end{proof}
\begin{proof}[Proof of Lemma \ref{invariance4}]
The isomorphism is defined by relabeling all choices by the Hamiltonian diffeomorphism preserving the structure of being cylindrical at infinity [Proposition 3.3.3.(3), \cite{varolgunes19}]. 
\end{proof}

McLean proved that any compact subvariety of a K\"ahler manifold of positive codimension is stably displaceable by h-principle [Proposition 6.20, Corollary 6.21, \cite{mclean18}]. By proposition \ref{invariance3} and \ref{mv}, we have the following.
\begin{prop} \cite{mclean18} \cite{varolgunesthesis} \label{stablydisplconj}
Let $M$ be an affine variety and $Y\subset M$ a hypersurface. We can choose a small tubular neighborhood $\mathbb{Y}$ of $Y\cap \underline{M}$ inside $\underline{M}$ that is stably displaceable and a tubular neighborhood $\mathbb{Y}^{c}$ of $\underline{M}\setminus Y$ satisfying that $\partial \mathbb{Y}\cap \partial \mathbb{Y}^{c}=\emptyset$ and $Y\cap \mathbb{Y}^{c}=\emptyset$. Therefore, $SH^{*}(\mathbb{Y}\subset M)=0$ and $SH^{*}(\mathbb{Y}^{c}\subset M)\cong SH^{*}(\underline{M}\subset M)$
\end{prop}

\section{Proof of the Main Theorem} \label{section3}

The main theorem of this paper is the following. Let $M$ be a smooth affine variety of complex dimension $n$ together with a trivialization of some power of its canonical bundle (viewed as a complex vector bundle, rather than a holomorphic bundle). Let $X$ be a smooth projective variety compactifying $M$ with an ample\footnote{Symplectic cohomology of a complement of a nef divisor could be defined but the intersection number with some curve would be zero.} divisor $D$.

\begin{theorem}[\ref{main theorem0}] \label{main theorem00}
If a cohomology class $[\mho]\in H^{m}(M)$ be in the image of the map $\delta$ for $m=2k$ or $2k+1$ for some $k\in \N$. Then there exists a subvariety $\Xi_{\mho}\subset M$ of dimension at least $n-k$ satisfying the following properties:
\begin{enumerate}
    \item $\Xi_{\mho}$ is $\C$-uniruled. In other words, for each $p\in \Xi_{\mho}$, there exists a non-constant algebraic map $\overline{v}_{p}:\C\to \Xi_{\mho}$ whose image contains $p$.
    \item For any exhausting finite type Morse function $f$ on $X\setminus D$, $\Xi_{\mho}$ set-theoretically intersect with a unstable submanifold $\mho_{f}$ of $f$ representing $[\mho]\in H^{m}(M)\cong HM^{m}(M,f)$.
\end{enumerate}
\end{theorem}

\subsection{$J$-holomorphic $(k,\mathcal{E}, [\mho])$-uniruledness}
In this subsection, as a generalization of the definition for a Liouville domain $\underline{M}$ being $(k,\mathcal{E})$-uniruled in \cite{mclean14}, we will introduce a definition for a Liouville domain $\underline{M}$ to be $(k,\mathcal{E},[\mho])$-uniruled, roughly saying, $\underline{M}$ is $(k,\mathcal{E})$-uniruled anchored at a unstable submanifold of any Morse function, representing a cohomology class $[\mho]$. Denote $\Omega$ a stable submanifold representing the Lefschetz dual of $[\mho]$, and $\Omega \hookrightarrow M$, a $r$-pseudocycle. Denote $[\Omega]$ the corresponding integral homology class by [Theorem 1.1 \cite{zinger08}].

\begin{defn}[Definition 2.2, \cite{mclean14}]
Let $k>0$ be an integer, and let $\mathcal{E}>0$ be a real number. We say that a Liouville domain $\underline{M}$ is $(k, \mathcal{E})$-uniruled, if
\begin{itemize}
    \item for every convex $d\theta_{\underline{M}}$-compatible almost complex structure $J$ on $\underline{M}$ and every point $x\in$ int$(\underline{M})$ of the interior of $\underline{M}$, where $J$ is integrable on a neighborhood of $x$, there is a proper $J$-holomorphic map $u:\Sigma\to$ int$(\underline{M})$ from a genus $0$ Riemann surface $\Sigma$, with the image of $u$, denote im$(u)$, contains $x$, 
    \item the rank of $H_{1}(\Sigma,\Q)$, rk$H_{1}(\Sigma,\Q)\leq k-1$, and 
    \item the energy of $u$, $\int_{S}u^{*}d\theta_{\underline{M}}\leq \mathcal{E}$.
\end{itemize}
\end{defn}

 We generalize the notion of $(k,\mathcal{E})$-uniruledness to $(k,\mathcal{E}, [\mho])$-uniruledness that has cohomology constraints.

\begin{defn} \label{cohomology intersection uniruled}
Let $k>0$ be an integer, $\mathcal{E}>0$ be a real number, $[\mho]\in H^{k}(\underline{M})$. A Liouville domain $\underline{M}$ is $(k, \mathcal{E}, [\mho])$-ruled, if
\begin{itemize}
    \item for every convex tamed almost complex structure $J$ on $\underline{M}$, and for every Morse function $f:\underline{M}\to \R$ with $\partial \underline{M}$ as its highest regular level set and for a unstable submanifold $\mho_{f}$ of $f$ representing $[\mho]$, there is a proper $J$-holomorphic map $u:S\to\text{int}(\underline{M})$, where $S$ is a genus $0$ Riemann surface, with $\overline{\text{image}(u)}\cap \mho_{f} \neq \emptyset$,
    \item $\text{rk}H_{1}(S,\Q)\leq k-1$, and
    \item the energy of $u$ is at most $\mathcal{E}$, $\int_{S} u^{*}\omega\leq \mathcal{E}$. 
\end{itemize}
\end{defn}

For a $J$-holomorphic map $u:S\to \text{int}(M)$, we denote $[u]\cap\mho\neq \emptyset$ if for any Morse function $f:M\to \R$ and its unstable submanifold $\mho_{f}$ representing $\mho$, $\overline{\text{image}(u)}\cap \mho_{f} \neq \emptyset$. 

\begin{lemma}\label{symplectic invariance of uniruled}
$[$Lemma $3.2$, Proposition $3.1$, \cite{mclean14}$]$
Let $N$, $L$ be Liouville domains such that $N$ is a codimension $0$ symplectic submanifold of $L$ with the inclusion map $\iota:N\to L$ is a symplectic embedding and a homotopy equivalence. If $L$ is $(k, \mathcal{E}, [\mho])$-uniruled, then the submanifold $N$ is also $(k, \mathcal{E}, \iota ^{*} [\mho])$-uniruled.
\end{lemma}

\proof $(k, \mathcal{E})$-uniruledness in \cite{mclean14} will be reviewed for reader's convenience. Since $\iota:(N,\theta_{N})\hookrightarrow (L,\theta_{L})$ is a symplectic embedding, $\theta_{N}-\theta_{L}|_{N}$ is a closed $1$-form so $[(\theta_{N}-\theta_{L})|_{N}]\in H^{1}(N;\R)$. Since $\iota$ is a homotopy equivalence, $\iota^{*}:H^{1}(L;\R)\to H^{1}(N;\R)$ is an isomorphism. So there is a closed $1$-form $\eta$ on $L$ so that $\iota^{*}\eta=[\theta_{N}-\theta_{L}|_{N}]$. Define $\theta':=\theta_{L}+\eta$, then $[\theta'|_{N}-\theta_{N}]=0$ and $d\theta'=d\theta_{L}$. We can choose a convex almost complex structures $J_{N},J_{L}$ on $N,L$, respectively by Lemma \ref{lem1}. Since the space of all almost complex structures compatible with a symplectic form is contractible, we can choose a compatible almost complex structure $J'$ on $L$ such that $J'=J_{L}$ near $\partial L$ and $J'|_{N}=J_{N}$. \par
We can choose a collar neighborhood $(1-\epsilon, 1]\times \partial N$ of $\partial N$ inside $N$ such that $d\theta_{N}\circ J_{N}=dr$ where $r$ is a coordinate of $(1-\epsilon, 1]$. For $R\in (1-\epsilon, 1)$, define $N_{R}:=N\setminus (R,1]\times\partial N$. Assume that $L$ is $(k,\mathcal{E})$-uniruled, then there is a proper $J$-holomorphic map $u:\Sigma\to$ Int$(L)$ passing through $p$ of finite energy at most $\mathcal{E}$. Show that $H_{1}(u^{-1}(\text{int}(N)))\to H_{1}(\Sigma)$ is injective. Since $H_{1}(u^{-1}(\text{int}(N))$ is the direct limit of $H_{1}(u^{-1}(N_{R}))$ as $R\to 1$, it is enough to show that $H_{1}(u^{-1}(N_{R}))\to H_{1}(\Sigma)$ is injective. For generic $R$, $\partial N_{R}$ is transverse to $u$ and $S_{R}:=u^{-1}(N_{R})$ is an arithmetic genus $0$ compact nodal Riemann surface with boundary. By the maximum principle, every irreducible component $S_{1}', \cdots, S_{l}'$ of closure $\bar{S_{R}}$ of $S_{R}$ is noncompact. Suppose that $|H_{1}(u^{-1}(\text{int}(N)))|>k-1$ then the genus of $\Sigma$ should be greater than 1, contradiction. So $|H_{1}(u^{-1}(\text{int}(N)))|\geq k-1$. $N$ is $(k,\mathcal{E})$-uniruled. \par
Let us show that $N$ is $(k, \mathcal{E}, \iota^{*}[\mho])$-uniruled. Let $f$ be a Morse function on $N$ with $\partial N$ as the highest regular level set of it. Choose a Morse function $h$ on $M$ satisfying the following.
\begin{itemize}
    \item $h|_{N}=f$,
    \item $h|_{M\setminus N}\geq \text{ max }(f)$, and
    \item $\partial M$ is the highest regular level set of $h$.
\end{itemize}
Then the descending manifolds of $f$ and $h$ representing $[\mho]$ and $\iota^{*}[\mho]$ respectively are the same. 
\qed

\begin{lemma} $[$Theorem $2.3$, \cite{mclean14}$]$ 
    Let $(N,\theta_{N})$, $(L,\theta_{L})$ be two Liouville domains so that the completions $\widehat{N}$, $\widehat{L}$ are symplectomorphic, under $\phi:\widehat{N}\to\widehat{L}$. If $N$ is $(k, \mathcal{E}, [\mho])$-uniruled, then $\exists \mathcal{E} '>0$ so that $L$ is $(k, \mathcal{E}', \phi^{*}[\mho])$-uniruled.
\end{lemma}

\proof By Lemma 1.1 in \cite{BEE12}, we can assume that $\phi$ is an exact syplectomorphism.  $\phi^{*}\theta_{L}=\theta_{N}+df$ for some function $f$. Let $\Phi_{t}:\widehat{N}\to\widehat{N}$ be the time $t$ flow of the $X_{\theta_{N}}$, where $\iota_{X_{\theta_{N}}}d\theta_{N}=d\theta_{N}$. Near infinity, $N$ has the cylindrical coordinate $(r,\theta)$ so that the vector field $X_{\theta_{N}}=r\frac{\partial}{\partial r}$. So there exists $T>0$ with $\phi^{-1}(L)\subset \Phi_{T}(N)$ as a codimension $0$ exact submanifold. Since the Liouville domain $\Phi_{T}(N)$ is $(k, e^{T}\mathcal{E}, \Phi_{T}^{*}[\mho])$-uniruled, $\phi^{-1}(L)$ is $(k, e^{T}\mathcal{E}, \iota ^{*} \Phi_{T}^{*}[\mho])$-uniruled by Lemma \ref{symplectic invariance of uniruled}. $L$ is also $(k, \mathcal{E}', \phi^{*}([\mho]))$-uniruled. 
\qed

\begin{corollary}\label{Liouville deform} $[$Corollary $3.3$, \cite{mclean14}$]$ 
Let $N$, $L$ be two Liouville deformation equivalent Liouville domains under $\phi$. If $N$ is $(k, \mathcal{E}, [\mho])$-uniruled, then $\exists \mathcal{E} '>0$ so that $L$ is $(k, \mathcal{E}', \phi^{*}([\mho]))$-uniruled.
\end{corollary}

\proof Their completions $\widehat{N},\widehat{L}$ are symplectomorphic. 
\qed \vspace{2mm}

\subsection{$J$-holomorphic $(k,\mathcal{E}, [\mho])$-uniruled $\underline{M}$ $\Rightarrow$ Such a $\overline{M}$}
The sandwich theorem \ref{sandwich} says that for any affine variety $M$, we can always find two Liouville equivalent Liouville domains $\underline{M}, \overline{M}$ with $\underline{M}\hookrightarrow M\hookrightarrow \overline{M}$, where two inclusions are exact symplectic embeddings of a codimension $0$ submanifold and a homotopy equivalence. As a corollary of corollary \ref{Liouville deform}, 

\begin{prop}
If $\underline{M}$ is $J$-holomorphic $(k,\mathcal{E}, [V])$-uniruled, then so is $\overline{M}$.
\end{prop}

The following lemma is used to prove the number of boundary components of rational curve in $F\setminus E$ is not increasing.

\begin{lemma}\label{no critical points} $[$Lemma $4.5$, \cite{mclean14}$]$
Let $J$ be an almost complex structure on $X$ which agrees with the standard complex structure on $X$ near $D$. Let $u:S\to X$ be a $J$-holomorphic map where $S$ is a compact nodal Riemann surface so that no component of $S$ maps into $D$ entirely. Let $f:M\to \R$ be a plusisubharmonic function associated to an ample line bundle $\mathcal{L}$ on $X$ where a section $s$ of $\mathcal{L}$ have $s^{-1}(0)=D$. Define $u':=u|_{S\cap u^{-1}(M)}$. Then there is a small neighborhood $ND$ of $D$ such that $f\circ u'$ has no critical points on $(f\circ u')^{-1}(ND\cap M)$.
\end{lemma}

\proof of Lemma \ref{no critical points}
We refer \cite{mclean14}. 
\qed

\subsection{$J$-holomorphic $(k,\mathcal{E}, [\mho])$-uniruled $\overline{M}$ $\Rightarrow$ Polynomial maps $(\P^{1}\setminus\Gamma) \to M$}
Using degeneration to the normal cone, McLean showed that $J$-holomorphic $(k,\mathcal{E})$-uniruledness implies algebaic $k$-uniruledness in \cite{mclean14}. In this section, we will generalize Theorem 2.5 in \cite{mclean14} for $(1,\mathcal{E}, [\mho])$-uniruledness to get algebraic stratified-uniruledness of an affine variety $M$. Let us state the main propositions as follows.

\begin{prop} \cite{mclean14} \label{JtoAlg}
If $\underline{M}\subset M\subset \overline{M}$ is $(1,\mathcal{E},[\mho])$-uniruled, then there exists a non-constant polynomial map $\bar{v}:\C\to M$ with $\overline{\text{image}(\bar{v})}\cap \overline{\mho}\neq \emptyset$ where $\overline{\mho}$ is the closure of $\mho$ in $M$.
\end{prop}

More generally,
\begin{prop}\label{make holomorphic}
$[$Lemma $4.6$, Theorem $2.5$, \cite{mclean14}$]$
If $\underline{M}\subset M\subset \overline{M}$ is $(k,\mathcal{E},[\mho])$-uniruled, then there exists a non-constant polynomial map $\bar{v}:\P^{1}\setminus 
\Gamma \to M$ with $\overline{\text{image}(\bar{v})}\cap \overline{\mho}\neq \emptyset$, where $\Gamma$ is the set of at most $k$ distinct points in $\P^{1}$.
\end{prop}

The proof in \cite{mclean14} will be provided for reader's convenience. The two main ingredients are deformation to the normal cone and Fish's target-local Gromov compactness results on a sequence of $J-$holomorphic disks with boundary in a family of symplectic manifolds. 

Let $v:\C \cong \P^{1}\setminus\{\infty\}\rightarrow M$ is a non-constant proper $J$-holomorphic map of finite energy with $[v:\C\to M]\cap [\mho]\neq \emptyset$. 

\subsubsection{McLean's Degeneration to the Normal Cone}
We review the degeneration-to-the-normal-cone construction, \cite{mclean14}.
The projective variety $X$ can be embedded into $\P^{N}$ so that $X\setminus M=X\cap \P^{N-1}$ is an effective ample divisor of $X$. Define,
\begin{itemize}
    \item $\Delta:=\{\infty\}\times \P^{N}+\P^{1}\times\P^{N-1}$, an ample divisor of $\P^{1}\times\P^{N}$.
    \item $Bl:Bl_{\{0\}\times\P^{N-1}}\P^{1}\times\P^{N}\to\P^{1}\times\P^{N}$, the blow up of $\P^{1}\times\P^{N}$ along $\{0\}\times\P^{N-1}$.
    \item $\widetilde{\Delta}:=Bl^{-1}(\Delta)$.
    \item $E:=Bl^{-1}(\{0\}\times\P^{N-1})$.
    \item $\pi:=pr_{1}\circ Bl:Bl_{\{0\}\times\P^{N-1}}\P^{1}\times\P^{N}\to \P^{1}$.
    \item $\Delta_{\infty}:=d\cdot \widetilde{\Delta}+(d-1)\cdot\pi^{-1}(\infty)$.
\end{itemize}
The fiber $\pi^{-1}(0)=F+E$ is linearly equivalent to $\pi^{-1}(\infty)$. For $d\gg 0$, $d\cdot \widetilde{\Delta}+(d-1)\cdot E=d\cdot \widetilde{\Delta}+(d-1)\cdot\pi^{-1}(\infty)-(d-1)\cdot F$ is ample in $Bl_{\{0\}\times\P^{N-1}}\P^{1}\times\P^{N}$.  The associated line bundle $\mathcal{L}_{d\cdot  \widetilde{\Delta}+(d-1)\cdot E}$ on $Bl_{\{0\}\times\P^{N-1}}\P^{1}\times\P^{N}$ admits a metric $||\cdot||$ where the curvature form is a positive $(1,1)$-form inducing a symplectic form on $X$.
 \begin{itemize}
     \item Let $s$ be a meromorphic section of it with $s^{-1}(0)-s^{-1}(\infty)=d\cdot \widetilde{\Delta}+(d-1)\cdot E$.
     \item $\phi:=(-log||s||)^{-1}$ satisfying $\phi|_{\text{supp}(-(d-1)\cdot F)=F}=-\infty$, $\phi|_{\text{supp}\Delta_{\infty}}=\infty$.
     \item $N_{c}:=(-log||s||)^{-1}((-\infty,c])\cup F$, a compact submanifold of $Bl_{\{0\}\times\P^{N-1}}\P^{1}\times\P^{N}\setminus \text{supp}(\Delta_{\infty})$ for generic $c\gg 1$.
     \item $\widetilde{\P^{1}\times X}:=Bl^{-1}(\P^{1}\times X)$
     \item $\pi_{X}:=\pi|_{\widetilde{\P^{1}\times X}}$
     \item $M_{x}:=\pi_{X}^{-1}(x)\setminus \text{supp}(\Delta_{\infty})$
 \end{itemize}

\begin{figure}
    \centering
    \includegraphics[width=7cm, height=5cm]{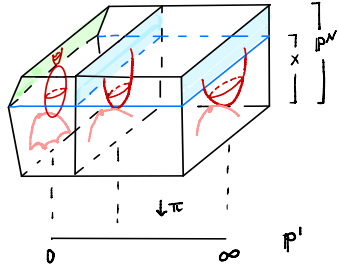}
    \caption{Degeneration to the normal cone}
    \label{fig:normal cone}
\end{figure} 

Then $M_{x}$ is isomorphic to $M$ when $x\neq 0$ and $(M_{x}, -dd^{c}log||s||)$'s are symplectomorphic. Note that $\widetilde{\P^{1}\times X}\setminus (\text{supp}(\Delta_{\infty})\cup E)\cong \C\times M$.

Choose a sequence of numbers $x_{i}\in \C\setminus \{0\}$ conversing to $0$ as $i\to \infty$. Choose a sequence of Morse functions $f_{x_{i}}$ on $M_{x_{i}}$ for each $i$, satisfying that all the Morse critical points are contained in $N_{c}\cup M_{x_{i}}$ and  $C^{0}$-converging to a Morse function $f_{0}$ on $F\in M_{0}$ (Choose a Morse function $f$ on $M$ and pull it back to $\widetilde{\P^{1}\times X}\setminus ((\text{supp}(\Delta_{\infty})\cup E))$ then restrict to each fiber over $x_{1}$). Then, since $M_{x}$ are isomorphic to $M$, for a given cohomology class $[\mho]\in H^{m}(M)$, we have a sequence of unstable submanifolds $\mho_{f_{x_{i}}}$ of $f_{x_{i}}$ converging to $\mho_{f_{o}}$ representing $[\mho]$. For each $i$, choose a Liouville domain $N_{x_{i}}$ which is an exact codimension $0$ symplectic submanifold of $M_{x_{i}}$ containing $N_{c}\cap M_{x_{i}}$ and the embedding $N_{x_{i}}\hookrightarrow M_{x_{i}}$ is a homotopy equivalence. Since $M_{x_{i}}$ is $(k, \lambda, \mho)$-uniruled, $N_{x_{i}}$ is $(k, \lambda', \mho)$-uniruled for some $\lambda'>0$ by lemma \ref{symplectic invariance of uniruled}.

For each $i$, there is a proper $J$-holomorphic map $u_{x_{i}}:S_{x_{i}}\to int(N_{x_{i}})$, where $S_{x_{i}}$ is at most $(k-1)$-punctured sphere, $int(N_{x_{i}})$ is an interior of $N_{x_{i}}$ and $u_{x_{i}}$ has energy at most $\lambda'$. $u_{x_{i}}|_{u_{x_{i}}^{-1}(int(N_{c}\cap M_{x_{i}}))}$'s are properly embedded holomorphic curves inside the interior of the compact manifold $N_{c}$ because almost complex structures become integrable in $int(N_{c}\cap M_{x_{i}})$.

Let $Q:=Bl_{\{0\}\times \P^{N-1}}\P^{1}\times P^{N}\setminus \text{supp}(\Delta_{\infty})$ be a quasi-projective variety. Then $N_{x_{i}}\subset \pi|_{Q}^{-1}(x_{i})$ and $u_{x_{i}}:S_{x_{i}}\to int(N_{x_{i}})\subset \pi|_{Q}^{-1}(x_{i})$. Let $V:=Q\cap \widetilde{\P^{1}\times X}$ be a closed variety. Choose real compact codimension $0$ submanifolds with boundaries $N_{1}, N_{2}$ of $N_{c}$ with $F\subset \text{int}(N_{1})\subset N_{1}\subset\text{int}(N_{2})\subset N_{2}\subset N_{c}$. After perturbing, we can assume that $\partial N_{1}$, $\partial N_{2}$ are transverse to $U_{x_{i}}$ for all $i$.

\subsubsection{Fish's Target-local Gromov Compactness}
To get a rational curve on an affine variety from a sequence of $J$-holomorphic disks in the family of Liouville domains, Fish's generalized Gromov compactness result will be used. Sequences of $J$-holomorphic curves with an unbounded number of free boundaries in families of degenerating target manifolds without uniform boundedness of energy still have a convergent subsequence of sub-$J$-holomorphic curves when we carve the target out near the boundary of target properly and restrict $J$-holomorphic curves to the preimage of carved target (target-local) \cite{fishcpt11}.
\proof of Proposition \ref{make holomorphic} By the Gromov-Fish compactness theorem \cite{fishcpt11} to $N_{2}$ and $u_{x_{i}}|_{u_{x}^{-1}(N_{2})}$, we have a sequence of compact subcurves $\underline{S_{i}}\subset S_{x_{i}}$ satisfying 
\begin{enumerate}
    \item $u_{x_{i}}(\partial \underline{S_{i}})\subset N^{c}_{1}$.
    \item There is a compact surface $\underline{S}$ with boundary and a sequence of diffeomorphisms $a_{i}:\underline{S}\to\underline{S_{i}}$ such that $u_{x_{i}}\circ a_{i}$ is $C^{0}$-converges to some continuous map $v':\underline{S}\to \pi^{-1}(0)=F\cup E$ which is smooth away from some union of curves $\Gamma$ in int$\underline{S}$ (decoration in the symplectic field theory) and $u'_{x_{i}}\circ a_{i}$ $C^{\infty}_{\text{loc}}$-converges to $v'$ outside of $\Gamma$.
    \item The map $v'$ is equal to $v''\circ \psi$ where $\psi :\underline{S}\to S$ is a continous surjection to a nodal Riemann surface $S$ with boundary and a diffeomorphism onto its image away from $\Gamma$ and $v''$ is a holomorphic map from $S$. The map $\psi$ sends the curves $\Gamma$ to the nodes of $S$ and the map $v''$ sends $\partial S$ to the complement of $N_{1}$.    
\end{enumerate}

Since $v''(\partial S)\subset N^{c}_{1}$, $F\subset \text{int}(N_{1})$, $v''(S)\cap \mho_{f_{0}}\neq \emptyset$ and $\mho_{f_{x_{i}}}$ converges to $\mho_{f_{0}}\subset F$, there is an irreducible component $K\cong \P^{1}$ of $S$ where $v''$ maps $K$ to $F$. Choose $v=v''|_{K}:\P^{1}\to F$. Then $v$ is an holomorphic map from $\P^{1}$ to a projective space $F$, so it is an algebraic map satisfying $v(\P^{1})\cap [\mho_{f_{0}}]\neq \emptyset$.
Show that $|v^{-1}(F\cap E)|\leq k$. By Lemma \ref{no critical points}, a exhausting plurisubharmonic function $f$ $v''(K)$ is smooth outside of a compact set. So we can choose $C\gg 1$ so that $f^{-1}(C)$ is transverse to $v$ and $(f\circ v)^{-1}(C)$ is a disjoint union of $m$ circles and $(f\circ v)^{-1}((-\infty,C])$ is connected for $C\gg 1$. Since the maps $u'_{x_{i}}\circ a_{i}$ $C^{\infty}_{\text{loc}}$-converges to $v'$ outside of $\Gamma$, $u_{x_{i}}\circ a_{i}$ is $C^{0}$-converges to some continuous map $v'$, the connected component $S'_{i}$ if $(f\circ u_{x_{i}})^{-1}((-\infty,C])$ passing through $\mho_{x_{i}}$ has $m$ boundary components for $i\gg 1$. By Lemma \ref{symplectic invariance of uniruled}, $S'_{i}$ has at most $k$ boundary circles for $i\gg 1$. Therefore, $v^{-1}(E)$ is a union of at most $k$ points. Therefore, we get a non-constant algebraic map $\bar{v}:\C\to M$ with $[\bar{v}:\C\to M]\cap [\mho]\neq \emptyset$.
\qed

\subsection{Family of Rational Curves $\Rightarrow$ Algebraic Variety}
We will use Grothendieck's construction on Hilbert schemes to get a scheme parametrizing a family of rational curves. Grothendieck's philosophy was to identify a scheme $X$ with the functor represented by $X$, from the category of commutative rings to the category of sets, defined by $R\mapsto Hom(SpecR,X)$ \cite{grothendieck61}, \cite{kollar96}, \cite{lurie}. Let us recall some theorems in \cite{kollar96} without proof.

\begin{defn}[Definition I.1.1, \cite{kollar96}]
    Let $S$ be a scheme and $F$ a contravariant functor $F:\{schemes/S\}\to\{sets\}$. A pair $(X(F),U(F))$ of a scheme $X(F)/S$ and an element $U(F)\in F(X(F))$, called the universal element, represents $F$ if for every $S$-scheme $Z$, $\text{Hom}_{S}(Z,X(F))\to F(Z/S)$, defined by $g\mapsto g^{*}U(F)$, is an isomorphism.
 \end{defn}

\begin{defn}[Definition I.1.3, \cite{kollar96}]]
    Let $X/S$ be a scheme. The Hilbert functor $\mathcal{H}ilb(X/S):\{schemes/S\}\to\{sets\}$ is defined by
    $$\mathcal{H}ilb(X/S)(Z):=\left\{
    \begin{array}{c}
         \text{Subschemes}\\
         V\in X\times_{S}Z\\
         \text{proper,flat/Z}
    \end{array}\right\}$$
\end{defn}

\begin{prop}[Theorem I.1.4, \cite{kollar96}]
Let $X/S$ be a projective scheme, $\mathcal{O}(1)$ a relatively ample line bundle and $P$ a polynomial. The functor $\mathcal{H}ilb_{P}(X/S)$ is represented by a morphism $\text{Univ}_{P}(X/S)\xrightarrow{u}\text{Hilb}_{P}(X/S)$ where $\text{Univ}_{P}(X/S)\subset X\times_{S}\text{Hilb}_{P}(X/S)$.
\end{prop}

\begin{defn}
    Let $X/S$ and $Y/S$ be schemes. $\mathcal{H}om(X,Y)$ is the functor defined by\\
    $\mathcal{H}om_{S}(X,Y)(T):=\{T\text{-morphisms}:X\times_{S}T\to Y\times_{S}T\}$
\end{defn}
\begin{prop}[Theorem I.1.10, \cite{kollar96}]
Let $X/S$ and $Y/S$ be projective schemes over $S$. Assume that $X$ is flat over $S$. Then $\mathcal{H}om_{S}(X,Y)$ is represented by an open subscheme $Hom_{S}(X,Y)\subset Hilb(X\times_{S}Y/S)$.  
\end{prop}
By choosing $S=Spec(\C)$, $X=\P^{1}$, $F$ a scheme over $Spec(\C)$,
\begin{corollary}
\begin{enumerate}
    \item A functor $\mathcal{H}om_{Spec(\C)}(\P^{1},F)=\left\{T\mapsto \left(\begin{tikzcd} [cramped, row sep=0.6em, column sep=1.5em]
\P^{1}\times T \arrow{dr}{pr} \arrow{rr}{f} &                & F\times T \arrow{dl}{pr}\\
                              & T
\end{tikzcd}\right)\right\}$\\
    is represented by an open subscheme $Hom(\P^{1},X)$ over $Spec(\C)$
    \item $Hom_{Spec(\C)}^{\text{nonconstant}}(\P^{1},F):=\{f|_{\P^{1}\times\{y\}}$ is non-constant $\forall y\in Y\}=Hom(\P^{1},F)\setminus [F]$, where $[F]\in Hom(\P^{1},F)\subset Hilb(\P^{1}\times_{\text{Spec}(\C)}F)\cong Hilb(F)$ is the point corresponding to $F$.
\end{enumerate}
\end{corollary}

\begin{prop}[Def.-Prop.II.2.11 \cite{kollar96}]
RatCurves$^{n}_{d}(X/S)$ is quasi projective over $S$ for every $d$. 
\end{prop}
Since the evaluation map $ev_{0}:Hom_{Spec(\C)}^{\text{nonconstant}}(\P^{1},F)\to F$ is a morphism defined by
$$ev_{0}\left(\begin{tikzcd}[row sep=0.8em]
\P^{1}\times T \arrow{dr}{pr} \arrow{rr}{f} &                & F\times T \arrow{dl}{pr}\\
                              & T
\end{tikzcd}\right)=(f|_{\{0\}\times T}:T\to F)\in Hom(T,F)$$
is represented by a variety,

\begin{prop} \label{family}
The image of the evalutation map $ev_{0}:Hom_{Spec(\C)}^{\text{nonconstant}}(\P^{1},F)\to F$ is a countable union of subvarieties.
\end{prop}

\subsection{Proof of the Main Theorem}
We are ready to prove our main theorem. Let $M$ be an affine variety with the standard symplectic structure and a complex structure $J$. Suppose that $[\mho]\in H^{m}(M)$ be in the image of the map $\delta$, where $m=2k$ or $2k+1$ for some $k\in \N$.
$$\begin{tikzcd} [cramped, sep=small]
\cdots \to SH^{m-1}(M)\arrow[r] & SH^{m-1}_{+}(M)\arrow[r,"\delta"] & H^{m}(M)\arrow[r,"\iota"] & SH^{m}(M)\to \cdots.
\end{tikzcd}
$$

By the correspondence theorem of Bourgeois-Oancea's Morse-Bott symplectic cohomology [Theorem 3.7, \cite{bourgeois.oancea09}], there exist a open-dense subset of the set of admissible pairs $(H,J)$ so that there exist a cascade $(\nu, v) \in \mathcal{M}^{A}_{1}(\gamma,q_{\mho};H,\{f_{x}\},J)$, where $\nu$ is a Morse trajectory limiting to a critical point $q_{\mho}$ of $f$ and $v$ is a $J$-holomorphic curve. Since $J|_{\text{int}\underline{M}}$ is integrable, $v|_{\text{int}\underline{M}}$ is holomorphic. By the maximum principle, the number of boundary component of $v\cap \underline{M}$ is $1$. 

For any $J$ and any $f$, there exist a non-constant proper $J$-holomorphic map $v:\D\to \underline{M}$ of finite energy with image$(v:\D\to \underline{M})\cap \mho_{f}\neq \emptyset$ and $\partial\D\subset \partial \underline{M}$. Therefore, $\underline{M}$ is $(1,\mathcal{E}, [\alpha])$-uniruled. By proposition \ref{JtoAlg}, we get a polynomial map $\bar{v}:\C\to M$ or a rational curve $\bar{\bar{v}}: \P^{1}\to X:=M\cup D$ with $\bar{\bar{v}}^{-1}(D)={\infty}$ satisfying image$(\bar{\bar{v}}:\C\to X)\cap \overline{\mho^{\alpha}_{f}}\neq \emptyset$. Then by proposition \ref{family}, a family of all such rational curves(projective lines) forms a uniruled subvariety $\overline{\Xi}$ of $X$. By Thom transversality for generic pair $(f,g)$ of a Morse function on $X$ with metric $g$, $\overline{\Xi}$ intersect transversely with each Whitney strata of the closure of $\mho^{\alpha}_{f}$ by induction on the dimension. Therefore $\overline{\Xi}$ is of dimension at least $n-k$. Since each $\bar{\bar{v}}$ satisfies $\bar{\bar{v}}^{-1}(D)={\infty}$, $\Xi:=\overline{\Xi}\cap M$ is a $\C$-uniruled subvariety of $M$ of dimension at least $n-k$. We are done.
\qed

\section{Applications}
In this section, we explain the applications of two main theorems. Let $M$ be a smooth affine variety.
\subsection{Main Application: $\ell(M) \Rightarrow$ Uniruled Subvarieties}

\begin{corollary} [\ref{main theorem1}] \label{4.1}
If $\ell(M)=2k$ or $2k+1$ ($0\leq k <n$), then $M$ contains a $2(n-k)$-dimensional family of affine lines, where an affine line means the image of $\C$ in $M$ under a nonconstant rational map.
\end{corollary}

\begin{corollary} \label{cor1} If $\ell(M)=0$ or $1$, any projective variety $X$ compactifying $M$ is uniruled. 
\end{corollary}

\begin{corollary}\cite{zhou19} \label{SH=0}
If $SH^{*}(M)=0$, then $\ell(M)=0$ and $M$ is $\C$-uniruled. 
\end{corollary}


\begin{corollary}
If $SH^{*}(M,\underline{\Lambda}_{\tau \beta})=0$ for some $\beta\in H^{2}(M)$, then $M$ is $\C$-uniruled. 
\end{corollary}

\begin{example} 
\begin{enumerate}
\item $SH^{*}(\C^{n})\cong SH^{*}(\{\text{a point}\})=0$ and $\C^{n}$ is $\C$-uniruled.
\item By theorem \ref{flexible}, flexible Weinstein domains are $\C$-uniruled if they are affine varieties.
\item By the K\"unneth formula, $SH^{*}(\C\times M)=0$ and, indeed, $\C\times M$ is $\C$-uniruled.
\item (Example, p.1043, \cite{ritter14}) Let $X$ be a monotone closed symplectic manifold $(X, \omega)$, $c_{1}(X)=\lambda \omega$ for $\lambda >0$. Let $M$ be the total space of canonical line bundle $\mathcal{K}_{X}$ over $X$. Since $c_{1}(T_{M})=c_{1}(T_{X})+c_{1}(\mathcal{L})$ and $c_{1}(T_{X})=-c_{1}(\mathcal{K}_{X})$, $SH^{*}(M)=0$. Therefore, $M$ is $\C$-uniruled.
\item Let $X$ be a smooth projective variety. Assume that $D$ is a divisor of $X$ satisfying that $-\mathcal{K}_{X}-D$ is ample. Then $SH^{*}(\text{Tot}((\mathcal{K}_{X}+D)|_{D}\to D))=0$.
\end{enumerate}
\end{example}

\subsubsection{Uniruled Subvarieties for Pairs $(X,D)$}
Let $X$ be a projective variety, $D$ be an ample divisor, $K$ be a compact subset of $X\setminus D$. Let $SH^{*}_{D,a,b}(X)$, (resp., $SH^{*}_{D,a,b}(K\subset X)$) denote $SH^{*}_{a,b}(X\setminus D)$ (resp., $SH^{*}_{a,b}(K\subset X\setminus D)$). Then $\widehat{\ell}(X,D):=\widehat{\ell}(X\setminus D)$ detects uniruled subvarieties of $X$ with respect to given ample divisor. By the long exact sequence, 
$$\begin{tikzcd} [cramped, sep=small]
\cdots \to SH^{m-1}(X\setminus D)\arrow[r]\arrow[d,equal] & SH^{m-1}_{+}(X\setminus D)\arrow[r,"\delta_{D}"]\arrow[d,equal] & H^{m}(X\setminus D)\arrow[r,"\iota"]\arrow[d,equal] & SH^{m}(X\setminus D)\arrow[d,equal] \to \cdots.\\
\cdots \to SH^{m-1}_{D}(X) \arrow[r] & SH^{m-1}_{D,+}(X)\arrow[r,"\delta_{D}"] & H^{m}(X\setminus D)\arrow[r,"\iota"] & SH^{m}_{D}(X) \to \cdots.
\end{tikzcd}$$
Using the notation, we can also define a stronger measurement for the dimension of uniruled subvariety a projective variety $X$. $\ell(X):=$min$\{\widehat{\ell}(X,D):\forall \text{ ample divisor } D\}$.

As a corollary of the main theorem, we can detect uniruled subvarieties of a projective variety $X$ by considering all the ample divisors of $X$: $\bigcup_{D^{\text{ample}}}\bigcup_{\alpha\in \delta_{D}}\overline{\Xi}_{\alpha}$, where $\overline{\Xi}_{\alpha}$ is a uniruled subvariety of $X$ corresponding to $\alpha \in \delta_{D}$.

A projective variety is called to be rationally-connected if for a generic pair of two points $(x_{1},x_{2})\in X\times X$, there exist a rational curve $\P^{1}\to X$ passing through $x_{1}, x_{2}$. For example, every Fano variety is rationally-connected, \cite{kollar.miyaoka.mori92}. Rationally-connectedness of projective varieties is conjectured to be symplectic deformation invariant by Koll\'{a}r and it was proven in the case of dimension $3$, \cite{voisin08}, \cite{tian10}. When a projective variety $X$ has a rationally-connected ample divisor, we can apply $\ell(X, D)$ to detect rationally-connectedness as follow. 

\begin{corollary}
\begin{enumerate}
    \item Let $X$ be a projective variety. Suppose that $X$ has a rationally-connected ample divisor $D$ (each strata of $D$ is rationally-connected) and $\ell(X, D)=0$. Then, $X$ is rationally-connected.
    \item Let $X$ be a rationally-connected projective variety. Suppose that $X$ has rationally-connected ample divisor. Then $\ell(X, D)=0$.
    \item Let $X$ be a projective variety with a rationally-connected ample divisor $D$. If $\ell(X, D)=0$, then $X$ is rationally-connected.
\end{enumerate}
\end{corollary}
However, having a rationally-connected ample divisor is a very strong condition and it is not known that every rationally-connected projective variety has a rationally-connected ample divisor which was conjectured by Fano. \footnote{The author learned the conjecture from Jason Starr}  

\subsection{Variance of $\ell(M)$ under Symplectic Surgeries}

\begin{theorem} \label{1h-attach}
Let $\underline{M}$ be the associated Liouville domain obtained by intersection of $M$ and a large $2n$-ball. Assume that $\underline{M}$ is of complex dimension equal to or bigger than $3$. Suppose that $\underline{M}$ is a connected Liouville domain with a Weinstein $1$-handle attached, then $M$ is $\C$-uniruled.
\end{theorem}

\begin{proof}
By lemma \ref{sandwich}, given an affine variety $M$, we can construct two Liouville domains $\underline{M}, L$ with $\underline{M}\hookrightarrow M\hookrightarrow L$, where $\underline{M}, L$ are Liouville deformation equivalent. Suppose that the associated Liouville domain $\underline{M}$ of $M$ is Weinstein $1$-handles attached so that it has $1$-cycle representing a non-trivial cohomology class $[\alpha]$. Since $\underline{M}, L$ are Liouville deformation equivalent, we consider $0\neq [\alpha]\in H^{1}(L)$. 
\begin{center}
\begin{tikzcd}
\cdots \to SH^{0}(L)\arrow[r] \arrow[d,"\cong"] & SH^{0}_{+}(L) \arrow[r,"\delta_{L}"] \arrow[d,"C_{+}^{M,L}"] & H^{1}(L) \arrow[r,"\iota_{L}"] \arrow[d,"C^{M,L}"] &SH^{1}(L)\arrow[d,"\cong"]\to \cdots \\  
\cdots \to SH^{0}(M)\arrow[r]                   & SH^{0}_{+}(M) \arrow[r,"\delta_{M}"]           & H^{1}(M) \arrow[r,"\iota_{M}"]                     &SH^{1}(M)\to \cdots 
\end{tikzcd}
\end{center}
Since the cohomology classes of affine varieties are in even degree, $C_{M,L}([\alpha])=0$. So By the commuting diagram, $\iota_{L}([\alpha])=0$. By the exactness, there exist $\gamma\in SH^{0}_{+}(L)$ such that $\delta_{L}(\gamma)=[\alpha]$. Therefore, $\ell(\underline{M})=l(M)=1$. By \ref{main theorem1}, $M$ is $\C$-uniruled.
\end{proof}
More generally, in a category of Weinstein manifolds, we consider attaching flexible Weinstein handles in \cite{weinsteinhandle}, \cite{murphy12}, \cite{cieliebakeliashberg12}, \cite{cieliebakeliashberg14}, which does not change symplectic cohomology [Theorem 5.6, \cite{BEE12}]. 
\begin{theorem} \label{flexiblehandle}
Let $W$ be a Weinstein manifold of dim$_{\R}W=2n$ with $\ell(W)=\infty$. Suppose that we have a Weinstein manifold $W_{\Supset k}$, obtained by attaching flexible $k$-handles to $W$ so that $\text{rank }H_{k}(W_{\Supset k})>\text{rank }H_{k}(W)$. Then $\ell(W_{\Supset k})=2n-k$. Hence, if $W_{\Supset k}$ is symplectomorphic to an affine variety $M$, then $M$ admits a $\C$-uniruled subvariety of complex dimension $n-k$.
\end{theorem}
\begin{proof}
There is a exact triangle under the Legendrian surgery,
$$\cdots\to LH(_{\Supset k})\to SH(W_{\Supset k})\to SH(W) \to LH(_{\Supset k})\to \cdots,$$
where $_{\Supset k}$ denotes a Weinstein cobordism of $k$-handle attachment. If $_{\Supset k}$ is flexible, then $ LH(_{\Supset k})=0$. Therefore, $SH(W_{\Supset k})\cong SH(W)$. However, $H^{k}(W_{\Supset k})>H^{k}(W)$. We get $\ell(W_{\Supset k})=2n-k$. Moreover, if $W_{\Supset k}$ is symplectomorphic to an affine variety $M$, then $M$ admits a $\C$-uniruled subvariety of complex dimension $n-k$ by the theorem \ref{main theorem0}.
\end{proof}

\begin{theorem}
Let $W$ be a Weinstein manifold of dim$_{\R}W=2n$ with $\ell(W)=\infty$. Suppose that we have a Weinstein manifold $W_{\Supset k}$, obtained by attaching flexible $k$-handles to $W$ so that $\text{rank }H_{k}(W_{\Supset k})>\text{rank }H_{k}(W)$. Let $\phi$ be a symplectomorphism on $M$ and $\tilde{\phi}$ be a symplectomorphism on $W_{\Supset k}$ extending $\phi$ and being identity on the $k$-handles. Then $\ell_{p.p}((W_{\Supset k})_{\phi})=2n-k$. Hence, if $(W_{\Supset k})_{\phi}$ is an affine variety, then $(W_{\Supset k})_{\phi}$ has $\C^{*}$-uniruled subvariety of complex dimension $n-k$.
\end{theorem}

\begin{remark}
To apply the theorem above for affine varieties, we need to understand the condition when $W_{\Supset k}$ is symplectomorphic to an affine variety. One of obstructions for a Weinstein manifold to be an affine variety is the growth rate of symplectic cohomology in \cite{mclean12}.
\end{remark}

\subsection{Examples of Affine Varieties with Nontrivial $\ell(M)$}

We will construct affine varieties which have isomorphic symplectic cohomologies but non-isomorphic (co)homologies. It leads us to construct birational affine varieties with distinct quantities $\ell (M)$, i.e., with different stratification level of $\C$-uniruledness. We use an algebraic operation, called the Kaliman modification, which is also a symplecic end-connected sum. First, let us work in dim$_{\C}M=3$. Let $D:=\{xyz=0\}\subset \C^{3}$. Choose a smooth point, $p_{1}:=(1,0,0)$, $p_{2}:=(2,0,0)\in D$. Let $\mathcal{B} l_{\{p_{1},p_{2}\}}\C^{3}$ be the blow-up of $\C^{2}$ at two distinct points $p_{1},p_{2}$ and $\tilde{D}$ be the proper transform of $D$. Define $K:=\mathcal{B} l_{\{p_{1},p_{2}\}}\C^{3}\setminus \tilde{D}$, the Kaliman modification of $(\C^{3},D,\{p_{1},p_{2}\})$\cite{kaliman}. Then the symplectic cohomologies of $(\C^{*})^{3}=\C^{3}\setminus D$ and $K$ are isomorphic by considering their Lefschetz fibration structures: $SH^{*}(\mathcal{B}l_{p}M\setminus \widetilde{(Z\cap M)})\cong SH^{*}(M\setminus Z)$ in case $c_{1}(\mathcal{B} l_{p}M\setminus \widetilde{(Z\cap M)})=c_{1}(M\setminus Z)=0$ [See \ref{marklef}]. On the other hand, their topological structures are different. So we get non-trivial class $P.D.[\C^{2}]\in H^{2}(K)$ with $\delta(P.D.[\C^{2}])\neq 0 \in SH^{3}_{+}(K)$ from the following diagram. 

\begin{center}
\begin{tikzcd} [cramped, sep=small]
 0 \arrow[d]                              &                                                   & P.D.[\C^{2}]\neq 0 \arrow[hookrightarrow]{d}    &     0 \arrow[d]\\
\cdots\to SH^{1}(K)\arrow[r] \arrow[d]          & SH^{1}_{+}(K) \arrow[r] \arrow[d]  & H^{2}(K) \arrow[r]\arrow[d]  & SH^{2}(K) \arrow[d] \to\cdots\\
\cdots\to SH^{1}((C^{*})^{3})\arrow[r]\arrow[d] & SH^{1}_{+}((C^{*})^{3}) \arrow[d] \arrow[r] & H^{2}((C^{*})^{3}) \arrow[r]\arrow[d] & SH^{2}((C^{*})^{3})\to \arrow[d] \cdots\\
0 & 0 & 0 & 0
\end{tikzcd}
\end{center}

Therefore, $\ell(K)=2$ and $K$ has real $4$-dimensional family of affine lines, $\C^{2}$ indeed. 

\begin{theorem} \label{construction1}
For each $n\geq 3$, there is an affine variety $M^{2n}_{2}$ of dim$_{\C}M=n$ with $\ell(M^{2n}_{2})=2$.
\end{theorem}

\begin{proof}
We use the construction above. For any $n\geq 3$, define $K_{n}:=\mathcal{B}l_{\{p_{1},p_{2}\}}\C^{n}\setminus \tilde{D_{n}}$, the Kaliman modification of $(\C^{n},D_{n},\{p_{1},p_{2}\})$, where $D_{n}:=\{z_{1}z_{2}\cdots z_{n}=0\}\subset \C^{n}$. Then $\ell(K_{n})=2$. Actually, $\ell(K_{m}\times (\C^{*})^{n-m})=2$. 
$\ell(K_{n}\times K_{m})=2$ with dim$_{\C}=n+m$ for any $n,m$.
In general, 
$$\ell \bigg(\prod^{k}_{i=1} K_{n_{i}} \bigg)=2 \text{ with } \text{dim}_{\C}\bigg(\prod^{k}_{i=1}K_{n_{i}}\bigg)=\sum^{k}_{i=1}n_{i}.$$
\end{proof}

\subsection{Towards the Log Minimal Model Program}

\subsubsection{Cylindrical Affine Varieties}
An affine variety $M$ is called cylindrical if it contains  a dense principal Zariski open subset $U=M\setminus (f=0)$, for some $f\in \mathcal{O}(X)$, isomorphic to $\C\times M'$ for an affine variety $M'$. We call such $U$ a cylinder [Definition 3.1.4 \cite{kishimoto.prokhorov.zaidenberg11}]. By definition, a cylindrical affine variety is $\C$-uniruled. We show that a cylindrical affine variety has vanishing symplectic cohomology.

\begin{theorem} \label{cylindrical}
Let $M$ be an cylindrical affine variety having a dense affine open subset $U=M\setminus (f=0)$ for smooth $f$. Then $SH^{*}(\underline{M}\subset M)=0$. Therefore $\widehat{\ell}(M)=0$ and $M$ is $\C$-uniruled.
\end{theorem}

\begin{proof} of \ref{cylindrical}. By definition of being cylindrical, there exists a hypersurface $Y\subset M$ so that $M\setminus Y\cong \C\times M'$ for some affine variety $M'$. By the K\"unneth formula, $SH^{*}(\underline{M\setminus Y}\subset M\setminus Y)=0$. We can choose a small tubular neighborhood $\mathbb{Y}$ of $Y\cap \underline{M}$ in $\underline{M}$ and $\mathbb{Y}^{c}$ a small neighborhood of $\underline{M}\setminus Y$ in $\underline{M}$ so that $\mathbb{Y}$ is stably displaceable (Proposition \ref{stablydisplconj}) and the boundaries are disjoint, $\partial \mathbb{Y}\cap \partial \mathbb{Y}^{c}=\emptyset$. By proposition \ref{invariance3}, $SH^{*}(\mathbb{Y}\subset M)=0$ and $SH^{*}(\mathbb{Y}\cap \mathbb{Y}^{c}\subset M)=0$. By Varolgunes-Mayer-Vietoris sequence \ref{mv}, 
$$\begin{tikzcd} 
\lra{\delta} SH^{*}(\underline{M}\subset M)\to SH^{*}(\mathbb{Y}\subset M)\oplus SH^{*}(\mathbb{Y}^{c}\subset M)\to SH^{*}(\mathbb{Y}\cap \mathbb{Y}^{c}\subset M) \to \cdots
\end{tikzcd}$$
we have $SH^{*}(\underline{M}\subset M)\cong SH^{*}(\mathbb{Y}^{c}\subset M)$. By corollary \ref{spectralseqconj1}, there exists a spectral sequence converging to $SH^{*}(\mathbb{Y}^{c}\subset M)$ whose $E_{1}$-page is $SH^{*}(\mathbb{Y}^{c}\subset M\setminus Y)= SH^{*}\big(\underline{M\setminus Y}\subset M\setminus Y\big)=0$. Therefore, $SH^{*}(\underline{M}\subset M)\cong SH^{*}(\mathbb{Y}^{c}\subset M)=0$.
\end{proof}
\begin{remark}
If a Liouville domain $\underline{M}$ has a global Hamiltonian $S^{1}$-action that compatible with periodic orbits on the cylindrical ends, then symplectic cohomology of $M$ vanishes. On the other hand, an irreducible affine variety $M$ is cylindrical if and only if it has an effective $\mathbb{G}_{a}$-action [Proposition 3.1.5, \cite{kishimoto.prokhorov.zaidenberg11}].
\end{remark}

\subsubsection{Log Minimal Model Program on Affine Threefolds}
Kishimoto defined a half-point attachment (Definition \ref{halfpointattach}) to a normal quasi-projective threefold as a weighted blow-ups and describe the $\#$- minimal model program restricted to each birational step happening outside of divisors and extend it to log minimal model program \cite{kishimoto06}, \cite{kishimoto}. Let $M$ be a smooth affine threefold and $X$ be a smooth projective threefold compactifying $M$ with a nef divisor $D$. Assume that the complete linear system $|D|$ contains a smooth member. Then we can construct a sequence of birational maps of pairs $(X^{i}, D^{i})$ satisfying the following conditions [Theorem 1.2, 3.1, 4.1, \cite{kishimoto06}].
$$\begin{tikzcd}
(X,D)\arrow[r, dashed, "\phi^{1}"] & (X^{1},D^{1})\arrow[r, dashed, "\phi^{2}"]& \cdots \arrow[r, dashed, "\phi^{s}"] & (X^{s},D^{s})\cdots \arrow[r, dashed, "\phi^{\#}"] & (X^{\#},D^{\#}).
\end{tikzcd}$$
\begin{enumerate} \label{kishimoto model}
\item $X^{i}$'s are normal projective threefolds with only $\Q$-factorial terminal singularities. Each $D^{i}$ is the proper transform of $D$ on $X^{i}$ with a nef, Cartier linear system $|D^{i}|$ on $X^{i}$ with smooth member. 
\item (Flips) For $0\leq i\leq s$, the exceptional set(either an exceptional divisor or a union of the flipping curves) of the birational map $\phi^{i}$ is contained in $D^{i-1}$. In case where $\phi^{i}$ is a flip, the resulting flipped curves are contained in $D^{i}$. Furthermore, $M^{i}\cong M$.
\item (Terminal Divisorial Contractions) For $s<i<\#$, the birational maps $\phi^{i}$ contracts the exceptional divisor to a smooth point $q^{i}$ and $\phi^{i}$ is the weighted blow-up at $q^{i}$ with weights $(1,1,b^{i})$ for some $b^{i}\in \N$ with $1\leq b^{s+1}\leq b^{s+1}\leq \cdots \leq b^{\#}$. $X^{i}$ is the half-point attachment to $X^{i+1}$ of type $(b^{i},k^{i})$ for some $1\leq k^{i}\leq b^{i}$ unless $M^{i}\cong M^{i+1}$.
\item (Terminal Object according to the log Kodaira dimension $l.k.d(M)$)
\begin{enumerate}
    \item If $l.k.d(M)=-\infty$, then $X^{\#}$ is a Mori fiber space.
    \item If $l.k.d(M)\geq 0$, then $X^{\#}$ is nef(a log minimal model) and $l.k.d(M)=l.k.d(M^{\#})$.
    \item If $l.k.d(M)=2$, then $M$ is a $\C^{*}$-fibration over a normal surface.
\end{enumerate} 
\end{enumerate}

We have the following theorems as a corollary of propositions \ref{invariance3}, \ref{stablydisplconj}.

\begin{theorem} \label{affine3fold1}
Consider the special case of Kishimoto's $\#$- minimal model program on affine threefolds that restrict half-point attachment at a point on a hypersurface. Then
\begin{enumerate}
    \item $\ell(M^{i})$'s are not decreasing. 
    \item Let $M$ be an affine threefold and $M^{\#}$ be a terminal object of $M$ in Kishimoto's $\#$-minimal model program. Suppose that $\pi_{1}(M^{\#})=0$, $M\ncong M^{\#}$ and at least one of steps of half-point attachment is done on a hypersurface divisor. Then $\ell(M)\geq 2$.
\end{enumerate}
\end{theorem}

\begin{proof}
Let $M$ be an affine variety given by a complement $A\setminus Y$ of a hypersurface $Y$ inside an affine variety $A$. A half-point attachment of type $(b,k)$ of $M$ is $M':=\mathcal{B}l_{p}A\setminus \tilde{Y}=M\cup (E\cap M')$, where $p$ is a point on $Y$, $\tilde{Y}$ is a proper transformation of $Y$, and $E$ is an exceptional divisor. By theorem \ref{spectralseqconj}, there is a spectral sequence converging to $SH^{*}(M')$ whose $E_{1}$-page is $SH^{*}(M'\setminus (E\cap M'))=SH^{*}(M)$. Therefore, $\widehat{\ell}(M')\leq \widehat{\ell}(M)$.
\end{proof}

Let $M$ be an affine variety and $Y_{1}, Y_{2}\subset M$ two smooth hypersurfaces. By corollary \ref{spectralseqconj1}, for each $i=1,2$, there exists a spectral sequence converging to $SH^{*}(\mathbb{Y}^{c}_{i}\subset M)\cong SH^{*}\big(\underline{M\setminus Y_{i}}\subset M\big)$ whose $E_{1}$-page is $SH^{*}\big(\mathbb{Y}^{c}_{i}\subset \underline{M\setminus Y_{i}}\big)\cong SH^{*}\big(\underline{M\setminus Y_{i}}\subset M\setminus Y_{i}\big)$. Denote the spectral sequence as $SH^{*}\big(\underline{M\setminus Y_{i}}\subset M\setminus Y_{i}\big)\Rightarrow SH^{*}(\underline{M}\subset M)$, for each $i$.

$$\begin{tikzcd} [cramped, sep=small]
   & M \arrow[ld,dashrightarrow] \arrow[rd,dashrightarrow] & \\
  M\setminus Y_{1} \arrow[rd,dashrightarrow] &  & M\setminus Y_{2} \arrow[ld,dashrightarrow]\\
   & M\setminus (Y_{1}\cup Y_{2}) &  \end{tikzcd}$$
   
$$\begin{tikzcd} [cramped, sep=small]
&SH^{*}(\underline{M}\subset M) &\\
SH^{*}\Big(\underline{M\setminus Y_{1}}\subset M\setminus Y_{1}\Big)\arrow[ru,Rightarrow]& &SH^{*}\Big(\underline{M\setminus Y_{2}}\subset M\setminus Y_{2}\Big)\arrow[lu,Rightarrow]\\
&SH^{*}\Big(\underline{M\setminus (Y_{1}\cup Y_{2})}\subset M\setminus (Y_{1}\cup Y_{2})\Big)\arrow[ru,Rightarrow] \arrow[lu,Rightarrow]&\end{tikzcd}$$

From the diagram,
$$\widehat{\ell}(M) \leq \text{min}\{\widehat{\ell}(M\setminus Y_{1}), \widehat{\ell}(M\setminus Y_{2})\}
\leq \text{max}\{\widehat{\ell}(M\setminus Y_{1}), \widehat{\ell}(M\setminus Y_{2})\}
\leq \widehat{\ell}(M\setminus (Y_{1}\cup Y_{2})).$$

\begin{remark} We expect that similar idea works for the L-minimal model for affine threefolds, [Theorem 1.1, \cite{kishimoto}], where $M$ is a smooth affine threefold and $X$ is a smooth projective threefold compactifying $M$ with a nef divisor $D$ and where the complete linear system $|D|$ is assumed to contain a Du Val member $S$.
\end{remark}

\subsubsection{Log Kodaira Dimension}
Using similar idea on exact sequences of symplectic cohomologies of Liouville cobordisms in \cite{cieliebakoancea18}, we get a restriction on affine varieties with non-negative log Kodaira dimension $l.k.d(M)\geq 0$. Let us remind the definitions.

\begin{defn}
The log Kodaira dimension of an affine variety $M$ is the Iitaka dimension of the line bundle $\mathcal{K}_{X}+\mathcal{D}$ on $X$, where $X$ is a projective variety obtained by compactifying $M$ with simple normal crossing divisors $D$ and with the canonical divisor $K_{X}$, $$l.k.d(M):=\kappa(X,K_{X}+D):=tr.deg_{\C}(\oplus_{d}H^{0}(X,\mathcal{K}_{X}^{\otimes d}\otimes \mathcal{D}^{\otimes d}))-1.$$
\end{defn}

Let $\mathcal{L}$ be a line bundle on a projective variety $X$. If $\mathcal{L}^{\otimes d}$ has a global section for some $d$, then $\mathcal{L}^{\otimes d}$ defines a rational map from $X$ to $Proj(H^{0}(X,\mathcal{L}^{\otimes d}))$. The Iitaka dimension $\kappa(X,L)$ of the line bundle $\mathcal{L}$ on $X$ is the maximum dimension of the image of this map for all $d$ where the map is defined and is $-\infty$ if it is not defined for any $d$, i.e., $\mathcal{L}^{\otimes d}$ has no global sections for any $d$. For example, for $X=\P^{n}$, $\mathcal{K}_{X}=\mathcal{O}_{X}(-n-1)$ so $\kappa(X,\mathcal{K}_{X})=-\infty$. For a Riemann surface $X$ with genus $g$, by the Riemann-Roch theorem, $\mathcal{K}_{X}=\mathcal{O}_{X}(2g-2)$, 
\begin{enumerate}
    \item $g=0$, $\mathcal{K}_{X}=\mathcal{O}_{X}(-2)  \iff \kappa(X,\mathcal{K}_{X})=-\infty$,
    \item $g=1$, $\mathcal{K}_{X}=\mathcal{O}_{X} \quad \quad \iff \kappa(X,\mathcal{K}_{X})=0$,
    \item $g\geq 2$, $\mathcal{K}_{X}=\mathcal{O}_{X}(\geq 2) \iff \kappa(X,\mathcal{K}_{X})=1$,
\end{enumerate}

A Fano projective variety $X$ (i.e., $\kappa(X,K_{X})=-\infty$) is known to be rationally-connected by Mori's bend-and-break technique, therefore it is uniruled \cite{kollar.miyaoka.mori92}. The lemma 7.1 in \cite{mclean14} explains the relation between uniruledness and the log Kodaira dimension as following: For an affine variety $M$, if $M$ is $\C$-uniruled, then $l.k.d(M)=-\infty$. If $M$ is $\C^{*}$-uniruled, then $l.k.d(M)\leq$ dim$_{\C}M-1$.
\begin{corollary} As a direct corollary of results of \cite{biolley04} and \cite{mclean14}, for an affine variety $M$,
\begin{enumerate}
    \item If $\ell(M)=0$, then $l.k.d(M)=-\infty$.
    \item If $\ell_{p.p}(M)=0$, then $l.k.d(M)\leq \text{dim}_{\C}M-1$.
\end{enumerate}
\end{corollary}

\section{Further Questions and Concluding Remarks}
We summarize further questions as follows.
\begin{question}
In theorem \ref{main theorem0}, when is the union $\bigcup_{\omega\in \text{Amp }\subset H^{2}(M)} \bigcup_{[\mho]\in \text{im}\delta}\Xi_{\mho}$ of all the uniruled subvarieties of $M$ from the image of the map $\delta$ maximal as stratified uniruled-subvarieties of $M$ ($\bigcup_{D \text{, ample }} \bigcup_{[\mho]\in \text{im}\widehat{\delta}_{D}}\widetilde{\Xi}_{\mho}$ for a projective variety $X$, respectively)?
\end{question}

For a log-Calabi-Yau pairs $(X,D)$ (log-general for an ample $D$), it is known that quantum cohomology $QH^{*}(X)$ is a deformation of symplectic cohomology $SH^{*}(X\setminus D)$, \cite{borman.sheridan.varolgunes21}, \cite{mclean18}.
\begin{question}
Is $\bigcup_{D \text{, ample }} \bigcup_{[\mho]\in \text{im}\widehat{\delta}_{D}}\widetilde{\Xi}_{\mho}$ related to the genus $0$ Gromov-Witten invariant of $X$?
\end{question}

From the long exact sequences for Liouville cobordisms \cite{cieliebakoancea18},\\
\begin{tikzcd} 
\cdots\to SH^{*-1}(C, \partial^{-}C)\to    SH^{*-1}_{+}(C, \partial^{-}C)\lra{\delta} H^{*}(C, \partial^{-}C) \to SH^{*}(C, \partial^{-}C)\to\cdots,
\end{tikzcd}
\begin{defn} \label{defn2}
    $\ell(C, \partial^{-}C):=$min$\{$deg$([\alpha]): [\alpha] \in H^*(C,\partial^{-}C)$ with $[\alpha]=\delta(\gamma)$ for some $0\neq \gamma\in SH^{*-1}_{+}(C,\partial^{-}C)\}$.
\end{defn}
We reviewed that symplectic cohomology is invariant under the Kaliman modification \cite{mclean08} and that symplectic cohomology of blow-up of ADE-singularities is vanishing \cite{mclean-ritter18}. Considering that h-principle for symplectic cobordism of dimension $2n>4$ with concave overtwisted contact boundary exist \cite{eliashberg.murphy21}, we have the following questions as well. 

\begin{question} Let $C$ be an exceptional locus under the Kishimoto's half-point attachment. Is there a suitable (possibly, non-exact) Liouville cobordism structure on $C$, after deforming Liouville structures? If so, does $\ell(C, \partial^{-}C)$ vanish?
\end{question}

\begin{question} \label{connectedsum2}
If $M$ is an affine variety of dim$_{\C}M>1$ with non-negative log Kodaira dimension. Then is $M$ a "nontrivial" connected sum of affine varieties of non-negative log Kodaira dimension, by using $\ell(C, \partial^{-}C)$?
\end{question}

\begin{question} \label{construction2}
For each $n\geq 5$, is there an affine variety $M^{2n}_{2k}$ of dim$_{\C}M=n$ with $\ell(M^{2n}_{2k})=2k$ for all $4\leq 2k<n$? For example, $M^{8}_{4}$, $M^{10}_{4}$, $M^{12}_{6}$?
\end{question}

A smooth complex affine surface $M$ is $\C$-uniruled if and only if $l.k.d(M)=-\infty$ and is known to be $\C$-fibration \cite{miyanishi.sugie80}. For a smooth projective variety $X$, separably rationally connectedness implies that Kodaira dimension of $X$ is $-\infty$, i.e., dim$_{\C}H^{0}(X,\mathcal{K}_{X}^{\otimes m})=0$ for all $m>0$. The converse is conjectured to be true and the cases of dim$_{\C}X\leq 3$ are proven [3.8, 3.8.1, \cite{kollar96}, \cite{kollar.miyaoka.mori92}]. We ask similar questions for the  $\C$-uniruledness.

\begin{question} \label{2uniruled}
If $SH^*(M)$ has an invertible element other than the unit and the idempotents, then is $M$ $2$-uniruled(i.e. $\C^{*}:=(\C\setminus\{0\})$-uniruled or $\C$-uniruled)? Is there any relation to Seidel's representation in quantum cohomology and invertible elements in \cite{seidelthesis}?
\end{question}

\begin{remark}
Symplectic cohomology has a fruitful structure, the $\mathcal{L}_{\infty}$ structure, from higher homotopy of Chas-Sullivan string topology on the loop space of a symplectic manifold [\cite{chas.sullivan99}, \cite{fabert19}, \cite{fabert11}, \cite{siegel21}]. Can we detect rationally connectedness of $(X,D)$ by understanding obstruction of higher operations on symplectic cohomology which is from TQFT(Topological Quantum Field Theory) or from SFT(Symplectic Field Theory), which is symplectic invariant? If so, how could we compute them effectively?
\end{remark}

$S^{1}$-equivariant symplectic cohomology has the Gysin sequence which make a commutative diagram with the action-filtered long exact sequences by Bourgeois, Oansea in \cite{bourgeois.oancea13gysin}.

$$\begin{tikzcd} [cramped, sep=small]
\cdots \to SH^{*}(M)\arrow[r]\arrow[d]           & SH^{*}_{+}(M)\arrow[r,"\delta"]\arrow[d]                 & H^{*+1}(M)\arrow[r,"\iota"]\arrow[d]         & SH^{*+1}(M)\arrow[d]\to \cdots\\
\cdots \to SH^{*}_{S^{1}}(M)\arrow[r]\arrow[d]   & SH^{*}_{S^{1},+}(M)\arrow[r,"\delta_{S^{1}}"]\arrow[d]   & H^{*+1}_{S^{1}}(M)\arrow[r,"\iota"]\arrow[d] & SH^{*+1}_{S^{1}}(M)\arrow[d]\to \cdots\\
\cdots \to SH^{*+2}_{S^{1}}(M)\arrow[r]\arrow[d] & SH^{*+2}_{S^{1},+}(M)\arrow[r,"\delta_{S^{1}}"]\arrow[d] & H^{*+3}_{S^{1}}(M)\arrow[r,"\iota"]\arrow[d] & SH^{*+3}_{S^{1}}(M)\arrow[d]\to \cdots\\
\cdots \to SH^{*+1}(M)\arrow[r]                  & SH^{*+1}_{+}(M)\arrow[r,"\delta"]                        & H^{*+2}(M)\arrow[r,"\iota"]                  & SH^{*+2}(M) \to \cdots.
\end{tikzcd}$$
We can define the following definition and want to generalize the result of this paper to equivariant symplectic cohomology.
\begin{defn}
$\ell_{S^{1}}(M):=$min$\{$deg$([\alpha]):[\alpha] \in H^{*}_{S^{1}}(M)$ with $[\alpha]=\delta_{S^{1}}(\gamma)$ for some $0\neq \gamma\in SH^{*-1}_{S^{1},+}(M)\}$.
\end{defn}
$S^{1}$-equivariant symplectic cohomology $SH^{*}_{S^{1}, k}(\C^{*}\rtimes_{\phi} M)$ of the symplectic mapping cylinder of $M$ is isomorphic to the fixed point symplectic cohomology of $\phi:M\to M$ [Lemma B. 16, \cite{mclean19}]. It is proven that the invariants of isolated singularities, the multiplicity and the log canonical threshold, can be computed from $SH^{*}_{S^{1}}(\C^{*}\rtimes_{\phi} M)$ [Corollary 1.4, \cite{mclean19}]. 
\begin{question} 
\begin{enumerate}
    \item Find geometric meaning of differential maps or a map induced from Seidel's connection of $S^{1}$-equivariant cohomology. 
    \item As the long exact sequence of $SH^{*}(M)$ detects rational curves on $M$, can we detect subvarieties of $M$ from $SH^{*}_{S^{1}}(M)$ or $SH^{*}_{S^{1}}(\C^{*}\rtimes_{\phi} M)$? If so, what is birational geometric meaning of the subvarieties of $M$ or a mirror of $M$?
\end{enumerate}
\end{question}

\subsubsection{Symplectic View on Birational Geometry: Cohomological or Categorical} Mathematical objects are explored up to equivalences, or up to symmetries. Radial/conformal symmetry is one of them. Projectivisation of algebraic varieties gives many benefits to understand the varieties, for example, compactness without boundary and has a strong fact that every irreducible $n$-dimensional smooth projective complex variety with ample tangent bundle is isomorphic to $\C P^{n}$ \cite{mori79}. However, when we classify projective varieties or log-pairs, we need to relax the condition on "up to isomorphisms" because two non-isomorphic algebraic varieties are meaningfully related by birational equivalence, being isomorphic away from some divisors or subvarieties. Birational equivalence would come from the choice of poles or various points of view from divisors when we projectivise varieties or different steps of "projecitivation/de-projectivisation". The minimal model program is to find core objects/properties of algebraic varieties up to birational equivalence. To use rational curves on a variety is one of Mori's groundbreaking works that initiated the program. Roughly speaking, for any smooth complex projective variety $X$ (we refer \cite{kollar96} and references therein.),
\begin{itemize}
    \item If the Kodaira dimension $\kappa(X,\mathcal{K}_{X})=-\infty$, $X$ is birational to a Fano fibration.
    \item If $\kappa(X,\mathcal{K}_{X})\geq 0$, then we want to get rid of rational curves $C$ with $K_{X}.C<0$ so that the result variety $X'$ has the nef canonical divisor $K_{X}$.
\end{itemize}
There have been approaches understanding birational geometry using quantum cohomology \cite{li.ruan09} and are many evidences that show symplectic cohomology can be also useful to understand birational geometry, \cite{mclean09}, \cite{mclean12}, \cite{mclean14}, \cite{mcleancomputation}, \cite{mclean19}, \cite{mclean18}. For example, Li-Ruan showed that any two smooth minimal models of a projective Calabi-Yau 3-fold have the same quantum cohomology [Corollary A.3, \cite{li.ruan01}]. Two birational projective Calabi-Yau varieties have isomorphic Zariski-dense open affine subvarieties. So, considering that small quantum cohomologies are deformation of Hamiltonian Floer cohomology of affine subvarietes, McLean showed that birational projective Calabi-Yau varieties of $\text{dim}_{\C}=n>3$ have isomorphic small quantum cohomology algebras after tensoring with Novikov rings corresponding to certain divisors (relating to K\"ahler structures) \cite{mclean18}.

Meanwhile, various aspects of mirror symmetry have been found: relations between symplectic geometry and algebraic geometry, or geometry and algebra, or covariant functors and contravariant functors, or duality between cores and cocores, and the Gross-Siebert program. Roughly speaking, homological mirror symmetry predicts the equivalence of two categories, the derived Fukaya category of intersections of Lagrangian submanifolds of a symplectic manifold $X$ and the derived category of coherent sheaves on mirror algebraic variety $X^{\vee}$. Historically, Hodge-theoretic and then enumerative geometric mirror symmetry was found first but it is conjectured that homological mirror symmetry recovers those classical mirror symmetry. Homological mirror symmetry for compact manifolds without boundary extends to compact manifolds with boundary or a log-pair. The open-closed map relates Hochschild cohomology of (Wrapped) Fukaya category and quantum cohomology(resp. symplectic cohomology). Conjecturally, a mirror $M^{\vee}$ of $M$ satisfies that $SH^{*}(M;\C)\cong \oplus_{p+q=*}H^{q}(M^{\vee},\wedge^{p}\mathcal{T}_{M^{\vee}})$. Very roughly speaking, conjectural constructions of a mirror of affine variety $M$ is $Spec(SH^{*}(M))$ with Landau-Ginzburg superpotential on it. For example, a mirror of $\C^{*}$ is known to be $\C^{*}$ itself: $SH^{*}(\C^{*};\Lambda)=SH^{*}(T^{*} S^{1};\Lambda)=\Lambda [x,x^{-1}]$ and Spec$(SH^{*}(\C^{*}))=\C^{*}$. Under this conjectural mirror picture, mirror of a variety having vanishing symplectic cohomology is a point. Therefore, uniruled locus in a variety is "simple". A remark is that in the minimal model program, the canonical ring, defined by $\oplus_{m}H^{0}(X,\mathcal{K}_{X}^{\otimes m})$ is an important birational invariant: If the canonical ring is finitely generated and $\mathcal{K}_{X}$ is big, then the canonical model of $X$ in the minimal model program is $\text{Proj}_{\C}(\oplus_{m}H^{0}(X,\mathcal{K}_{X}^{\otimes m}))$. Birkar-Cascini-Hacon-M$^{\text{c}}$Kernan proved that the canonical ring is finitely generated and the existence of minimal models \cite{BCHM10}. On the other hand of symplectic geometry, for any affine variety $M$ from a log Calabi-You pair $(X,D)$ and any field $\K$ [Theorem 1.1, \cite{pomerleano21}],
\begin{enumerate}
    \item $SH^{0}(M,\K)$ is a finitely generated $\K$-algebra.
    \item $SH^{*}(M,\K)$ is a finitely generated module over $SH^{0}(M,\K)$.
    \item For any Lagrangian submanifolds $L, L'$ of $M$, the wrapped Lagrangian Floer cohomology $WF^{*}(L,L',\K)$ is a finitely generated module over $SH^{0}(M,\K)$.
\end{enumerate}
In general, it is hard to compute symplectic cohomology. One of solutions could be to apply mirror symmetry to compute symplectic cohomology in more systematic way and it would be interesting to characterize the minimal model program in mirror side. On the other hand, it would be also interesting to understand how Lagrangian skeletons with arboreal singularities, which are Weinstein homotoped to a Liouville domain or a log pair, behave under birational morphisms (see also \cite{nadler15}, \cite{starkston18}, \cite{mustata.nicaise15}, \cite{nicaise.xu16}, \cite{mikhalkin19}) and how Nadler's categories (and equivalence of categories, see \cite{nadler.arboreal17}, \cite{nadler17}, \cite{nadler19}, \cite{nadler16}) behave under the Minimal Model Program. 
\section{Appendix}
\subsection{Several Conventions} \label{several conventions}
Let $(M,\omega)$ be a symplectic manifold of dim$_{\R}M=2n$, $u:\R\times \R/\Z\to M$ with $\lim_{s\to \pm \infty}u(s,t)=x_
{\pm}(t)$ be a solution of the Floer equation: $J\cdot \frac{\partial u}{\partial s}=\frac{\partial u}{\partial t}-X_{H}$ with asymptotic Hamiltonian periodic orbits $x_{\pm}:\R/\Z\to M$ with $\dot{x}_{\pm}=X_{H}(x_{\pm})$ and $v$ is a symplectic filling of $x$. Let us compare four conventions which are "the same" up to direction of differentials, signs and degrees. 
\begin{center}
\begin{tabular}{|c|c|} 
 \hline
 Symplectic Homology \cite{cieliebakoancea18} & Symplectic Cohomology \cite{abouzaid15}, \cite{ritter10}\\
 \hline
 $\omega(X_{H},\cdot)=-dH(\cdot)$, $\omega(\cdot, J\cdot)=g(\cdot,\cdot)$ & $\omega(X_{H},\cdot)=-dH(\cdot)$, $\omega(\cdot, J\cdot)=g(\cdot,\cdot)$\\
 $J\cdot X_{H}=-grad_{g}H$ & $J\cdot X_{H}=grad_{g}H$\\
 $\mathcal{A}=\int v^{*}\omega -\int x^{*}H$ & $\mathcal{A}=-\int v^{*}\omega +\int x^{*}H$\\
 $\mathcal{A}_{H}(x_{+})-\mathcal{A}_{H}(x_{-})=\int_{\R\times \R/\Z}|\partial_{s}u|^{2}dsdt$ &  $\mathcal{A}_{H}(x_{+})-\mathcal{A}_{H}(x_{-})=-\int_{\R\times \R/\Z}|\partial_{s}u|^{2}dsdt$\\
 Conley-Zehnder Index$=n-$Morse Index & Conley-Zehnder Index$=n-$Morse Index\\
 Differentials decreases the action, the grading & Differentials increases the action, the grading\\
 $CF^{[a,b)}_{*}\to CF^{[a,c)}_{*}\to CF^{[b,c)}_{*}$ for $a<b<c$ & $CF_{[a,b)}^{*}\to CF_{[a,c)}^{*}\to CF_{[b,c)}^{*}$ for $a<b<c$\\
 \hline
 Symplectic Homology \cite{bourgeois.oancea09}, \cite{mclean12}  & Symplectic Cohomology \cite{varolgunes19}\\
 \hline
 $\omega(X_{H},\cdot)=dH(\cdot)$, $\omega(\cdot, J\cdot)=g(\cdot,\cdot)$ & $\omega(X_{H},\cdot)=dH(\cdot)$, $\omega(\cdot, J\cdot)=g(\cdot,\cdot)$\\
 $J\cdot X_{H}=-grad_{g}H$& $J\cdot X_{H}=grad_{g}H$\\
 $\mathcal{A}=-\int v^{*}\omega -\int x^{*}H$& $\mathcal{A}=\int v^{*}\omega +\int x^{*}H$\\
 $\mathcal{A}_{H}(x_{+})-\mathcal{A}_{H}(x_{-})= -\int_{\R\times \R/\Z}|\partial_{s}u|^{2}dsdt$ & $\mathcal{A}_{H}(x_{+})-\mathcal{A}_{H}(x_{-})= \int_{\R\times \R/\Z}|\partial_{s}u|^{2}dsdt$\\
 $-$Conley-Zehnder Index=$-n$+Morse Index & Conley-Zehnder Index\\
 Differentials decreases the grading & Differentials increases the grading\\
 \hline
\end{tabular}
\end{center}

\end{document}